\documentclass[11pt,a4paper]{amsart}
\usepackage{spinal}

\title[On symplectic fillings of spinal open book decompositions~I]{On 
symplectic fillings of spinal open book decompositions~I:\\
Geometric constructions}
\author{Samuel Lisi}
\address{University of Mississippi\\
Department of Mathematics\\
USA }
\email{stlisi@olemiss.edu}
\author{Jeremy Van Horn-Morris}
\address{Department of Mathematical Sciences \\
The University of Arkansas\\
USA}
\email{jvhm@uark.edu}
\author{Chris Wendl}
\address{Institut f\"ur Mathematik \\
Humboldt-Universit\"at zu Berlin \\
Germany}
\email{wendl@math.hu-berlin.de}
\thanks{S.L.~was partially supported during this project by the ERC Starting Grant of Fr\'ed\'eric Bourgeois
StG-239781-ContactMath, Vincent Colin's ERC Grant geodycon, and by a University
of Mississippi CLA SRG. J.V.H.--M.~was partially supported by 
Simons Foundation grant No.~279342 and NSF grant DMS-1612412. 
C.W.~was partially supported by a 
Humboldt Foundation Postdoctoral Fellowship, a Royal Society University
Research Fellowship, and EPSRC grant EP/K011588/1.
}
\subjclass[2010]{Primary 32Q65; Secondary 57R17}

\begin{document}

\begin{abstract}
A spinal open book decomposition on a contact manifold is a
generalization of a supporting open book which exists naturally 
e.g.~on the boundary of a symplectic filling with a Lefschetz fibration 
over any compact oriented surface with boundary.  In this first paper of
a two-part series, we introduce the basic notions relating spinal open
books to contact structures and symplectic or Stein structures on Lefschetz fibrations, leading to
the definition of a new symplectic cobordism construction called 
\emph{spine removal} surgery, which generalizes previous constructions due to
Eliashberg \cite{Eliashberg:cap}, Gay-Stipsicz \cite{GayStipsicz:cap}
and the third author \cite{Wendl:cobordisms}.  As an application, spine removal
yields a large class of new examples of contact manifolds that are not
strongly (and sometimes not weakly) symplectically fillable.  This paper
also lays the geometric groundwork for a theorem to be proved in part~II,
where holomorphic curves are used to classify the symplectic and Stein fillings of contact
$3$-manifolds admitting a spinal open book with a planar page.
\end{abstract}

\maketitle

\tableofcontents

\setcounter{section}{-1}
\section{Introduction}
\label{sec:overview}

The present paper is the first in a two-part series aimed at generalizing
the well-known interplay between contact structures with supporting open book 
decompositions and their fillings by symplectic or Stein manifolds with
symplectic Lefschetz fibrations.  We can point to at least two specific previous
applications of open books in contact topology as inspiration
for this project:
\begin{enumerate}
\item In \cites{Wendl:fillable,NiederkruegerWendl}, the third author proved
that for every contact $3$-manifold supported by a planar open book, the deformation
classes of its symplectic fillings are in bijective correspondence to the
diffeomorphism classes of Lefschetz fibrations over $\DD^2$ that fill the
open book.  Some version of this statement is true moreover for all of the
usual notions of symplectic fillability (i.e.~weak, strong, Liouville and Stein),
thus proving that for planar contact manifolds, they are all equivalent.
The problem of classifying fillings for such contact manifolds was reduced
in this way to a factorization problem on the mapping class group of surfaces,
cf.~\cites{PlamenevskayaVanHorn,Plamenevskaya:surgeries,Wand:planar,
KalotiLi:Stein,Kaloti:Stein}.
\item In \cite{Eliashberg:cap}, Eliashberg used non-exact symplectic 
$2$-handles attached along the binding of an open book to construct symplectic
caps for all closed contact $3$-manifolds.  This served among other things as
an ingredient in Kronheimer-Mrowka's proof of Property~P \cite{KronheimerMrowka:propertyP}, 
and it was later
generalized to various forms of non-exact symplectic cobordism between 
contact manifolds, cf.~\cites{Gay:GirouxTorsion,GayStipsicz:cap,Wendl:cobordisms}.
\end{enumerate}
The motivating question behind the present project was as follows: 
what structure naturally arises on the convex boundary of a Lefschetz
fibration with exact symplectic fibers over a surface with boundary
other than~$\DD^2$?  Spinal open books are the answer to this question, and
we will show that they give rise to far-reaching generalizations of both of
the results mentioned above.  One example of the first type appeared already
in \cite{Wendl:fillable}, where the symplectic fillings of $\TT^3$ were
classified in terms of Lefschetz fibrations over the annulus 
$[-1,1] \times S^1$.  This was proved using methods from the low-dimensional
theory of $J$-holomorphic curves, and the aim of the sequel to this 
paper \cite{LisiVanhornWendl2} will be to push those techniques as
far as they can reasonably be pushed.

Here is an initial sketch of the main idea.
Roughly speaking, a spinal open book decomposes a $3$-manifold
$M$ into two (possibly disconnected) pieces, called the \emph{paper} $M\paper$
and the \emph{spine} $M\spine$, where $M\paper$ consists of families of
\emph{pages} fibering over~$S^1$, 
$M\spine$ is an $S^1$-fibration over some collection of compact 
oriented surfaces, and the boundaries of fibers in $M\paper$ consist of 
fibers in~$M\spine$ (see Figure~\ref{fig:spinalPicture}).  The usual notion of 
open books is recovered if one takes the base of the fibration on $M\spine$ to 
be a disjoint union of disks (see Example~\ref{ex:openbook}); similarly, 
allowing annuli in the base produces
the notion of \emph{blown up summed open books} (Example~\ref{ex:bindingSum}), 
which were studied in \cite{Wendl:openbook2}.  
One of the main results of \cite{LisiVanhornWendl2} 
can be summarized as follows (see \S\ref{sec:results} below for the pertinent
definitions):

\begin{figure}
\includegraphics{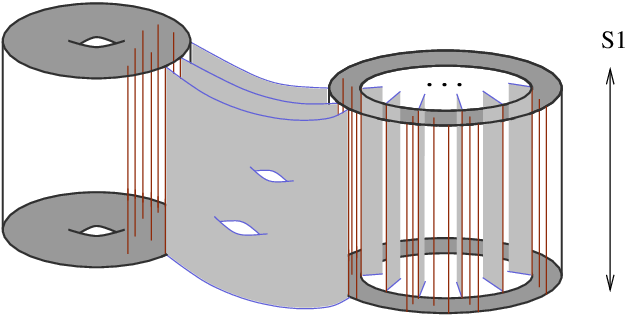}
\caption{\label{fig:spinalPicture} 
A spinal open book with two spine components, which are $S^1$-fibrations
over a genus~$1$ surface with one boundary component and an annulus respectively.
They are connected to each other by an $S^1$-family of pages with genus~$2$,
and we can also see a fragment of a second $S^1$-family of pages attached
to the annular spine component.}
\end{figure}

\begin{thma}[\cite{LisiVanhornWendl2}]
\label{thma:main}
Suppose $(M,\xi)$ is a closed contact $3$-manifold containing a domain
$M_0$ on which $\xi$ is supported by an amenable spinal
open book $\boldsymbol{\pi}$ that has a planar page in its interior. 
If $(M, \xi)$ admits a weak filling that is exact on the spine of~$\boldsymbol{\pi}$,
then $M = M_0$, and the set of weak symplectic fillings of $(M,\xi)$ that are 
exact on the spine is, up to symplectic deformation equivalence, 
in one-to-one correspondence
with the set of Lefschetz fibrations (up to diffeomorphism) that match
$\boldsymbol{\pi}$ at their boundaries.  Moreover, every such filling can
be deformed to a blowup of a Stein filling.
\end{thma}

To focus for a moment on Stein fillings in particular: most previous results
classifying Stein fillings have classified them up to diffeomorphism or
symplectic deformation, the only exceptions we are aware of being results of
Eliashberg \cites{Eliashberg:diskFilling,CieliebakEliashberg} 
and Hind \cites{Hind:RP3,Hind:Lens},
which achieved uniqueness up to Stein deformation equivalence for
fillings of $S^3$, connected sums of $S^1 \times S^2$, and certain lens spaces.
In these examples, the classification up to Stein deformation matches the
classification up to symplectic deformation, and we will see that
this is not a coincidence---it can be seen as a symptom of a general
\emph{quasiflexibility} phenomenon for Stein surfaces:

\begin{thma}[\cite{LisiVanhornWendl2}]
\label{thma:quasi}
Suppose $W$ is a compact $4$-manifold with boundary, 
admitting two Stein structures
$J_0$ and $J_1$ such that $(W,J_0)$ is compatible with a Lefschetz fibration
(over an arbitrary compact oriented surface) with fibers of genus zero.
Then $J_0$ and $J_1$ are Stein homotopic if and only if their induced 
symplectic structures are homotopic as symplectic structures
convex at the boundary.
\end{thma}

Note that the symplectic deformation in this statement need not be in a fixed
cohomology class---in particular, quasiflexibility is a very different phenomenon
from the familiar relationship between Stein and Weinstein structures
(cf.~\cite{CieliebakEliashberg}).

While the results quoted above require holomorphic curve techniques, this
first paper in the series will focus on the less analytical but more geometric
aspects of the theory of spinal open books.  We will start by giving natural constructions
of contact structures supported by spinal open books and symplectic or
Stein structures related to them.  The most subtle of these results pertains specifically
to Stein (or equivalently Weinstein) structures, and gives a verifiable criterion in terms of Lefschetz fibrations
for two Stein structures to be Stein homotopic.
This will serve in \cite{LisiVanhornWendl2} as an 
essential ingredient for the classification of Stein fillings up to Stein
deformation and the proof of Theorem~\ref{thma:quasi}.  The result is most
easily stated in terms \emph{almost Stein structures}, which are 
pairs $(J,f)$ consisting of an almost complex
structure $J$ and a $J$-convex function~$f$.  Here $J$ is not required to be
integrable, and $f$ need not be constant at the boundary, thus they do not
immediately define a Stein structure, but if we assume the Liouville vector
field dual to $-df \circ J$ is outwardly transverse at the boundary, then
$(J,f)$ nonetheless determines a Weinstein structure 
canonically up to Weinstein homotopy (see \S\ref{sec:BLF}).

The following theorem can be interpreted as saying that the Stein homotopy class
of a Stein structure can be deduced from a Lefschetz fibration if it satisfies
fairly strict compatibility conditions near the boundary but a minimum of
reasonable conditions in the interior---this result is well suited in particular
to the scenario in which fibers of a Lefschetz fibration are $J$-holomorphic curves.

\begin{thma}[see Theorem~\ref{thm:SteinHomotopy}]
\label{thma:SteinHomotopy}
Suppose $\Pi : E \to \Sigma$ is a Lefschetz fibration 
whose regular fibers and
base are each compact oriented surfaces with nonempty boundary, and write
$$
\p_v E := \Pi^{-1}(\p\Sigma), \qquad \p_h E := \bigcup_{z \in \Sigma} \p E_z.
$$
For $\tau =0,1$, assume $J_\tau$ is an almost complex structure on $E$ and
$f_\tau : E \to \RR$ is a smooth $J_\tau$-convex function such that the
following conditions are satisfied:
\begin{enumerate}
\item $J_\tau$ preserves the vertical subbundle of $TE$ and is compatible with
its orientation;
\item $f_\tau$ is constant on the boundary components of every fiber;
\item The Liouville form $\lambda_\tau := -df_\tau \circ J_\tau$ restricts to
both $\p_v E$ and $\p_h E$ as contact forms, the induced Reeb vector field
on $\p_h E$ is tangent to the fibers, and its flow preserves the maximal
$J_\tau$-complex subbundle of $T(\p_h E)$;
\item There exists a complex structure $j_\tau$ on $\Sigma$ and
an open neighborhood $\uU \subset \Sigma$ of $\p\Sigma$
such that the Cauchy-Riemann equation $T\Pi \circ J_\tau = j_\tau \circ T\Pi$
is satisfied on $E|_{\uU}$ and~$\p_h E$.
\end{enumerate}
Then the Weinstein structures on $E$ (after smoothing the corners) determined
by $(J_0,f_0)$ and $(J_1,f_1)$ are Weinstein homotopic.
\end{thma}

With this groundwork in place, we will then introduce
a new construction of non-exact symplectic cobordisms that generalizes
previous results from \cites{Eliashberg:cap,GayStipsicz:cap,Wendl:cobordisms} and 
arises from a natural topological operation on
spinal open books called \emph{spine removal surgery}.  An informal version
of the result can be stated as follows:

\begin{thma}[see Theorem~\ref{thm:spineRemoval}]
\label{thma:spineRemoval}
Assume $(M,\xi)$ is a contact $3$-manifold supported by a spinal open book
$\boldsymbol{\pi}$, $\Sigma\remove \times S^1 \cong M\remove \subset M\spine$ is an
open and closed subset of the spine of~$\boldsymbol{\pi}$, and $\widetilde{\boldsymbol{\pi}}$
is a spinal open book on a contact $3$-manifold $(\widetilde{M},\widetilde{\xi})$ 
defined by deleting $M\remove$ from $M$ and capping off all adjacent boundary components of pages of
$\boldsymbol{\pi}$ by disks.  Then there exists a symplectic cobordism
with strongly concave boundary $(M,\xi)$ and weakly convex
boundary $(\widetilde{M},\widetilde{\xi})$, defined by attaching the ``handle''
$\Sigma\remove \times \DD^2$ with a product symplectic structure along 
$\Sigma\remove \times S^1 \cong M\remove$.
\end{thma}

Special cases of this operation were used in \cite{Wendl:cobordisms} to
construct non-exact symplectic cobordisms between pairs of contact $3$-manifolds
that do not admit exact ones, e.g.~it showed that all of the known examples
of contact $3$-manifolds with finite orders of algebraic torsion
(cf.~\cite{LatschevWendl}) are symplectically cobordant to overtwisted ones.
We will use the general version in this paper to prove the vast
majority of cases of Theorem~\ref{thma:main} for which the contact manifold
turns out to be non-fillable, a result that can be interpreted as generalizing the
local filling obstruction defined as \emph{planar torsion} in \cite{Wendl:openbook2}.
A slightly different kind of 
application appears in \cite{LisiWendl:geography}, where spine removal is used 
to prove that contact $3$-manifolds supported by planar spinal open books
satisfy a universal bound on the geography of their symplectic fillings.
This generalizes a previous result
for the case of planar open books due to Plamenevskaya 
\cite{Plamenevskaya:surgeries} (see also \cite{Kaloti:Stein}).

In our project we have focused specifically on dimension three, since 
that is where the strongest results on classification of fillings can be
proved, but it should be mentioned that the theory of spinal open books has
already had some impact on developments in higher-dimensional contact 
topology.  In dimension $2n-1$, it is natural to consider decompositions
$M = M\spine \cup M\paper$ where $M\paper$ is a fibration of Liouville
domains over a contact manifold and $M\spine$ is a strict contact
fibration over a Liouville domain.  Taking $\DD^2$ and $S^1$ as bases produces 
the usual notion of open books in arbitrary dimensions, but it is sometimes
also useful to allow higher-dimensional bases,
e.g.~the first author has observed that Bourgeois's 
construction \cite{Bourgeois:tori} of contact structures on $M \times \TT^2$
can be understood as an operation replacing $\DD^2$ and $S^1$ with
$T^*\TT^2$ and $\TT^3$ as base spaces in a spinal open book
(cf.~\cite{LisiMarinkovicNiederkrueger}).  Working with
strictly low-dimensional fibers but higher-dimensional bases, 
\cite{MassotNiederkruegerWendl} constructed a higher-dimensional version of
a spine removal cobordism in order to establish the first examples of
higher-dimensional tight contact manifolds that are not symplectically fillable.
More recently, Moreno \cites{Moreno:thesis,Moreno:algebraicGiroux} and
Zhou \cite{Zhou:APT} have used high-dimensional
spinal open books to construct new examples of contact manifolds with
higher-order algebraic torsion.  Spinal open books also appeared
in work of Acu and Moreno \cite{AcuMoreno:planarity}, which used a variant 
of spine removal surgery to study a higher-dimensional
analogue of planar contact manifolds, and in Zhou's construction \cite{Zhou:infinite}
of contact manifolds in arbitrary dimensions admitting infinitely many
distinct Weinstein fillings.

\subsubsection*{A remark on timing}

While this paper is intended as the ``official''
introduction to spinal open books in dimension three, 
the project has by now been in preparation 
long enough for some of the fundamental notions to have appeared already
in several other papers, some by the authors plus collaborators
(e.g.~\cite{BaykurVanhorn:large}) and some by others
(e.g.~\cites{MinRoyWang:exotic,PlamenevskayaStarkston:nearly}).
We have tried to make sure all definitions are consistent with what has
previously appeared, but in the event of any discrepancies, the present
paper is meant to be definitive.

\subsubsection*{Outline of the paper}

Section~\ref{sec:results} is an extended introduction, intended to give
precise versions of all the essential definitions and main results, including
some definitions that are needed mainly for the classification discussion
in~\cite{LisiVanhornWendl2}.
Section~\ref{sec:topology} then proves the essential theorems relating
spinal open books and Lefschetz fibrations to their associated deformation
classes of contact and symplectic structures, and \S\ref{sec:SteinHomotopy}
proves Theorem~\ref{thma:SteinHomotopy} on Stein homotopy classes.
In \S\ref{sec:model}, we construct a concrete symplectic model for collar
neighborhoods (in the symplectization) of a contact manifold
supported by an arbitrary
spinal open book, which is then
used to prove the main theorem on symplectic cobordisms arising from
spine removal surgery.  This result is then applied in
\S\ref{sec:spineRemoval} to establish new criteria for nonfillability.

\subsubsection*{Acknowledgments}

This project has taken several years to come to fruition, and we are
grateful to many people for valuable conversations along the way,
including especially Denis Auroux, \.{I}nan\c{c} Baykur, Michael Hutchings, 
Tom Mark, Patrick Massot, Richard Siefring, and Otto van Koert.
We would also like to thank the American Institute of Mathematics for bringing
the three of us together at key junctures in this project.

\section{Definitions and results}
\label{sec:results}

\subsection{Main definitions}
\label{sec:defns}

In this section we give the main definitions and state precise versions of
the main results of the paper.

\subsubsection{Types of symplectic fillings}
\label{sec:fillings}

Throughout this paper, we assume all contact structures on oriented 
$3$-manifolds to be \emph{co-oriented} and \emph{positive}, i.e.~they
can always be written as $\xi = \ker\alpha$ where the contact form
$\alpha$ satisfies $\alpha \wedge d\alpha > 0$.
Suppose $(M,\xi)$ is a closed contact $3$-manifold and $(W,\omega)$ is a
compact connected symplectic $4$-manifold with boundary. 
Then $(W,\omega)$ is a \defin{weak filling} of $(M,\xi)$ if
$\p W$ can be identified via an orientation-preserving diffeomorphism with~$M$
such that $\omega|_\xi > 0$.  We also say in this case that $\omega$
\defin{dominates} $\xi$ at the boundary, and that the boundary is
\defin{weakly convex} with respect to~$\omega$.  If there additionally exists a $1$-form
$\lambda$ near $\p W$ that satisfies $d\lambda=\omega$ and restricts to the boundary as a contact form
for~$\xi$ under the above identification $\p W \cong M$, then the boundary is called \defin{convex} and  
$(W,\omega)$ is called a \defin{strong filling} of $(M,\xi)$.
We say that two weak/strong
fillings $(W,\omega)$ and $(W',\omega')$ of contact manifolds
$(M,\xi)$ and $(M',\xi')$ respectively are weakly/strongly 
\defin{symplectically deformation equivalent} if there exists a diffeomorphism
$\varphi : W \to W'$ and smooth $1$-parameter families of symplectic
structures $\{ \omega_\tau \}_{\tau \in [0,1]}$ on $W$ and
contact structures $\{ \xi_\tau \}_{\tau \in [0,1]}$ on $M$ such that
$\omega_0 = \omega$, $\omega_1 = \varphi^*\omega'$, $\xi_0 = \xi$,
$\xi_1 = \varphi^*\xi'$, and $(W,\omega_\tau)$ is a weak/strong filling
of $(M,\xi_\tau)$ for each $\tau \in [0,1]$.  Note that by
Gray's stability theorem, deformation equivalence implies that
$(M,\xi)$ and $(M',\xi')$ must be contactomorphic.

Recall that a symplectic $4$-manifold is
said to be \defin{minimal} if it does not contain any \emph{exceptional
spheres}, i.e.~symplectically embedded $2$-spheres with self-intersection
number~$-1$.  By an argument due to McDuff \cite{McDuff:rationalRuled}, minimality
is invariant under (strong or weak) symplectic deformation.\footnote{In our
context, McDuff's argument that minimality is preserved under deformations 
depends on the conditions we impose on $\omega$ at~$\p W$:
these guarantee in particular that one can always make the boundary
$J$-convex for a tame almost complex structure~$J$, thus preventing
$J$-holomorphic spheres from escaping the interior.}

We call $(W,\omega)$ an \defin{exact filling} of $(M,\xi)$ if it is a strong
filling such that the $1$-form $\lambda$ as defined above near the boundary
extends to a global primitive of~$\omega$ on~$W$.  In this case $(W,d\lambda)$
is also called a \defin{Liouville domain}, with \defin{Liouville form}
$\lambda$, which determines the \defin{Liouville vector field} $V_\lambda$
via the condition
$$
\omega(V_\lambda,\cdot) = \lambda.
$$
Two exact fillings are said to be \defin{Liouville deformation equivalent} if 
they are strongly symplectically deformation equivalent and each of the symplectic
structures in the smooth homotopy defines an exact filling.  Note that 
for any fixed $\omega$ on a Liouville domain, the space of Liouville forms
$\lambda$ satisfying $d\lambda = \omega$ is convex, thus every Liouville
deformation in this sense can be realized by a smooth homotopy of
Liouville forms.

Finally, a \defin{Stein filling} of $(M,\xi)$ is a compact 
connected complex manifold 
$(W,J)$, also called a \defin{Stein domain}, with oriented boundary 
identified with~$M$ such that $\xi \subset TM$
is the maximal complex-linear subbundle, and such that there exists a
smooth function $f : W \to \RR$ that has the boundary as a regular level set
(we say that $f$ is \defin{exhausting})
and is \defin{plurisubharmonic}.  The latter means that $\lambda_J := -df
\circ J$ is a Liouville form and the resulting symplectic form
$\omega_J := d\lambda_J$ \defin{tames}~$J$, i.e.
$$
\omega_J(X,JX) > 0 \text{ for all nonzero $X \in TW$.}
$$
Two Stein fillings are \defin{Stein deformation equivalent} if they
can be identified via a diffeomorphism so that the two complex structures
are homotopic through a smooth family of integrable complex structures
that all admit exhausting plurisubharmonic functions.

Note that for a given~$J$, the space of exhausting plurisubharmonic functions
is convex, and the plurisubharmonicity condition is open with respect to~$J$;
one can use these facts to show that any smooth homotopy of Stein structures
can be accompanied by a smooth homotopy of exhausting plurisubharmonic
functions.  By the correspondence $J \mapsto \lambda_J \mapsto \omega_J$
defined above, it follows that
a Stein deformation class of Stein fillings always gives rise to a canonical
Liouville deformation class of exact symplectic fillings.
The exact fillings arising in this way have the additional feature that
their Liouville vector fields are gradient-like: indeed, any exhausting
plurisubharmonic function $f : W \to \RR$ on a Stein domain $(W,J)$ is also
a Lyapunov function for the Liouville vector field $V_J$ dual to~$\lambda_J$,
thus giving $(W,\omega_J,V_J,f)$ the structure of a \defin{Weinstein domain}.
We will occasionally make use of the deep theorem from
\cite{CieliebakEliashberg} giving
a one-to-one correspondence between deformation classes of
Stein domains and Weinstein domains respectively.

\begin{remark}
\label{remark:Morse}
Strictly speaking, the function $f$ in a Weinstein structure should always be
required to be Morse (or generalized Morse in the case of deformations),
but on Stein domains this can always be achieved via small perturbations of 
plurisubharmonic functions since the plurisubharmonicity condition is open.
\end{remark}

\subsubsection{Spinal open books}
\label{sec:spinalDef}

The following topological notion will be of central importance in this paper.

\begin{defn}
\label{defn:spinal}
A \defin{spinal open book decomposition} on a compact oriented
$3$-dimensional manifold~$M$, possibly with boundary, 
is a decomposition $M = M\spine \cup M\paper$, where the pieces $M\spine$ and
$M\paper$ (called the \defin{spine} and \defin{paper} respectively) 
are smooth compact $3$-dimensional submanifolds
with disjoint interiors such that $\p M\spine \subset \p M\paper$,
carrying the following additional structure:
\begin{enumerate}
\item A smooth fiber bundle $\pi\spine : M\spine \to \Sigma$ with connected
and oriented
fibers, all of which are either disjoint from $\p M\spine$ or contained in it.
Here, $\Sigma$ is a compact oriented surface whose connected components 
(called \defin{vertebrae}\footnote{The use of the bookbinding metaphor for open
book decompositions was the original inspiration for our choice of the 
terms ``spine'' and ``paper'', though the alternative anatomical 
meaning of ``spine'' also has some advantages.  The term ``vertebrae'' makes
sense especially when one observes that the fibration $\pi\spine : M\spine \to \Sigma$
is necessarily trivial, thus the spine can be foliated by vertebrae.  It
makes less sense perhaps in higher-dimensional analogues of spinal
open books, where $\pi\spine : M\spine \to \Sigma$ need not
always be a trivial fibration.})
all have nonempty boundary.
\item A smooth fiber bundle $\pi\paper : M\paper \to S^1$ with oriented fibers
whose connected components (called \defin{pages}) are each preserved
by the monodromy map, have nonempty boundary and meet $\p M\paper$ transversely.  
Moreover, the intersection of any fiber of $\pi\paper$ with $M\spine$ consists of
fibers of~$\pi\spine$.
\item At each connected boundary component $T \subset \p M$ 
(which is necessarily a
$2$-torus component of $\p M\paper$), there is a preferred homology class
$m_T \in H_1(T)$ with the property that if $f_T \in H_1(T)$
denotes the homology class of a connected component of 
$\pi\paper^{-1}(*) \cap T$ oriented as boundary of the fiber,
then $(m_T,f_T)$ defines a positively oriented basis of $H_1(T) \cong \ZZ^2$
for the boundary orientation of~$\p M$.
We call $m_T$ the \defin{preferred meridian} at~$T$.
\end{enumerate}
\end{defn}

We should emphasize that in the above definition, neither $M$ nor its
paper or spine is required to be connected, though pages and vertebrae are
connected by definition.  One can also allow the spine to be empty,
in which case $M$ must have nonempty boundary (since the pages do).
We shall typically denote the full collection of
data defining a spinal open book on~$M$ by
$$
\boldsymbol{\pi} := \Big(\pi\spine : M\spine \to \Sigma,
\pi\paper : M\paper \to S^1, \{m_T\}_{T \subset \p M}\Big).
$$
For any connected component $\gamma \subset \p \Sigma$, the fact that
boundary components of pages are also fibers of $\pi\spine$ means that there
is a well-defined map
\begin{equation}
\label{eqn:multiplicity}
\gamma \to S^1 : \phi \mapsto \pi\paper(\pi\spine^{-1}(\phi)).
\end{equation}
This map is always a diffeomorphism for ordinary open books
(see Example~\ref{ex:openbook}), but more generally it may be a finite cover.

\begin{defn}
\label{defn:multiplicity}
Given the spinal open book $\boldsymbol{\pi}$ as described above, we define 
the \defin{multiplicity} of $\pi\paper$ at a boundary component
$T \subset \p M\paper$ as the number of distinct page boundary components 
that touch~$T$.  If $T \subset M\paper \cap M\spine$, then the multiplicity
can equivalently be described as the degree
of the map $\gamma \to S^1$ defined in \eqref{eqn:multiplicity}.
\end{defn}

\begin{defn}
\label{defn:supported}
Given a spinal open book $\boldsymbol{\pi}$ on~$M$,
a positive contact form $\alpha$ on~$M$ will be called a
\defin{Giroux form} for $\boldsymbol{\pi}$ if the following conditions
hold:
\begin{enumerate}
\item The $2$-form $d\alpha$ is positive on the interior of every page;
\item The Reeb vector field $R_\alpha$ is positively tangent to every
oriented fiber of $\pi\spine : M\spine \to \Sigma$;
\item At $\p M$, $R_\alpha$ is positively tangent to the fibers of 
$\pi\paper|_{\p M} : \p M \to S^1$ and the
characteristic foliation defined by $\ker\alpha$ on $\p M$ has only closed 
leaves, which are homologous on each connected component $T \subset \p M$
to the preferred meridian~$m_T$.\footnote{The 
characteristic foliation $\ker (\alpha|_{T(\p M)}) \subset T(\p M)$ is 
oriented by any vector field $X$ that satisfies $\Omega(X,\cdot) = 
\alpha|_{T(\p M)}$ for a positive area form $\Omega$ on~$\p M$.}

\end{enumerate}
A contact structure $\xi$ on~$M$ will be said to be
\defin{supported} by $\boldsymbol{\pi}$ whenever it admits a contact form which
is a Giroux form.
\end{defn}

In order to obtain the existence and uniqueness of contact structures 
supported by a given spinal open book, technical issues will require us to
examine the smooth compatibility of the spine and paper at their common
boundary components slightly closer.

\begin{defn}
\label{defn:overlap}
We will say that a spinal open book $\boldsymbol{\pi}$ admits
a \defin{smooth overlap} if the fibration $\pi\paper : M\paper \to S^1$
can be extended over an open neighborhood $M\paper' \subset M$ containing
$M\paper$ such that all fibers of $\pi\spine$ intersecting $M\paper'$ are
contained in fibers of the extended~$\pi\paper$.
\end{defn}

\begin{remark}
\label{remark:smoothingSOB}
Any spinal open book can be modified, via a pair of smooth 
isotopies on the spine and paper which match on their common boundary 
components, so as to produce a spinal open book admitting a smooth overlap.
The result of this ``smoothing'' operation is also unique up to isotopy.
\end{remark}

In \S\ref{sec:support} we shall prove the following generalization of
the standard theorem of Thurston and Winkelnkemper 
\cite{ThurstonWinkelnkemper} on open books:

\begin{thm}
\label{thm:GirouxForms}
Suppose $M$ is a compact oriented $3$-manifold, possibly with boundary,
and $\boldsymbol{\pi}$ is a
spinal open book on~$M$ which admits a smooth overlap.  Then the
space of Giroux forms for $\boldsymbol{\pi}$ is nonempty and contractible.
In particular, any isotopy class of spinal open books gives rise to a
canonical isotopy class of supported contact structures.
\end{thm}

\begin{remark}
\label{remark:Gray}
When $\p M \ne \emptyset$, the above statement about uniqueness up to isotopy depends on the following
version of Gray's stability theorem for manifolds with boundary: 
a smooth $1$-parameter family of contact structures on a compact manifold with
boundary is induced by a smooth isotopy if and only if the resulting
characteristic foliations at the boundary are all isotopic.
This follows by a variation on the usual proof of the standard version
(see e.g.~\cite{Geiges:book}): if the characteristic foliations are isotopic,
then after an isotopy near the boundary one can assume they are constant,
and then check that the contact isotopy constructed in the standard way is
generated by a vector field tangent to the boundary.
For this reason it is important that supported contact structures always induce
characteristic foliations on $\p M$ with closed leaves in a fixed
homology class.
\end{remark}

\begin{example}
\label{ex:openbook}
An ordinary open book is the special case of 
a spinal open book where the spine is a tubular neighborhood of a 
transverse link $B \subset M$, 
i.e.~each connected component of $M\spine$ is of the form
$\DD^2 \times S^1$, and the multiplicities of Definition~\ref{defn:multiplicity}
are all~$1$.  Our definition of a
Giroux form in this case does not quite match the standard one, but a
Giroux form in our sense can be perturbed to the standard version,
so the notion of a supported contact structure is the same.
\end{example}

\begin{example}
\label{ex:rationalOB}
In the previous example, relaxing the condition that all multiplicities 
equal~$1$ generalizes from open books to certain types of \emph{rational}
open books as in \cite{BakerEtnyreVanhorn}.
\end{example}

\begin{example}
\label{ex:bindingSum}
Any \emph{blown up summed open book} as defined in \cite{Wendl:openbook2}
can be viewed as a spinal open book whose vertebrae are all disks or annuli.
For instance, one can understand
the \emph{binding sum} construction of \cite{Wendl:openbook2} as follows.
Topologically, it is defined by taking an ordinary open book 
$\pi : M \setminus B \to S^1$
with at least two binding circles $B_1, B_2 \subset B$, removing
tubular neighborhoods of $B_1$ and $B_2$ and attaching the resulting 
boundary tori
by an orientation reversing diffeomorphism that maps oriented 
boundaries of pages to each other 
and maps meridians to meridians (with reversed 
orientation).  In terms of
spinal open books, this is the same as removing two solid torus components
$$
(\DD^2\times S^1) \amalg (\DD^2 \times S^1) \subset M\spine
$$
from the spine and replacing these with $([-1,1] \times S^1) \times S^1$,
which we view as a spinal component with the annulus as a vertebra.
In contact geometric terms, the binding sum on a supported contact
structure produces a contact fiber sum (cf.~\cite{Geiges:book}), 
and it is not hard to show that
the resulting contact structure is supported by the spinal open book
described above.
\end{example}

\begin{example}
\label{ex:blowup}
Spinal open books with boundary can always be constructed from closed
spinal open books by deleting components of the spine and then choosing
suitable preferred meridians.  For example,
suppose ${\boldsymbol{\pi}}$ is an ordinary open book as characterized
in Example~\ref{ex:openbook}, so it is a spinal open book whose spinal
components are all trivial fibrations $\DD^2 \times S^1 \to \DD^2$.
We can then define a new spinal open book by deleting one such 
component $\DD^2 \times S^1$ from the spine; this produces a new boundary
component on the paper, which inherits a canonical meridian, namely
$[\p\DD^2 \times \{*\}] \in H_1(\p(\DD^2 \times S^1))$.
Topologically this has the effect
of removing a tubular neighborhood of one binding component, and the effect
on supported contact structures is exactly what is described in
\cite{Wendl:openbook2} as \emph{blowing up} along the binding.
\end{example}

\begin{remark}
\label{remark:corner}
Many of the notions of this section are also well defined without assuming
that $M$ is a globally smooth manifold: to define a spinal open book, $M$ must
at minimum be a topological
manifold that is obtained by gluing together two smooth manifolds
$M\spine$ and $M\paper$ along a smooth embedding $\p M\spine \hookrightarrow
\p M\paper$, so there are well-defined smooth structures on $M\spine$, 
$M\paper$ and $M\spine \cap M\paper$ but not necessarily on a 
neighborhood of the latter in~$M$.
In particular, it will be useful in the next section to take $M = \p E$
where $E$ is a smooth $4$-manifold with boundary and corners;
here the smooth faces of the boundary are $M\spine$ and $M\paper$ and
the corner is $M\spine \cap M\paper$.  In this case, a
Giroux form will be assumed to be the restriction to $M = \p E$ of a smooth
$1$-form on a neighborhood of the boundary in~$E$, such that the conditions
of Definition~\ref{defn:supported} are satisfied separately on each of the
smooth faces $M\spine$ and~$M\paper$.
\end{remark}

\subsubsection{Bordered Lefschetz fibrations}
\label{sec:BLF}

The motivating example of a spinal open book is obtained by considering
boundaries of Lefschetz fibrations.
In the following, we assume
$E$ to be a smooth, compact, oriented and connected $4$-manifold with 
boundary and  corners such that $\p E$ is the union of two smooth faces
$$
\p E = \p_h E \cup \p_v E
$$
which intersect at a corner of codimension two.  Likewise,
$\Sigma$ will denote a compact, oriented and connected surface with
nonempty boundary.

\begin{defn}
\label{defn:bordered}
A \defin{bordered Lefschetz fibration} of $E$ over $\Sigma$ is a smooth
map $\Pi : E \to \Sigma$ with finitely many interior critical points
$E\crit \subset \mathring{E}$ and critical values 
$\Sigma\crit \subset \mathring{\Sigma}$ such that the following
conditions hold:
\begin{enumerate}
\item $\Pi^{-1}(\p\Sigma) = \p_v E$ and $\Pi|_{\p_v E} : \p_v E \to \p\Sigma$
is a smooth fiber bundle;
\item $\Pi|_{\p_h E} : \p_h E \to \Sigma$ is also a smooth fiber bundle;
\item There exist integrable complex structures near $E\crit$ and
$\Sigma\crit$ such that $\Pi$ is holomorphic near $E\crit$ and
the critical points are nondegenerate;
\item All fibers $E_z := \Pi^{-1}(z)$ for $z \in \Sigma$ are connected and have
nonempty boundary in $\p_h E$.
\end{enumerate}
We call $E_z$ a \defin{regular fiber} if $z \in \Sigma\setminus
\Sigma\crit$ and otherwise a \defin{singular fiber}; the latter are
necessarily unions of smoothly immersed connected surfaces
(the \defin{irreducible components})
with positive transverse intersections.  We say that $\Pi$ is
\defin{allowable} if all the irreducible components of its fibers have
nonempty boundary.  
\end{defn}

By the complex Morse lemma, one can find holomorphic coordinates near
$E\crit$ and $\Sigma\crit$ so that $\Pi$ takes the form
$$
\Pi(z_1,z_2) = z_1^2 + z_2^2
$$
near each critical point.
Note also that in the standard language of vanishing cycles
(cf.~\cite{GompfStipsicz}), the ``allowability'' condition defined above is
equivalent to requiring that no vanishing cycles be homologically 
trivial in the fiber.\footnote{In some sources in the literature, it is erroneously
stated that a Lefschetz fibration is allowable if and only if its vanishing
cycles are always nonseparating in the fiber.  We will often want to consider
situations in which vanishing cycles are homologically nontrivial but
separating, e.g.~when the fiber is an annulus.
In the case where fibers have genus zero, a
Lefschetz fibration is allowable if and only if it is
\emph{relatively minimal}.}

A bordered Lefschetz fibration $\Pi : E \to \Sigma$ naturally gives rise to a
spinal open book on $\p E$, with spine $M\spine := \p_h E$ and paper 
$M\paper := \p_v E$.
The fibration $\pi\paper : \p_v E \to S^1$ is defined as the restriction 
$\Pi|_{\p_v E} : \p_v E \to \p\Sigma$ after choosing an orientation preserving
identification of each connected component of~$\p\Sigma$ with~$S^1$.
Likewise, $\Pi|_{\p_h E} : \p_h E \to \Sigma$ defines a smooth fibration
whose fibers are disjoint unions of finitely many circles, hence it can be 
factored as
$$
\p_h E \stackrel{\pi\spine}{\longrightarrow} \widetilde{\Sigma} 
\stackrel{p}{\longrightarrow} \Sigma,
$$
where $\pi\spine : \p_h E \to \widetilde{\Sigma}$ is a fiber bundle with 
connected fibers over another compact oriented surface 
$\widetilde{\Sigma}$ with boundary, and $p : \widetilde{\Sigma} \to \Sigma$
is a smooth finite covering map.
As discussed in Remark~\ref{remark:corner}, the fact that
$\p E$ is not naturally a smooth manifold does not present any problem here.
In this class of examples, 
every vertebra is a finite cover of the base $\Sigma$, the 
pages are all
diffeomorphic to the regular fibers of~$\Pi$, and the boundary components of 
these fibers form the fibers on the spine.
Whenever $\boldsymbol{\pi}$ is a spinal open book on a $3$-manifold $M$
admitting a homeomorphism to~$\p E$ that restricts to diffeomorphisms
$M\spine \to \p_h E$ and $M\paper \to \p_v E$ such that $\boldsymbol{\pi}$
is related to $\Pi : E \to \Sigma$ as described above,
we shall indicate this relationship by writing
$$
\p\Pi \cong \boldsymbol{\pi}.
$$

Clearly not all spinal open books can be obtained as boundaries of
Lefschetz fibrations, so those that can deserve a special name.
\begin{defn}
\label{defn:simple}
A spinal open book $\boldsymbol{\pi}$ on a $3$-manifold 
$M$ will be called \defin{symmetric} if 
\begin{enumerate}[label=(\roman{enumi})]
\item $\p M = \emptyset$;
\item All pages are diffeomorphic;
\item For each of the vertebrae $\Sigma_1,\ldots,\Sigma_r \subset \Sigma$,
there are corresponding numbers $k_1,\ldots,k_r \in \NN$ such that every
page has exactly $k_i$ boundary components in $\pi\spine^{-1}(\p\Sigma_i)$ for
$i=1,\ldots,r$.
\end{enumerate}
We shall say that $\boldsymbol{\pi}$ is \defin{uniform} if, in addition
to the above conditions, there exists a fixed compact
oriented surface $\Sigma_0$ whose boundary components correspond bijectively
with the connected components of $M\paper$ such that for each $i=1,\ldots,r$ 
there exists a $k_i$-fold branched cover 
$$
\Sigma_i \to \Sigma_0
$$
for which the restriction to each connected boundary component $\gamma \subset \p\Sigma_i$
is an $m_\gamma$-fold cover of the component of $\p\Sigma_0$ corresponding
to the component of~$M\paper$ touching $\pi\spine^{-1}(\gamma)$,
where $m_\gamma$ denotes the multiplicity of $\pi\paper$ 
at~$\pi\spine^{-1}(\gamma)$ (see Definition~\ref{defn:multiplicity}).
Finally, $\boldsymbol{\pi}$ is \defin{Lefschetz-amenable} if it is uniform
and all branched covers satisfying the above conditions have no branch
points.
\end{defn}

\begin{remark}
\label{remark:easy}
In many examples of interest---in particular for the circle bundles over
oriented surfaces studied in \S\ref{subsec:circleBundles} and further in
\cite{LisiVanhornWendl2}, $\boldsymbol{\pi}$ is symmetric with
$k_1 = \ldots = k_r = 1$, in which case it is uniform if and only if
all vertebrae are diffeomorphic.  The Lefschetz-amenability condition
is trivially satisfied in such cases since branched covers of degree~$1$
are diffeomorphisms.  In more general situations, the uniformity and
amenability conditions can often both be checked via the Riemann-Hurwitz
formula;\footnote{Note that by capping $\Sigma_i$ and $\Sigma_0$ with disks,
the existence of the required branched cover $\Sigma_i \to \Sigma_0$ is equivalent
to a question about the existence of a branched cover of closed surfaces with
certain prescribed branching orders.  Questions of this type can be subtle in
general, but are trivial e.g.~if the degree is~$2$, or more generally if all
branch points are required to be simple, cf.~\cite{EdmondsKulkarniStong}*{Prop.~2.8}.}; 
we will use this in \cite{LisiVanhornWendl2} to classify the fillings of 
certain non-orientable contact circle bundles over
non-orientable surfaces.
\end{remark}

The discussion above shows that for any bordered Lefschetz fibration 
$\Pi : E \to \Sigma$, the spinal open book $\boldsymbol{\pi} := \p\Pi$ is 
necessarily uniform, and the associated branched covers
$\Sigma_i \to \Sigma$ have no branch points.  The more precise version of
Theorem~\ref{thma:main} proved in \cite{LisiVanhornWendl2} will imply
that every spinal open book which contains a planar page and supports a
strongly fillable contact structure must be uniform---moreover, if it is
also amenable, then its strong fillings can be classified entirely in terms
of Lefschetz fibrations.

\begin{example}
For any bordered Lefschetz fibration over the disk,
the spinal open book induced at its boundary is an ordinary open book
(see Example~\ref{ex:openbook}).  In fact,
any ordinary open book on a closed and connected $3$-manifold, when regarded
as a spinal open book, is uniform and Lefschetz-amenable.  
Of course not every open book is the
boundary of a bordered Lefschetz fibration; this depends on its
monodromy!
\end{example}

\begin{example}
For a bordered Lefschetz fibration over the annulus, if the fibration 
restricted to the horizontal boundary is trivial, then the induced spinal
open book at the boundary is equivalent to a \emph{symmetric summed open book}
as defined in \cite{Wendl:openbook2}.
\end{example}

We now define various types of symplectic structures that are natural to
consider on the total space of
a bordered Lefschetz fibration $\Pi : E \to \Sigma$.
Note that the orientations of $E$ and $\Sigma$ give rise to a natural
orientation of the fibers.  We shall say
that a symplectic form $\omega$ on~$E$ is \defin{supported} by~$\Pi$ whenever
the following conditions hold:
\begin{enumerate}
\item Every oriented fiber is a symplectic submanifold away from $E\crit$;
\item A neighborhood of $E\crit$ admits a smooth almost complex structure $J$ 
which restricts to a positively oriented complex structure on the smooth
part of each fiber and satisfies $\omega(v, J v) > 0$ for every nonzero vector
$v \in TE|_{E\crit}$, i.e.~$J$ is \emph{tamed} by $\omega$ at $E\crit$.
\end{enumerate}

For the following definitions, assume always that $\omega$ is a symplectic
structure supported by~$\Pi$.

\begin{defn}
\label{defn:weak}
We say that $\omega$ is \defin{weakly convex} if it can be written near
$\p_h E$ as $\omega = d\lambda$, where $\lambda$ is a smooth $1$-form
that restricts to $\p_h E$ as a contact form whose Reeb orbits are 
boundary components of fibers.
\end{defn}

\begin{defn}
\label{defn:strong}
We say that $\omega$ is \defin{strongly convex} if it can be written
near $\p E$ as $\omega = d\lambda$, where $\lambda$ is a smooth $1$-form
that restricts to $\p E$ as a Giroux form for $\boldsymbol{\pi} = \p\Pi$
(see Remark~\ref{remark:corner}).
\end{defn}

\begin{defn}
\label{defn:Liouville}
We say that $\omega$ is \defin{Liouville} if it is strongly convex and
the primitive $\lambda$ of Definition~\ref{defn:strong} extends to a global
primitive of~$\omega$ on~$E$.
\end{defn}

These three definitions are designed so that a suitable smoothing
of $E$ at the corners 
will inherit the structure of a weak/strong/exact
symplectic filling of $(M,\xi)$, with $\xi$ supported by~$\boldsymbol{\pi}$,
see \S\ref{subsec:smoothing}.

To move from the Liouville to the Stein case, it will be convenient to
introduce a notion that is intermediate between Weinstein and Stein
structures.

\begin{defn}
\label{defn:almostStein}
Suppose $W$ is a compact manifold with boundary, possibly also with corners.
An \defin{almost Stein structure} on $W$ is a pair $(J,f)$ consisting of an 
almost complex structure $J$ and a smooth function $f : W \to \RR$ such
that, writing $\lambda := -df \circ J$, $d\lambda$ is a symplectic form
taming~$J$ (i.e.~$f$ is \defin{$J$-convex}) and $\lambda$ restricts to a 
contact form on every smooth face of~$\p W$.  If $M := \p W$ is smooth and
$\xi = \ker (\lambda|_{TM})$, we will call $(W,J,f)$ an \defin{almost Stein
filling} of~$(M,\xi)$.
\end{defn}

We assign the natural $C^\infty$-topology to the space of almost Stein
structures and say that two such structures are \defin{almost Stein
homotopic} if they lie in the same connected component of this space.
Any Stein structure $J$ determines an almost Stein structure $(J,f)$
uniquely up to homotopy, where uniqueness follows from the fact that the
space of exhausting $J$-convex functions is convex.
We should point out two aspects of almost Stein structures that differ from
Stein structures: first, $J$ is not assumed integrable, and second,
$f$ is not assumed constant at the boundary (indeed, it \emph{cannot} be
constant on $\p W$ if there are corners).  The latter has the consequence
that for a fixed~$J$, the space of functions $f$
making $(J,f)$ an almost Stein structure is \emph{not} generally convex---linear 
interpolations 
between two such functions may fail to induce contact structures at the
boundary.  For this reason we can no longer regard the $J$-convex function
as auxiliary data.
It is clear however that any almost Stein structure $(J,f)$ determines a 
Weinstein structure on the smooth manifold with boundary obtained by
rounding the corners of $\p W$, and this structure is
canonical up to Weinstein homotopy. Indeed,
the Liouville form $-df \circ J$ is dual to a Liouville vector field that
points transversely outward at every smooth face of $\p W$, and this vector
field is automatically gradient-like with respect to~$f$.
One can therefore perturb $f$ if necessary to make it Morse (cf.~Remark~\ref{remark:Morse}),
and then modify it outside a neighborhood of its critical points to a Lyapunov
function that is constant on the smoothed boundary, the result being unique
up to homotopy through Lyapunov functions fixed near the
critical points.  Using \cite{CieliebakEliashberg}, this 
implies that for any manifold $W$ with
boundary and corners, there is a canonical one-to-one correspondence between 
almost Stein homotopy classes on $W$ 
and Stein homotopy classes on $W$ after smoothing corners.

\begin{defn}
\label{defn:almostSteinFibration}
Given a bordered Lefschetz fibration $\Pi : E \to \Sigma$, we will
say that an almost Stein structure $(J,f)$ on~$E$ is \defin{supported by~$\Pi$} 
if the following conditions are satisfied:
\begin{itemize}
\item
There exists a complex structure $j$ on $\Sigma$ such that
$\Pi : (E,J) \to (\Sigma,j)$ is pseudoholomorphic;
\item
The $1$-form $\lambda_J := -df \circ J$ restricts to
$\p E$ as a Giroux form for $\boldsymbol{\pi} = \p\Pi$
(see Remark~\ref{remark:corner});
\item
For every $z \in \Sigma$, $f$ is constant on each connected component
of $\p E_z$;
\item
The maximal $J$-complex subbundle of $T(\p_h E)$ is invariant under the
Reeb flow determined by $\lambda_J|_{T(\p_h E)}$.
\end{itemize}
\end{defn}

Observe that if $(J,f)$ is supported by~$\Pi$, then every fiber $E_z \subset E$
inherits a Stein structure $J|_{TE_z}$ with plurisubharmonic function
$f|_{E_z}$, so in particular the fibers are both $J$-holomorphic and
symplectic, and $-d(df \circ J)$ defines a supported Liouville structure.

The following variation on results of Thurston \cite{Thurston:symplectic}
and Gompf \cite{GompfStipsicz} will be proved in \S\ref{sec:Gompf}.

\begin{thm}
\label{thm:Gompf}
For any $4$-dimensional bordered Lefschetz fibration $\Pi : E \to \Sigma$, 
the spaces of supported symplectic
structures that are weakly or strongly convex are both nonempty and
contractible.  Moreover, if the Lefschetz fibration is allowable, then
the same is true for the spaces of supported Liouville structures 
and almost Stein structures.  In each case, the corners can be smoothed
to produce
a weak/strong/exact/Stein filling of the contact manifold supported by
$\boldsymbol{\pi} := \p\Pi$, and this filling is canonically defined up to
deformation equivalence.
\end{thm}

\subsection{Surgery on spinal open books}
\label{subsec:surgery}

There are various natural topological operations on spinal open books that
give rise to symplectic cobordisms.  We now briefly describe two such
operations.

\subsubsection{Spine removal surgery}

The following makes precise the non-exact cobordism construction that was sketched in
Theorem~\ref{thma:spineRemoval} of the introduction, generalizing previous
constructions from
\cites{Eliashberg:cap,GayStipsicz:cap,Wendl:cobordisms}
(see also the higher-dimensional analogues in
\cites{MassotNiederkruegerWendl,DoernerGeigesZehmisch,Klukas:openBooks}).
For this discussion, it is useful to allow a slight loosening of 
the main definition of this paper: we will say that a
\defin{generalized spinal open book} is an object satisfying all the
conditions of Definition~\ref{defn:spinal} except that fibers of the
paper $\pi\paper : M\paper \to S^1$ are allowed to have components
with no boundary.  A generalized spinal open book may therefore have some
connected components that have neither spine nor boundary, but are simply
fibrations of closed pages over $S^1$; note that a Giroux form cannot
exist in this case, due to Stokes' theorem.  Such an object is then a spinal
open book in the usual sense---and thus supports a contact structure---if 
and only if it has no closed pages.

Suppose $(M',\xi)$ is a closed contact $3$-manifold containing a
compact $3$-dimensional submanifold $M$ (possibly with boundary) on which~$\xi$
is supported by a spinal open book $\boldsymbol{\pi}$
with spine $\pi\spine : M\spine \to \Sigma$ and paper $\pi\paper :
M\paper \to S^1$.  Choose an open and closed subset
$$
\Sigma\remove \subset \Sigma.
$$
The boundary of the corresponding union of spinal components
$\pi\spine^{-1}(\Sigma\remove) \subset M\spine$ is
a disjoint union of $2$-tori, each foliated by an
$S^1$-family of oriented circles which are fibers of~$\pi\spine$.
We can then define a new compact manifold $\widetilde{M}$ from~$M$ 
(and a closed manifold $\widetilde{M}'$ from $M'$)
by removing the interior of
$\pi\spine^{-1}(\Sigma\remove)$ and attaching solid tori $S^1 \times \DD^2$ to each of
the connected components of $\p \left(\pi\spine^{-1}(\Sigma\remove)\right)$ 
so that the oriented circles
$\{*\} \times \p\DD^2$ match the leaves of the foliation.
(Some schematic pictures of this procedure are shown in
Figures~\ref{fig:spineRemoval} and~\ref{fig:spineRemoval3} in \S\ref{sec:easySpineRemoval}.)
The new domain $\widetilde{M} \subset \widetilde{M}'$ inherits from 
$\boldsymbol{\pi}$ a generalized spinal 
open book $\widetilde{\boldsymbol{\pi}}$ with spine $M\spine \setminus \pi\spine^{-1}(\Sigma\remove)$
and pages that are obtained from the pages of $\boldsymbol{\pi}$ by
attaching disks to cap every boundary component touching~$\pi\spine^{-1}(\Sigma\remove)$.
We say that $\widetilde{\boldsymbol{\pi}}$ is obtained from $\boldsymbol{\pi}$ by
\defin{spine removal surgery}.

The spine removal operation corresponds to a cobordism that can be
understood as a form of handle attachment.  In particular,
we can consider the compact $4$-dimensional manifold with boundary and corners
$$
X := ([0,1] \times M') \cup_{\{1\} \times \pi\spine^{-1}(\Sigma\remove)}
(\Sigma\remove \times \DD^2),
$$
where $\Sigma\remove \times \p\DD^2$ is identified with $\pi\spine^{-1}(\Sigma\remove)$
via a choice of trivialization $\pi\spine^{-1}(\Sigma\remove) \cong \Sigma\remove \times S^1$.
After smoothing the corners, we have
$$
\p X = - M' \amalg \widetilde{M}'.
$$
The following result will be proved in \S\ref{sec:easySpineRemoval}.

\begin{thm}
\label{thm:spineRemoval}
Suppose $\Omega$ is a closed $2$-form on~$M'$ such that
$\Omega|_{\xi} > 0$ and $\Omega|_{\pi\spine^{-1}(\Sigma\remove)}$ is exact.  
Then for any choice of compact subset $\Sigma_0$ in the interior of $\Sigma\remove$,
the cobordism $X$ described above admits 
a symplectic structure $\omega$ with the following properties:
\begin{enumerate}
\item $\omega|_{TM'} = \Omega$.
\item $\omega$ is positive on the interior of every page of $\widetilde{\boldsymbol{\pi}}$.
\item On every connected component of $\widetilde{M}'$ that is not
foliated by closed pages of $\widetilde{\boldsymbol{\pi}}$, there exists a
contact structure $\widetilde{\xi}$ which is supported by $\widetilde{\boldsymbol{\pi}}$
in $\widetilde{M}$, matches $\xi$ on $\widetilde{M}' \setminus \widetilde{M} = M'
\setminus M$, and satisfies $\omega|_{\widetilde{\xi}} > 0$.
\item For every $z \in \Sigma_0$, the \emph{core} $\{z\} \times \DD^2$ and \emph{co-core}
$\Sigma\remove \times \{0\}$ of the handle $\Sigma\remove \times \DD^2$
are both symplectic submanifolds, the former with reversed orientation.
\end{enumerate}
\end{thm}

We will see in the proof that the disk $\DD^2$ in the symplectic handle
$\Sigma\remove \times \DD^2$ can freely be replaced by any other compact oriented
surface with connected boundary; more generally, one could equally well
remove several spine components at once and replace them with $\Sigma\remove \times S$
for a compact oriented surface $S$ with the right number of boundary components.
The key intuition is to view $S$ as a symplectic \emph{cap} for the appropriate
disjoint union of fibers in the contact circle fibration 
$\pi\spine : M\spine \to \Sigma$, and this is also the right perspective in 
higher-dimensional cases such as \cites{DoernerGeigesZehmisch,Klukas:openBooks}.
We will not comment any further on these generalizations since the applications
in this paper do not require them.

\subsubsection{Fiber connected sum along pages}

The following is a special case of a construction due to
R.~Avdek \cite{Avdek:sums}.  As in the previous section, suppose 
$(M',\xi)$ is a closed contact $3$-manifold containing a compact
domain $M \subset M'$ on which $\xi$ is supported by a
spinal open book~$\boldsymbol{\pi}$.  Suppose $S$ is a compact, connected 
and oriented surface with boundary, and $S_0, S_1 \subset M\paper$ are pages
of $\boldsymbol{\pi}$ admitting orientation preserving diffeomorphisms
$$
\psi_i : S \to S_i, \qquad i = 0,1.
$$
By a minor adjustment to the proof of Theorem~\ref{thm:GirouxForms}
(see Lemma~\ref{lemma:fiberLiouville} in particular), one can find a
Giroux form $\alpha$ for $\boldsymbol{\pi}$ on~$M$ such that
$\psi_0^*\alpha = \psi_1^*\alpha$.  In the terminology of
\cite{Avdek:sums}, $S_0$ and $S_1$ can then be regarded as a pair of
identical \emph{Liouville hypersurfaces} in~$M$.  Choose neighborhoods
$[-1,1] \times S_i \cong \nN(S_i) \subset M\paper$ of $S_i$ for $i=0,1$ and 
define the compact $4$-manifold with boundary and corners
$$
X := ([0,1] \times M') \cup_{\nN(S_0) \amalg \nN(S_1)}
\left([0,1] \times [-1,1] \times S\right),
$$
by identifying $\{i\} \times [-1,1] \times S$ with $\nN(S_i)$ for $i=0,1$.
After smoothing corners, we have
$$
\p X = -M' \amalg \widetilde{M}',
$$
where $\widetilde{M}'$ is obtained from $M'$ by performing a
so-called \emph{Liouville connected sum} along $S_0$ and~$S_1$.
Then $\widetilde{M}'$ contains a compact subdomain $\widetilde{M}$ which naturally
carries a spinal open book $\widetilde{\boldsymbol{\pi}}$; it is obtained from
$\boldsymbol{\pi}$ by attaching $1$-handles to the vertebrae and concatenating
families of pages correspondingly.
The following is an immediate consequence of the main result in \cite{Avdek:sums}:

\begin{thmu}
The manifold $X$ described above can be given the structure of a
Stein cobordism with concave boundary $(M',\xi)$ and convex
boundary $(\widetilde{M}',\widetilde{\xi})$, where $\widetilde{\xi}$ is a contact structure
which matches $\xi$ on $\widetilde{M}' \setminus \widetilde{M} = M' \setminus M$
and is supported by $\widetilde{\boldsymbol{\pi}}$ on~$\widetilde{M}$.
\end{thmu}

It was observed in \cite{Avdek:sums} that the simplest case of this
operation turns ordinary open books into \emph{symmetric summed open books}
in the sense of \cite{Wendl:openbook2}, i.e.~disk vertebrae become annuli.
More generally, this construction can be used to give an
alternative proof of the fact that allowable bordered Lefschetz fibrations
over arbitrary compact oriented surfaces always admit Stein 
structures---the details of this argument have been
worked out by Baykur and the second author, see \cite{BaykurVanhorn:large}.

\subsection{Partially planar domains, torsion and filling obstructions}
\label{sec:invariantsSketch}

We now state a few theorems that are straightforward
generalizations of results from \cite{Wendl:openbook2}, and will
all be proved in \S\ref{sec:spineRemoval} using spine removal surgery.
Most of them can also be derived from algebraic counterparts that we will
prove in \cite{LisiVanhornWendl2}, involving contact invariants in
symplectic field theory and embedded contact homology.

The following is the basic condition needed in order to apply the machinery of
pseudoholomorphic curves in studying spinal open books.

\begin{defn}
A $3$-dimensional 
spinal open book will be called
\defin{partially planar} if its interior contains a page of genus zero.
A compact contact
$3$-manifold $(M,\xi)$, possibly with boundary, will be called a
\defin{partially planar domain} if $\xi$ is supported by a partially
planar spinal open book.  We then refer to any interior connected component of
the paper containing planar pages as a \defin{planar piece}.
\end{defn}

\begin{defn}
\label{defn:OmegaSeparating}
Suppose $(M,\xi)$ is a closed contact $3$-manifold and~$\Omega$ is a
closed $2$-form on~$M$.  A partially planar domain $M_0$ embedded 
in $(M,\xi)$ is called \defin{$\Omega$-separating} if it has a planar
piece $M_0^P \subset {\mathring M}_0$ such that $\Omega$ is exact on
every spinal component touching~$M_0^P$.
It is called \defin{fully separating} 
if this is true for all closed $2$-forms~$\Omega$ on~$M$.
\end{defn}

Note that this condition depends only on the cohomology class
$[\Omega] \in H^2_\dR(M)$, and it is vacuous if $\Omega$ is exact.
We will see that it determines precisely which results on the strong fillings
of spinal open books admit extensions for \emph{weak} fillings.

\begin{example}
Since all closed $2$-forms are exact on a solid torus $\DD^2 \times S^1$,
every planar open book is a fully separating partially planar domain
(cf.~Example~\ref{ex:openbook}).  As explained in \cites{Wendl:openbook2,Wendl:fillable},
a Giroux torsion domain can also be viewed as a partially planar domain
in terms of the binding sum construction, but its spinal components are
thickened $2$-tori and thus can have cohomology, so such a domain is fully
separating if and only if it separates the ambient $3$-manifold.
\end{example}

Our first main result about partially planar domains generalizes the main
theorem from \cite{AlbersBramhamWendl}; indeed, taking $\Omega=0$ in the
following statement produces an obstruction to the existence of non-separating
hypersurfaces of contact type.

\begin{thm}
\label{thm:nonseparating}
Suppose $(M,\xi)$ is a closed contact $3$-manifold, $\Omega$ is a closed
$2$-form on~$M$ and $(M,\xi)$ contains an $\Omega$-separating partially
planar domain.  Then there exists no closed symplectic $4$-manifold
$(W,\omega)$ admitting a non-separating embedding $\iota : M \hookrightarrow W$
for which $\iota^*\omega|_\xi > 0$ and $[\iota^*\omega] =
[\Omega] \in H^2_\dR(M)$.
\end{thm}

The following related result generalizes a planarity obstruction originally 
due to Etnyre \cite{Etnyre:planar}.  Recall that $(W,\omega)$ is called a symplectic
\defin{semifilling} of $(M,\xi)$ whenever it is a filling of the disjoint
union of $(M,\xi)$ with some other (possibly empty)
contact manifold.  

\begin{cor}
\label{cor:semifillings}
If $(M,\xi)$ is a closed contact $3$-manifold containing a partially
planar domain, then it admits no weak semifilling $(W,\omega)$ with disconnected
boundary for which the partially planar domain is 
$\left(\omega|_{TM}\right)$-separating.
\end{cor}
\begin{proof}
We use a suggestion by Etnyre that first appeared in \cite{AlbersBramhamWendl}:
if such a semifilling exists, then one can attach a Weinstein $1$-handle to
build a weak filling of the boundary connected sum of its two components,
and then cap the result via \cite{Eliashberg:cap} or \cite{Etnyre:fillings}.
This produces a closed symplectic manifold $(W,\omega)$ that contains $(M,\xi)$ 
as a non-separating hypersurface in violation of Theorem~\ref{thm:nonseparating}.
\end{proof}

Next, we can consider the natural generalization of the local filling 
obstruction known as \emph{planar $k$-torsion} from \cite{Wendl:openbook2}
into the spinal open book setting.

\begin{defn}
Suppose $(M,\xi)$ is a closed contact $3$-manifold and $\Omega$ is a
closed $2$-form on~$M$.
Then for $k \ge 0$ an integer, a partially planar domain
$M_0 \subset M$ is called a \defin{(spinal) planar torsion domain of
order~$k$} (or simply a \emph{planar $k$-torsion domain})
if it is not symmetric and contains an interior planar piece
$M_0^P \subset \mathring{M}_0$ whose pages have $k+1$ boundary
components.  Further, it is an \defin{$\Omega$-separating} planar
$k$-torsion domain if $\Omega$ is exact on all spinal components 
touching $M_0^P$, and a \defin{fully separating} planar $k$-torsion
domain if this is true for all closed $2$-forms $\Omega$ on~$M$.
Any contact $3$-manifold containing such a domain is said to have 
(perhaps \defin{$\Omega$-separating} or \defin{fully separating})
\defin{planar $k$-torsion}.
\end{defn}

The less general version of this definition in \cite{Wendl:openbook2} 
was expressed in the framework of \emph{blown up
summed open books}, i.e.~spinal open books whose vertebrae are all disks
and annuli.  We have inserted the word ``spinal'' in front of 
``planar torsion'' in the above definition to distinguish the new notion 
from the less general version, but we shall usually drop the word ``spinal''
from the nomenclature: this will not cause any confusion since anything
satisfying the old definition also satisfies the new one, and
the results
we are able to prove with the new definition parallel results in
\cite{Wendl:openbook2} almost exactly.  For instance, the less general
version of planar $0$-torsion was shown in \cite{Wendl:openbook2} to be
equivalent to overtwistedness, and this is still true in the new framework:

\begin{prop}
\label{prop:overtwisted}
A closed contact $3$-manifold is overtwisted if and only if it has
planar $0$-torsion.  
\end{prop}
\begin{proof}
If $(M,\xi)$ is overtwisted, then Eliashberg's flexibility result
\cite{Eliashberg:overtwisted} implies that $(M,\xi)$ contains a
so-called \emph{Lutz tube}, and any neighborhood of this contains
a planar $0$-torsion domain by \cite{Wendl:openbook2}*{Prop.~2.19}.
For the converse, see Lemma~\ref{lemma:overtwisted}.
\end{proof}

It was also shown in \cite{Wendl:openbook2} that anything with Giroux
torsion also has planar $1$-torsion, but it was left open whether the
converse might be true.  Working with spinal open books makes it easy to
find a counterexample to this converse:

\begin{prop}
\label{prop:GT}
If $(M,\xi)$ is a closed contact $3$-manifold with positive Giroux torsion 
then it has planar $1$-torsion.  However, there exist closed
contact $3$-manifolds that have planar $1$-torsion but no
Giroux torsion.
\end{prop}
\begin{proof}
The fact that Giroux torsion implies $1$-torsion was shown in
\cite{Wendl:openbook2}; in fact, any Giroux torsion domain has an open
neighborhood that contains a planar $1$-torsion domain whose pages
and vertebrae are all annuli.  Some examples with planar $1$-torsion
but no Giroux torsion are exhibited in~\S\ref{subsec:circleBundles};
see Corollary~\ref{cor:S1planarTorsion} and
Remark~\ref{remark:noGirouxTorsion}.
\end{proof}

We will prove the following statement in \S\ref{sec:spineRemoval} by using
spine removal surgery to reduce it to standard results in closed
holomorphic curve theory.

\begin{thm}
\label{thm:planarTorsion}
If $(M,\xi)$ has planar torsion, then it is not strongly fillable.
Moreover, if $(M,\xi)$ has $\Omega$-separating planar torsion for
some closed $2$-form $\Omega$ on~$M$, then it admits no weak filling
$(W,\omega)$ with $\omega|_{TM}$ cohomologous to~$M$.  In particular
$(M,\xi)$ is not weakly fillable whenever it has fully separating
planar torsion.
\end{thm}

\subsection{Fillability of circle bundles}
\label{subsec:circleBundles}

As an application of the filling obstructions in the previous subsection,
we now exhibit a large class of non-fillable contact $3$-manifolds that were
not previously accessible to holomorphic curve methods.  (Some of them can
be understood using techniques from Heegaard Floer homology; see
especially \cites{HondaKazezMatic,Massot:vanishing}.)
They take the form of circle bundles with $S^1$-invariant contact structures
partitioned by multicurves.

Throughout this subsection, assume $\pi : M \to B$ is a smooth $S^1$-bundle 
with structure group $\Ortho(2)$ acting on the circle by rotations and
reflections, where the base $B$ is a
closed and connected (but not necessarily orientable) surface, and the
total space $M$ is oriented.  If $B$ is orientable, then the
$\Ortho(2)$-structure lifts to the structure of a principal $S^1$-bundle,
with the $S^1$-action defined up to a sign, so we can speak of 
$S^1$-invariant contact structures on~$M$.  More generally, 
we will abuse terminology and call a contact structure \defin{$S^1$-invariant}
if its expression in every $\Ortho(2)$-compatible local trivialization
of $\pi : M \to B$ is $\Ortho(2)$-invariant.  This is the same as saying
that it lifts to an $S^1$-invariant contact structure on the induced fibration
over the canonical oriented double cover of~$B$.
As usual, all contact structures in this discussion are assumed to be
positive and co-oriented.

Any $S^1$-invariant contact structure $\xi$ on~$M$ determines a $1$-dimensional
submanifold $\Gamma \subset B$ (i.e.~a \defin{multicurve}) consisting of all points at 
which the fiber is Legendrian.  One says in this case that $\xi$ is
\defin{partitioned by}~$\Gamma$.  Notice that outside of $\Gamma$, 
transversality to $\xi$ determines an orientation of the bundle and therefore 
an orientation of $B \setminus \Gamma$.  Moreover, $\Gamma$ automatically
has the following property:

\begin{defn}
\label{defn:localOrient}
Suppose $B$ is a closed surface and $\Gamma \subset B$ is a multicurve such
that $B \setminus \Gamma$ is oriented.  We say that $\Gamma$ \defin{inverts
orientations} if for every sufficiently small neighborhood $\uU \subset B$ 
that is divided by $\Gamma$ into two components $\uU_+$ and $\uU_-$, $\uU$ 
can be given an orientation that matches that of $B \setminus \Gamma$
on $\uU_+$ and is the opposite on~$\uU_-$.
\end{defn}

If $B$ is orientable, this condition simply means that $\Gamma$ divides
$B$ into components $B_+$ and $B_-$ (each possibly disconnected) which
inherit opposite orientations.  Some concrete examples where $B$ is
non-orientable (in particular the Klein bottle) can be constructed in the
form of contact parabolic torus bundles; see \cite{LisiVanhornWendl2}.
The following result is due to Lutz \cite{Lutz:77} in the orientable case,
and in general it can easily be derived from 
Theorem~\ref{thm:GirouxForms} via Proposition~\ref{prop:S1SOBD} below.

\begin{prop}
\label{prop:Lutz}
Suppose $\pi : M \to B$ is a smooth circle bundle with structure group
$\Ortho(2)$, where $B$ is a closed connected surface and $M$ is oriented,
and $\Gamma \subset B$ is a nonempty multicurve such that $B \setminus \Gamma$ is
orientable and $\Gamma$ inverts orientations.  Then each choice of orientation 
on $B \setminus \Gamma$ determines an $S^1$-invariant contact structure 
$\xi_\Gamma$ that is partitioned by~$\Gamma$ and is positively transverse to 
the fibers over $B \setminus \Gamma$.  Moreover, the contact structure with 
these properties is unique up to isotopy.
\end{prop}

We will see in \cite{LisiVanhornWendl2} that the strong symplectic fillings
of each circle bundle $(M,\xi_\Gamma)$ arising from the above proposition
can be classified completely whenever its base is
orientable, and also in some cases where the base is not orientable.  The
basic observation behind this is that there is a
natural correspondence between $S^1$-bundles $\pi : M \to B$ with nonempty
multicurves $\Gamma \subset B$ satisfying the stated conditions in
Prop.~\ref{prop:Lutz} and spinal 
open book decompositions of $M$ with annular pages.  Topologically this is 
easy to see: choosing a tubular neighborhood $\uU_\Gamma \subset B$ of 
$\Gamma$, we identify each connected component of the closure 
$\overline{\uU}_\Gamma$ with an interval bundle over~$S^1$, which gives 
$\pi^{-1}(\overline{\uU}_\Gamma)$ 
the structure of a disjoint union of smooth annulus bundles over~$S^1$ whose 
fibers each have boundary equal to some pair of oriented fibers of~$\pi$.  We therefore
call $\pi^{-1}(\overline{\uU}_\Gamma)$ with its associated fibration
over $S^1$ the paper $\pi\paper : M\paper \to S^1$, and the spine
$\pi\spine : M\spine \to \Sigma$ is defined as the restriction of $\pi$ to
$\pi^{-1}(B \setminus \uU_\Gamma)$, with the fibers oriented to be compatible
with the orientation of $\Sigma := B \setminus \uU_\Gamma \subset B \setminus \Gamma$.
Note that if $B$ is orientable,
then every component of $\Gamma \subset B$ has trivial normal bundle, so
the monodromies of components of $\pi\paper : M\paper \to S^1$ (defined after
choosing a trivialization of $\pi\spine : M\spine \to \Sigma$) can be
taken to be powers of Dehn twists, fixing the boundary of the annulus.
This is not always true if $B$ is non-orientable: in particular, if
$\gamma \subset \Gamma$ is a component whose neighborhood in $B$ is a
M\"obius band, then the monodromy of the corresponding component of 
$\pi\paper : M\paper \to S^1$ interchanges the boundary components of the 
annulus, meaning this component of $M\paper$ has connected boundary,
with multiplicity $2$ (cf.~Definition~\ref{defn:multiplicity}).

\begin{prop}
\label{prop:S1SOBD}
The spinal open book $\boldsymbol{\pi}$ associated to a circle bundle $\pi :
M \to B$ and nonempty multicurve $\Gamma \subset B$ as described above
supports a contact structure that is $S^1$-invariant and partitioned
by~$\Gamma$.  Moreover, any $S^1$-invariant contact structure partitioned
by $\Gamma$ is isotopic to one that is supported by~$\boldsymbol{\pi}$.
\end{prop}
\begin{proof}
It is easy to check that the supported contact structure constructed by
Theorem~\ref{thm:GirouxForms} in this setting is $S^1$-invariant and
partitioned by~$\Gamma$.  In the other direction, suppose $\xi_\Gamma$ is
$S^1$-invariant and partitioned by~$\Gamma$, and choose a tubular
neighborhood $\uU_\Gamma$ of $\Gamma \subset B$ such that the pages of the resulting 
fibration $\pi\paper : M\paper \to S^1$ are tangent to~$\xi_\Gamma$
along $\pi^{-1}(\Gamma)$.  This means that every contact form 
for~$\xi_\Gamma$ has a Reeb vector field transverse to the pages in some
neighborhood of~$\pi^{-1}(\Gamma)$.  On $\pi^{-1}(B \setminus \uU_\Gamma)$, 
$\xi_\Gamma$ is positively transverse to the contact vector field which 
generates a fiber-preserving $S^1$-action, hence one can choose a contact 
form for which this vector field is the Reeb field.  One can now piece
this together with a contact form near $\pi^{-1}(\Gamma)$ whose Reeb field
is transverse to the pages so that the conditions of a Giroux form are
satisfied.
\end{proof}

Since the pages of the natural spinal open book on $(M,\xi_\Gamma)$ have
genus zero, we can immediately apply the $\Omega=0$ case of
Theorem~\ref{thm:nonseparating} and
Corollary~\ref{cor:semifillings} to conclude:

\begin{cor}
\label{cor:S1nonseparating}
For any nonempty multicurve $\Gamma \subset B$ as in Proposition~\ref{prop:Lutz},
the resulting contact circle bundle $(M,\xi_\Gamma)$ admits no non-separating
contact-type embeddings into any closed symplectic $4$-manifold, and it also
admits no strong semifillings with disconnected boundary.
\qed
\end{cor}

It is similarly easy to identify cases in which $(M,\xi_\Gamma)$ has planar
torsion, and is therefore not strongly fillable.  If there is torsion it will 
be of order~$1$, since the pages of the spinal open book $\boldsymbol{\pi}$
supporting $(M,\xi_\Gamma)$ are annuli, but the key question is in which
cases $\boldsymbol{\pi}$ is symmetric.  The vertebrae of $\boldsymbol{\pi}$
are equivalent to the connected components $B_1,\ldots,B_r$ of 
$B \setminus \Gamma$, and the pages come in $S^1$-families corresponding to
the connected components of $\Gamma$.  Given a component $\gamma \subset \Gamma$,
if it bounds the component $B_j \subset B \setminus \Gamma$, then $B_j$ lies
either on one side of $\gamma$ or on both, where the latter is possible only if
$\gamma$ has a non-orientable normal bundle, so $B$ is non-orientable.
Symmetry then means that 
there exist fixed numbers $k_1,\ldots,k_r \in \{0,1,2\}$ such that for each
$j=1,\ldots,r$, every component of $\Gamma$ touches $B_j$ on exactly
$k_j$ sides.  Clearly none of the $k_j$ can be $0$ in this case, since it
would mean there is a component $B_j$ whose closure does not touch
$\Gamma$ at all.  If any $k_j=2$, then it means every component of $\Gamma$
must have that particular component $B_j$ on both sides, hence $r=1$ and
$B$ is not orientable.  In the remaining case, $k_1= \ldots = k_r=1$,
and since at most two components $B_j$ can touch each component of~$\Gamma$,
we conclude $r=2$ and $B$ is orientable with $\Gamma$ splitting it into
two components.  We've proved:

\begin{cor}
\label{cor:S1planarTorsion}
Suppose $\xi_\Gamma$ is an $S^1$-invariant contact structure on a circle
bundle $\pi : M \to B$, partitioned by a nonempty multicurve $\Gamma$, and
that either of the following holds:
\begin{enumerate}[label=(\roman{enumi})]
\item $B \setminus \Gamma$ has at least three connected components;
\item $B \setminus \Gamma$ is disconnected and $B$ is non-orientable.
\end{enumerate}
Then $(M,\xi_\Gamma)$ has (untwisted) planar $1$-torsion, so in particular it 
is not strongly fillable.
\qed
\end{cor}

\begin{remark}
\label{remark:noGirouxTorsion}
If $B$ is oriented with positive genus and $M$ is not a torus bundle, then
\cite{Massot:vanishing}*{Theorem~3} implies that $(M,\xi_\Gamma)$ has
zero Giroux torsion whenever no two connected components of $\Gamma$ are
isotopic.  Using the theorem in \cite{LisiVanhornWendl2} that planar
$1$-torsion implies algebraic $1$-torsion,
one can now extract from Corollary~\ref{cor:S1planarTorsion} 
many new examples of contact manifolds with algebraic $1$-torsion but no 
Giroux torsion.  This generalizes a result for trivial circle bundles that
was proved in \cite{LatschevWendl}.
\end{remark}

\begin{remark}
Corollaries~\ref{cor:S1nonseparating} and~\ref{cor:S1planarTorsion} do not
generally hold for \emph{weak} fillings or semifillings.  Indeed,
\cite{NiederkruegerWendl}*{Theorem~5} implies that whenever every connected
component of the multicurve $\Gamma \subset B$ is nonseparating, 
$(M,\xi_\Gamma)$ admits both a weak filling and a weak semifilling with
disconnected boundary.  We can conclude from Corollary~\ref{cor:semifillings}
and Theorem~\ref{thm:planarTorsion} that the symplectic structures of these
weak fillings must always be nonexact on some spinal component
in~$(M,\xi_\Gamma)$.
\end{remark}

\section{Contact and symplectic structures}
\label{sec:topology}

The main objectives of this section are the proofs of 
Theorems~\ref{thm:GirouxForms} and~\ref{thm:Gompf} about the existence and
uniqueness of supported contact and symplectic structures, plus a related
result about almost Stein structures that will be needed for the classification of
Stein fillings in \cite{LisiVanhornWendl2}.  We begin in
\S\ref{sec:Thurston} with a short collection of ``Thurston-type'' lemmas for
defining contact or symplectic structures on fibrations.  
Section~\ref{sec:coordinates} will then fix some notation for collar 
coordinates and open coverings of spinal open books that will be useful 
throughout the rest of the paper.  Theorem~\ref{thm:GirouxForms} is proved in 
\S\ref{sec:support}, and the proof of Theorem~\ref{thm:Gompf} is carried out
mainly in \S\ref{sec:Gompf}, with the detail about smoothing corners dealt
with in \S\ref{subsec:smoothing}.

\subsection{Several varieties of the Thurston trick}
\label{sec:Thurston}

Since it will be useful in a wide range of contexts, we collect in this
subsection several elementary results that are variations on the main
trick behind Thurston's construction of symplectic forms on total spaces of
symplectic fibrations \cite{Thurston:symplectic}.  All of these results
have higher-dimensional analogues, but with the exception of
Remark~\ref{remark:higherDims}, we will keep things as brief as
possible by focusing on dimensions $3$ and~$4$.

All manifolds in the following will be compact and oriented, and though
it will not yet play a serious role in the discussion, 
they may also have boundary or corners.

\subsubsection{Contact forms}

\begin{prop}
\label{prop:GirouxVertical}
Assume $M$ is a compact oriented $3$-manifold,
$\pi : M \to S^1$ is a submersion, $\sigma$ is a positively oriented
volume form on $S^1$, and $\lambda$ is a $1$-form on $M$ such that
$d\lambda$ is positive on every fiber of~$\pi$.  Then
$\lambda_K := \lambda + K\, \pi^*\sigma$
is a contact form for all $K \gg 0$, and this is true for all $K \ge 0$ if
$\lambda$ is contact.
\end{prop}
\begin{proof}
We have $\pi^*\sigma \wedge d\lambda > 0$ since $d\lambda$ is positive on
fibers, so the result follows by writing
$$
\lambda_K \wedge d\lambda_K = K \left( \pi^*\sigma \wedge d\lambda + 
\frac{1}{K} \lambda \wedge d\lambda \right).
$$
\end{proof}

\begin{prop}
\label{prop:GirouxHorizontal}
Assume $M$ is a compact oriented $3$-manifold,
$\Sigma$ is a compact oriented surface, $\pi : M \to \Sigma$
is a submersion, $\sigma$ is a $1$-form on $\Sigma$
with $d\sigma > 0$, and $\lambda$ is a $1$-form on $M$ that is positive on
every fiber of~$\pi$ and satisfies $d\lambda(v,\cdot)=0$ for all $v \in \ker T\pi$.
Then $\lambda_K := \lambda + K\, \pi^*\sigma$ is a contact form for all
$K \gg 0$, and this is true for all $K \ge 0$ if $\lambda$ is contact.
\end{prop}
\begin{proof}
Since $\pi^*\sigma$ and $d\lambda$ both annihilate vertical vectors, we have
$\pi^*\sigma \wedge d\lambda \equiv 0$, but also $\lambda \wedge \pi^*d\sigma > 0$
due to the condition $d\sigma > 0$ and the positivity of $\lambda$ on fibers.
The result thus follows by writing
$$
\lambda_K \wedge d\lambda_K = K \left( \lambda \wedge \pi^*d\sigma + 
\frac{1}{K} \lambda \wedge d\lambda \right).
$$
\end{proof}

\begin{remark}
\label{remark:higherDims}
If we regard $\sigma$ in Proposition~\ref{prop:GirouxVertical} as a contact
form on $S^1$ and the fibers of $\pi : M \to S^1$ as Liouville domains with
respect to~$\lambda$, then the result has a straightforward generalization
to higher dimensions as a statement about a Liouville fibration over a contact
manifold.  Proposition~\ref{prop:GirouxHorizontal} similarly becomes a
statement about a contact fibration over a Liouville domain $(\Sigma,\sigma)$,
the only subtle point being the condition that $d\lambda$ should annihilate
vertical vectors: if $\pi : M \to \Sigma$ is a $3$-dimensional fibration,
then the secret meaning of this condition is that it reduces the structure
group to the group of \emph{strict} contactomorphisms on~$S^1$, 
i.e.~diffeomorphisms that preserve a fixed contact from and not only a
contact structure.
The natural generalization to higher dimensions can thus be phrased in
terms of strict contact fibrations.
\end{remark}

\subsubsection{Symplectic and Liouville forms}

\begin{prop}
\label{prop:ThurstonSymp}
Assume $E$ is a compact oriented $4$-manifold, $\Sigma$ is a compact oriented surface,
$\Pi : E \to \Sigma$ is a submersion, $\mu$ is a positive area form on $\Sigma$
and $\omega$ is a $2$-form on $E$ that is positive on all fibers of~$\Pi$.  Then
$\omega_K := \omega + K\, \Pi^*\mu$ is symplectic for all $K \gg 0$, and this
is true for all $K \ge 0$ if $\omega$ is symplectic.
\end{prop}
\begin{proof}
The positivity of $\omega$ on fibers implies $\Pi^*\mu \wedge \omega > 0$, so
the result follows by writing
$$
\omega_K \wedge \omega_K = K \left( 2\, \Pi^*\mu \wedge \omega + \frac{1}{K}
\omega \wedge \omega \right).
$$
\end{proof}

If both $\Sigma$ and the fibers of $\Pi : E \to \Sigma$ have nonempty boundary,
then we can also state a version specially for exact symplectic forms; note
that in this case $E$ must be a manifold with boundary \emph{and corners}.

\begin{cor}
\label{cor:ThurstonLiouville}
Assume $E$ is a compact oriented $4$-manifold, $\Sigma$ is a compact oriented surface,
$\Pi : E \to \Sigma$ is a submersion, $\sigma$ is a $1$-form on $\Sigma$
satifying $d\sigma > 0$, and $\lambda$ is a $1$-form on $E$ such that
$d\lambda$ is positive on fibers of~$\Pi$.  Then
$\lambda_K := \lambda + K\, \Pi^*\sigma$ is a Liouville form for all $K \gg 0$,
and this is true for all $K \ge 0$ if $\lambda$ is Liouville.
\qed
\end{cor}

\subsubsection{$J$-convex functions}

Recall that on an almost complex manifold $(W,J)$, a smooth function
$f : W \to \RR$ is called \defin{$J$-convex} (or \defin{plurisubharmonic})
if the $1$-form $\lambda_J := - d f \circ J$ is the primitive of a symplectic
form $d\lambda_J$ that tames~$J$.

\begin{prop}
\label{prop:ThurstonStein}
Assume $(E,J)$ is a compact almost complex $4$-manifold, $(\Sigma,j)$ is a 
compact Riemann surface, $\Pi : (E,J) \to (\Sigma,j)$ is a pseudoholomorphic
submersion, $\varphi : \Sigma \to \RR$ is a $j$-convex function and
$f : E \to \RR$ is a function whose restriction to every fiber of $\Pi$ is
$J$-convex.  Then $f_K := f + K(\varphi \circ \Pi)$ is a $J$-convex function
for all $K \gg 0$, and this is true for all $K \ge 0$ if $f$ is $J$-convex.
\end{prop}
\begin{proof}
The $1$-forms $\sigma := -d\varphi \circ j$ on $\Sigma$ and 
$\lambda = -d f \circ J$ on $E$ satisfy the hypotheses of 
Corollary~\ref{cor:ThurstonLiouville}, and since $T\Pi \circ J = j \circ T\Pi$,
we have
$$
\lambda_K := -d f_K \circ J = - d f \circ J - K\, d(\varphi \circ \Pi) \circ J
= \lambda + K \, \Pi^*\sigma.
$$
Then for any nontrivial $v \in TE$,
$$
d\lambda_K(v,Jv) = K\, \Pi^*d\sigma(v,Jv) + d\lambda(v,Jv) =
K \left[ d\sigma(\Pi_*v,j \Pi_*v) + \frac{1}{K} d\lambda(v,Jv)\right].
$$
This is clearly positive for all $K \ge 0$ if $d\lambda$ tames~$J$; more
generally, the second term will always be positive when $v$ lies in some
neighborhood of the vertical subbundle, and if $v$ is outside of this
neighborhood, then this term will be dominated by $d\sigma(\Pi_*v,j\Pi_*v)$
as long as $K > 0$ is large enough.
\end{proof}

\subsection{Collar neighborhoods and coordinates}
\label{sec:coordinates}

In this section we fix some notation that will be useful throughout the 
rest of the paper.

Fix a compact $3$-manifold $M$ with spinal open book
$$
\boldsymbol{\pi} = \left(\pi\spine : M\spine \to \Sigma , \ \pi\paper :
M\paper \to S^1 , \ \{m_T\}_{T \subset \p M}\right),
$$
and choose an oriented foliation $\fF$ of $\p M$ with closed leaves that
represent the homology classes~$m_T$.  
The circle bundle $\pi\spine : M\spine \to \Sigma$ is necessarily
trivializable, so for convenience we shall fix an identification of $M\spine$
with $\Sigma \times S^1$ such that
$$
\pi\spine : M\spine = \Sigma \times S^1 \to \Sigma
: (z,\theta) \mapsto z.
$$
This defines the coordinate $\theta \in S^1$ globally on~$M\spine$.
The boundary $\p\Sigma$ admits a collar neighborhood 
$$
\nN(\p\Sigma) \subset \Sigma
$$ 
whose connected components can be identified with $(-1,0] \times S^1$, carrying
coordinates $(s,\phi)$.  We shall denote the resulting collar
neighborhood of $\p M\spine$ in $M\spine$ by
$$
\nN(\p M\spine) := \pi\spine^{-1}(\nN(\p\Sigma)) = \nN(\p\Sigma) \times S^1;
$$
its connected components are identified with $(-1,0] \times S^1 \times S^1$
by our chosen trivialization and thus carry coordinates $(s,\phi,\theta)$.

The paper can be identified in turn with a mapping torus 
\begin{equation*}
\begin{split}
M\paper &= (\RR \times P) \Big/ \sim, \qquad 
\text{where $(\tau,p ) \sim (\tau + 1,\mu(p))$}, \\
M\paper &\stackrel{\pi\paper}{\longrightarrow} S^1 = \RR / \ZZ : [(\tau,p)] \to [\tau],
\end{split}
\end{equation*}
where the fiber $P := \pi\paper^{-1}(*)$ is a compact oriented (but not necessarily 
connected) surface with boundary, and the monodromy $\mu : P \to P$ is an
orientation-preserving diffeomorphism that preserves each connected component
of~$P$.  In contrast to the setting of ordinary open books, here we must
allow the possibility that $\mu$ is nontrivial near the boundary,
e.g.~it may permute boundary components.  We can assume without loss of 
generality however that $\p P$ has a collar neighborhood $\nN(\p P) \subset P$
whose connected components have coordinates
$(t,\theta) \in (-1,0] \times S^1$ in which
$\mu(t,\theta) = (t,\theta)$, hence the corresponding collar neighborhood
$$
\nN(\p M\paper) \subset M\paper
$$
of $\p M\paper$ is identified with
$$
\left( \RR \times (-1,0] \times S^1 \times \{1,\ldots,N\} \right)
\Big/ (\tau,t,\theta,i) \sim (\tau+1,t,\theta,\boldsymbol{\sigma}(i))
\subset M\paper
$$
for $N := \# \pi_0(\p P)$ and some permutation $\boldsymbol{\sigma} \in
S_N$.  The connected components of $\nN(\p M\paper)$ are now in one-to-one
correspondence with the invariant subsets of $\{1,\ldots,N\}$
on which the cyclic subgroup $\langle \boldsymbol{\sigma} \rangle
\subset S_N$ generated by $\boldsymbol{\sigma}$ acts transitively.
Given such a subset $I \subset \{1,\ldots,N\}$, if the action of
$\langle \boldsymbol{\sigma} \rangle$ on $I$ has order $m \in \NN$, then
the map $[(\tau,t,\theta,i)] \mapsto ([\tau/m],t,\theta)$ identifies the
corresponding component of $\nN(\p M\paper)$ diffeomorphically with
$S^1 \times (-1,0] \times S^1$.  Denote the resulting coordinates on
components of $\nN(\p M\paper)$ by $(\phi,t,\theta)$.  We now have
$$
\pi\paper(\phi,t,\theta) = m \phi,
$$
where the integer $m \in \NN$ may vary for different components
of $\nN(\p M\paper)$.  These integers are the \defin{multiplicities} 
of $\pi\paper : M\paper \to S^1$ at its boundary components
(cf.~Definition~\ref{defn:multiplicity}).
Note that there is some freedom to change these coordinates on each
component of $\nN(\p M\paper)$ without changing the formula for~$\pi\paper$,
thus we can assume without loss of generality that the chosen oriented 
foliation $\fF$ on $\p M$ with leaves homologous to the preferred meridians
is generated by the flow of the coordinate vector field~$\p_\phi$.

To summarize, we have defined collar neighborhoods of the boundary in
$\Sigma$, $M\spine$ and $M\paper$
whose connected components carry positively oriented coordinates as follows:
\begin{equation*}
\begin{split}
(s,\phi) \in (-1,0] \times S^1 &\subset (-1,0] \times \p \Sigma = 
\nN(\p\Sigma) \subset \Sigma,\\
(s,\phi,\theta) \in (-1,0] \times S^1 \times S^1 &\subset
(-1,0] \times \p M\spine = \nN(\p M\spine) \subset M\spine, \\
(\phi,t,\theta) \in S^1 \times (-1,0] \times S^1 &\subset
(-1,0] \times \p M\paper = \nN(\p M\paper) \subset M\paper.
\end{split}
\end{equation*}
These coordinates satisfy $\pi\spine(s,\phi,\theta) = (s,\phi) \in
\nN(\p \Sigma)$ on $\nN(\p M\spine)$ and
$\pi\paper(\phi,t,\theta) = m\phi \in S^1$ on $\nN(\p M\paper)$, 
where the multiplicity $m \in \NN$ may have different values on distinct 
connected components of~$\nN(\p M\paper)$.  We can also assume without loss of
generality that the coordinate labels
are consistent in the sense that the induced $2$-torus coordinates
$$
(\phi,\theta) \in S^1 \times S^1 \subset \p M\spine
$$
match the corresponding $\phi$- and $\theta$-coordinates defined on 
$\nN(\p M\paper)$ wherever it overlaps $\nN(\p M\spine)$.

To continue, let us add the assumption that $\boldsymbol{\pi}$ admits a
smooth overlap (see Definition~\ref{defn:overlap}).  
We can then introduce a decomposition of $M$ into open subsets
$$
M = \widecheck{M}\spine \cup \widecheck{M}\corner \cup \widecheck{M}\paper \cup 
\widecheck{M}\bndry
$$
defined as follows:
\begin{itemize}
\item $\widecheck{M}\paper$ is the complement of $\{ t \ge -1/2\} \subset 
\nN(\p M\paper)$ in $M\paper$, so the connected components of
$\nN(\p M\paper) \cap \widecheck{M}\paper$ 
inherit coordinates $(\phi,t,\theta) \in S^1 \times (-1,-1/2) \times S^1$.
\item $\widecheck{M}\spine$ is the complement of $\{ s \ge -1/2\} \subset
\nN(\p M\spine)$ in $M\spine$, and we will always use the chosen trivialization
of $\pi\spine$ to identify this with a subset of $\Sigma \times S^1$, denoting
the coordinate on $S^1$ by~$\theta$.  The connected components of
$\nN(\p M\spine) \cap \widecheck{M}\spine$
thus inherit coordinates
$(s,\phi,\theta) \in (-1,-1/2) \times S^1 \times S^1$.
\item $\widecheck{M}\corner$ is the union of $\nN(\p M\spine)$ with the components
of $\nN(\p M\paper)$ that touch~$M\spine$, so its connected components
can each be identified with $(-1,1) \times S^1 \times S^1$, and we assign
coordinates $(\rho,\phi,\theta)$ to these components such that 
$$
\widecheck{M}\corner \cap M\spine = \{ \rho \le 0 \}, \qquad
\widecheck{M}\corner \cap M\paper = \{ \rho \ge 0 \}.
$$
The smooth overlap assumption means we can also assume these coordinates
are related to the previously chosen coordinates on these subsets by 
$\rho = s$ and $\rho = -t$ respectively, with $\theta$ and $\phi$ matching
the existing coordinates on $\nN(\p M\spine)$ and $\nN(\p M\paper)$ 
in the obvious way.  The region $\widecheck{M}\corner$
will be called the \defin{interface} between the
spine and the paper.
\item $\widecheck{M}\bndry$ is the union of the components of $\nN(\p M\paper)$
that touch~$\p M$.  Its connected components therefore carry collar
coordinates $(\phi,t,\theta) \in S^1 \times (-1,0] \times S^1$, but for
consistency with $\widecheck{M}\corner$, we will prefer to use an alternative 
coordinate system
$$
(\rho,\phi,\theta) \in [0,1) \times S^1 \times S^1 \subset
\widecheck{M}\bndry
$$
defined by $\rho := -t$.
\end{itemize}
Note that the above definitions imply
$$
\nN(\p M\paper) \subset \widecheck{M}\corner \cup \widecheck{M}\bndry,
$$
hence we can and sometimes will use the $\rho$-coordinate as an alternative
to the $t$-coordinate on $\nN(\p M\paper)$; they are related by $\rho = -t$.

\subsection{Spinal open books support contact structures}
\label{sec:support}

We will say that a smooth
$1$-form $\alpha$ on~$M$ is a \defin{fiberwise Giroux form} if the
following conditions hold:
\begin{itemize}
\item $d\alpha$ is positive on the interior of every page;
\item $\alpha$ is positive on the fibers of $\pi\spine : M\spine \to \Sigma$,
and the tangent spaces to these fibers are contained in $\ker d\alpha$;
\item At $\p M$, $\alpha$ is positive on all boundaries of pages, the
tangent spaces to these boundaries are also contained in $\ker d\alpha$,
and $\alpha$ vanishes on the foliation~$\fF$ chosen at the beginning 
of~\S\ref{sec:coordinates}.
\end{itemize}
A fiberwise Giroux form is a Giroux form if and only if it is contact,
but since we have not required the latter in the above definition,
the space of fiberwise Giroux forms is \emph{convex}.  We will show in the
following that it is relatively easy to construct fiberwise Giroux forms, and
the main idea in the proof of Theorem~\ref{thm:GirouxForms} is---following
the ideas of Thurston outlined in \S\ref{sec:Thurston}---to turn these into Giroux forms by adding large
multiples of certain $1$-forms pulled back from the bases of the fibrations.

Observe that since every component of $\Sigma$ has nonempty
boundary, we can choose a $1$-form $\sigma$ on $\Sigma$ satisfying
$$
d\sigma > 0 \text{ on $\Sigma$}, \qquad
\sigma = e^s \, d\phi \text{ on $\nN(\p\Sigma)$}.
$$
Similarly:

\begin{lemma}
\label{lemma:fiberLiouville}
On $M\paper$ there exists a $1$-form $\eta$ such that $d\eta$ is positive on
each fiber of $\pi\paper : M\paper \to S^1$ and
$\eta = e^t \, d\theta$ in $\nN(\p M\paper)$.
\end{lemma}
\begin{proof}
One only has to observe that since every connected component 
of~$P := \pi\paper^{-1}(*)$ has
nonempty boundary by assumption, the space of Liouville forms on $P$
which match $e^t\, d\theta$ in the collars is nonempty and convex.  The desired
$1$-form $\eta$ can thus be constructed by choosing such a Liouville form
$\eta_0$ and defining $\eta$ on each fiber of the
mapping torus with monodromy $\mu$ as a suitable interpolation between 
$\eta_0$ and $\mu^*\eta_0$;
cf.~\cite{Etnyre:lectures}*{Theorem~3.13}.
\end{proof}

In order to construct a fiberwise Giroux form, we next choose a smooth
function 
$$
F : M\paper \to (0,1]
$$
which is identically equal to~$1$ outside of $\nN(\p M\paper)$
and takes the form $e^\rho f(\rho)$ in $(\rho,\phi,\theta)$-coordinates
on $\nN(\p M\paper) \subset \widecheck{M}\corner \cup \widecheck{M}\bndry$,
where $f : (-1,1) \to (0,1]$ is a smooth function satisfying the conditions
\begin{itemize}
\item $f(\rho) = 1$ for $\rho \le 0$;
\item $f'(\rho) < 0$ for $\rho > 0$;
\item $f(\rho) = e^{-\rho}$ for $\rho$ near~$1$.
\end{itemize}
In particular, this implies that $F$ admits a smooth extension over
$\nN(\p M\spine)$ of the form $F(s,\phi,\theta) = e^s$.
Now using the fiberwise Liouville form $\eta$ provided by 
Lemma~\ref{lemma:fiberLiouville},
we can define a fiberwise Giroux form on~$M$ by
$$
\alpha = 
\begin{cases}
d\theta & \text{ on $M\spine$}, \\
F\eta & \text{ on $M\paper$}.
\end{cases}
$$
It takes the form $\alpha = f(\rho)\, d\theta$ on
$\widecheck{M}\corner \cup \widecheck{M}\bndry$.

We show next how to turn fiberwise Giroux forms into Giroux forms.
For any constant $\delta \in (0,1/2)$, choose
a pair of smooth functions $g\spine^\delta, g\bndry^\delta : [0,1) \to [0,2]$ 
such that
\begin{itemize}
\item $g\spine^\delta(\rho) = e^\rho$ for $\rho$ near~$0$;
\item $g\bndry^\delta(0) = 0$ and $(g\bndry^\delta)'(0) > 0$;
\item $(g\spine^\delta)'(\rho)$ and $(g\bndry^\delta)'(\rho)$ are both nonnegative for all~$\rho$;
\item $g\spine^\delta(\rho) = g\bndry^\delta(\rho) = 2$ for all $\rho \ge \delta$.
\end{itemize}
Using this, we define a smooth function $G_\delta : M\paper \to [0,2]$ by
$$
G_\delta = 
\begin{cases}
2 & \text{ on $\widecheck{M}\paper$},\\
g\spine^\delta(\rho) & \text{ on $\nN(\p M\paper) \cap \widecheck{M}\corner$},\\
g\bndry^\delta(\rho) & \text{ on $\widecheck{M}\bndry$}.
\end{cases}
$$
Then, defining the Liouville form $\sigma$ as a $1$-form on $M\spine$ by
identifying it with its pullback $\pi\spine^*\sigma$, we define for any
$\delta \in (0,1/2)$ another smooth $1$-form on~$M$ by
$$
\beta_\delta = \begin{cases}
\sigma & \text{ on $M\spine$},\\
G_\delta\, d\phi & \text{ on $M\paper$}.
\end{cases}
$$

\begin{lemma}
\label{lemma:fiberwiseGiroux}
For any fiberwise Giroux form $\alpha$, there exist constants
$\delta_0 \in (0,1/2)$ and $K_0 \ge 0$ such that for all constants 
$\delta \in (0,\delta_0]$ and $K \ge K_0$,
$$
\alpha_{K,\delta} := \alpha + K \beta_\delta
$$
is a Giroux form.  Moreover, whenever $\alpha$ itself is a Giroux form,
one can take $K_0 = 0$.
\end{lemma}
\begin{proof}
Observe that $\alpha_{K,\delta}$ is automatically a \emph{fiberwise} Giroux
form for all $K \ge 0$, $\delta \in (0,1/2)$, so we only need to show that
$\alpha_{K,\delta}$ is contact for the right choices of these constants.
Since $\beta_\delta \wedge d\beta_\delta \equiv 0$, we have
$$
\alpha_{K,\delta} \wedge d\alpha_{K,\delta} = K \left( \alpha \wedge d\beta_\delta +
\beta_\delta \wedge d\alpha\right) + \alpha \wedge d\alpha,
$$
thus it suffices to show that whenever $\delta > 0$ is sufficiently small,
\begin{equation}
\label{eqn:isContact}
\alpha \wedge d\beta_\delta + \beta_\delta \wedge d\alpha > 0.
\end{equation}
The conditions on fiberwise Giroux forms imply that $\alpha(\p_\theta) > 0$
at $\p M\paper$, so this is also true on collars of the form
$\{ \rho \le \delta_0 \} \subset \nN(\p M\paper)$ for sufficiently 
small $\delta_0 > 0$.  Assuming $0 < \delta \le \delta_0$, we shall now
show that \eqref{eqn:isContact} holds everywhere on~$M$.

On $M\spine$, $\beta_\delta \wedge d\alpha = \sigma \wedge d\alpha = 0$
since $\sigma(\p_\theta) = d\alpha(\p_\theta,\cdot) = 0$, but
$\alpha \wedge d\beta_\delta > 0$ since $\alpha(\p_\theta) > 0$ and
$d\beta_\delta = d\sigma$ is positive on~$\Sigma$.

On $M\paper$ outside of the collars $\{ \rho \le \delta\}$, we have
$\beta_\delta = 2\, d\phi$ and thus $d\beta_\delta = 0$, while
$\beta_\delta \wedge d\alpha = 2\, d\phi \wedge d\alpha > 0$ due to the
assumption that $d\alpha$ is positive on the fibers of~$\pi\paper$.

On the collars $\{ \rho \le \delta\}$, we have $\beta_\delta = G_\delta\, d\phi$,
with $G_\delta > 0$ on the interior of $M\paper$, hence
$\beta_\delta \wedge d\alpha = G_\delta\, d\phi \wedge d\alpha > 0$ again
except at $\p M\paper$.  It thus remains only to show that 
$\alpha \wedge d\beta_\delta \ge 0$, with strict positivity at $\p M\paper$.
This follows from the fact that $\alpha(\p_\theta) > 0$ on this region,
since $\alpha \wedge d\beta_\delta = g'(\rho)\, \alpha \wedge
d\rho \wedge d\phi$, where $g(\rho)$ denotes either $g\spine^\delta(\rho)$ or 
$g\bndry^\delta(\rho)$,
both of which we assumed to have nonnegative first derivatives which are
strictly positive at $\rho=0$.
\end{proof}

\begin{proof}[Proof of Theorem~\ref{thm:GirouxForms}]
In light of the construction of a fiberwise Giroux form explained above,
the existence of a Giroux form follows immediately from 
Lemma~\ref{lemma:fiberwiseGiroux}.

We claim now that for any $n \in \NN$, a continuous family of 
Giroux forms
$$
\{ \alpha_\tau \}_{\tau \in S^{n-1}}
$$
can always be extended to a family of Giroux forms 
parametrized by the disk $\DD^n$.
As an initial step, note that the characteristic foliations induced
by $\alpha_\tau$ at $\p M$ may not be precisely the foliation
$\fF$ we fixed above, but they are guaranteed to be isotopic to it and
also transverse to the coordinate vector field $\p_\theta$ (which is
parallel to Reeb orbits at the boundary).  We can thus alter
$\alpha_\tau$ by a fiber preserving isotopy supported near $\p M$, producing
a homotopy through $S^{n-1}$-families of Giroux forms, to a family whose
characteristic foliations at $\p M$ are all generated by~$\p_\phi$.
Let us therefore assume without loss of generality that the given family
$\alpha_\tau$ has this property, so all the $\alpha_\tau$ are also fiberwise
Giroux forms by our definition.

Since the space of fiberwise Giroux forms is convex, 
$\{\alpha_\tau\}_{\tau \in S^{n-1}}$ can now be extended via
linear interpolation to a family $\{ \tilde{\alpha} \}_{\tau \in \DD^n}$ of 
fiberwise Giroux forms.  These forms are also contact for all $\tau$ in some
collar neighborhood of $\p\DD^n$, since the contact condition is open.
Choose a continuous ``bump'' function
$$
\psi : \DD^n \to [0,1]
$$
that equals~$0$ at $\p \DD^n$ and~$1$ outside this collar.  Next, observe that
since $\DD^n$ is compact, one can find constants $K \ge 0$ sufficiently 
large and $\delta > 0$ sufficiently small so that 
Lemma~\ref{lemma:fiberwiseGiroux} holds with the same constants for all
$\tilde{\alpha}_\tau$, $\tau \in \DD^n$.  Then 
$$
\alpha_\tau := \tilde{\alpha}_\tau + K \psi(\tau) \beta_{\delta}
$$
defines the desired family of Giroux forms.  This shows that the space of
Giroux forms has vanishing homotopy groups of all orders, so by
Whitehead's theorem, it is contractible.
\end{proof}

We can now fill in a loose end from \S\ref{sec:invariantsSketch} and complete
the proof of Proposition~\ref{prop:overtwisted}.

\begin{lemma}
\label{lemma:overtwisted}
Every contact manifold with planar $1$-torsion is overtwisted.
\end{lemma}
\begin{proof}
Suppose $(M, \xi)$ contains a planar $0$-torsion domain~$M_0$,
so $(M_0,\xi)$ is supported by a spinal open book $\boldsymbol{\pi}$
whose interior contains a page $D$ that is a disk.  Let
$M\spine^1 \subset M_0$ denote the spinal region adjacent to~$D$.
Since $\boldsymbol{\pi}$ is not symmetric, there is a paper component 
$M\paper^1 \subset M_0$ adjacent to $M\spine^1$ with a page $P_1 \subset M\paper^1$
that is not a disk.  Pick an embedded curve $L$ in the interior of $P_1$
that is smoothly isotopic to a boundary component adjacent to $M\spine^1$,
so $L$ is also smoothly isotopic to the boundary of $D$, 
and both page framings agree and are equal to~$0$.  We will show that 
one can realize $L$ as a Legendrian knot with Thurston-Bennequin
number $\tb(L)=0$, violating the Bennequin-Eliashberg bound if
$(M,\xi)$ is tight.

We consider two cases.  First, assume that $P_1$ has another 
boundary component. Then in the construction of the Giroux form,
the $1$-form $\eta$ of Lemma~\ref{lemma:fiberLiouville}
can be chosen to vanish identically along~$L$.  Choosing the support of 
the monodromy away from $L$, it remains Legendrian after converting 
$\eta$ into the compatible contact form $\alpha$, and the contact 
framing is $0$ relative to $P_1$, and hence also relative to the disk~$D$.

Alternatively, suppose $P_1$ has a single boundary component (to which
$L$ is isotopic), and since $\boldsymbol{\pi}$ is not symmetric, $P_1$
has genus $g > 0$.  Since $P_1$ is a convex surface (but with transverse boundary), 
we can flow it along a transverse contact vector field to create a
neighborhood of the form $P_1 \times [0,1] \subset M\paper^1$ and round the corners to produce a
convex handlebody whose dividing set is isotopic to 
$\p P_1 \times \{1/2\}$.  On this convex surface, $L$ is isolating 
(in the sense of Honda \cite{Honda:tight1}), but since $g > 0$, we can fold along any other 
(disjoint, homotopically nontrivial, embedded) curve in $P_1 \times \{1\}$, 
increasing the dividing set and making $L$ non-isolating. 
We can now Legendrian realize $L$, and since $L$ is disjoint from 
the dividing set, we can ensure this has contact framing $0$ relative to 
$P_1 \times \{1\}$.  This is an absolute $0$-framing since
the framings from $P_1\times \{1\}$, $P_1$ and the disk $D$ all agree.
\end{proof}

\subsection{Lefschetz fibrations and symplectic structures}
\label{sec:Gompf}

In this section we prove the main part of Theorem~\ref{thm:Gompf} regarding
the various spaces of
symplectic structures supported by a bordered Lefschetz fibration.  
The overall strategy
is similar to that of the previous section, and can be summarized as follows:
\begin{enumerate}
\item Define spaces of ``fiberwise'' symplectic structures which are 
manifestly contractible, and are nonempty under suitable 
assumptions.
\item Use the Thurston trick to turn fiberwise structures into
supported symplectic structures by adding large multiples of data pulled
back from the base.
\end{enumerate}
A version of Theorem~\ref{thm:Gompf} for almost Stein structures appeared
already in our appendix to \cite{BaykurVanhorn:large}, and we will repeat
some of those arguments here but will generalize them substantially
in \S\ref{sec:SteinHomotopy} below, with an eye toward classifying fillings
up to Stein homotopy.

For this subsection and the next, fix a bordered Lefschetz fibration 
$\Pi : E \to \Sigma$.  Recall that
a symplectic structure $\omega$ on~$E$ was defined to be \emph{supported}
by $\Pi$ if it is positive on fibers and also tames some almost complex
structure $J$ defined near the critical points $E\crit$ for which the
fibers are $J$-holomorphic.  It will be useful to note that this last
condition doesn't depend on the choice of~$J$:

\begin{prop}
\label{prop:J1J2}
Suppose $J_1$ and $J_2$ are two almost complex 
structures defined near $E\crit$ which each restrict to positively oriented
complex structures on the smooth part of every fiber.  
Then $J_1|_{T E\crit} = J_2|_{T E\crit}$.
\end{prop}
\begin{proof}
By \cite{Gompf:hyperpencils}*{Lemma~4.4(a)}, it will suffice to observe
that $J_1$ and $J_2$ determine the
same oriented complex $1$-dimensional subspaces in $TE|_{E\crit}$.  
Indeed, choosing local
complex coordinates $(z_1,z_2)$ near a point $p \in E\crit$ and a 
corresponding complex coordinate near $\Pi(p)$ such that 
$\Pi(z_1,z_2) = z_1^2 + z_2^2$, we see that in these coordinates every complex
$1$-dimensional subspace of $\CC^2$ occurs as a tangent space to a fiber
in any neighborhood of~$p$. Since such tangent spaces
are both $J_1$- and $J_2$-complex by assumption, the claim follows
by continuity.
\end{proof}

In the following, fix an integrable complex structure $J\crit$
near $E\crit$ for which
$\Pi$ is holomorphic near~$E\crit$.  Proposition~\ref{prop:J1J2} implies
that none of our definitions or results will depend on this choice.
We shall now define various spaces of smooth objects on $E$, each
assumed to carry the natural $C^\infty$-topology.
Denote the vertical subbundles in $E$ and $\p_h E$ by
$$
VE = \ker T\Pi \subset TE \quad \text{ and } \quad
V(\p_h E) = VE \cap T(\p_h E).
$$

\begin{defn}
\label{defn:fiberwiseGiroux}
Let $\GirouxFib(\p_h \Pi)$ denote the space of germs of $1$-forms 
$\lambda$ defined on a neighborhood of $\p_h E$ in~$E$ such that 
$\lambda|_{V(\p_h E)} > 0$ and 
$V(\p_h E) \subset \ker \left( d\lambda|_{T(\p_h E)} \right)$.
Similarly, $\GirouxFib(\p\Pi)$ will denote the space of germs of
$1$-forms $\lambda$ defined on a neighborhood of $\p E$ in~$E$ which 
satisfy the above conditions at $\p_h E$ and also
satisfy $d\lambda|_{VE} > 0$ at~$\p_v E$.
We call any $\lambda \in \GirouxFib(\p_h\Pi)$ or $\GirouxFib(\p\Pi)$
a \defin{fiberwise Giroux form near} $\p_h E$ or $\p E$ respectively.
\end{defn}

Observe that $\GirouxFib(\p_h \Pi)$ and $\GirouxFib(\p\Pi)$ are both
convex spaces.

\begin{defn}
\label{defn:GirouxNear}
The spaces of \defin{Giroux forms near} $\p_h E$ or
$\p E$ respectively (cf.~Remark~\ref{remark:corner}) are defined as
\begin{equation*}
\begin{split}
\Giroux(\p_h\Pi) &:= \left\{ \lambda \in \GirouxFib(\p_h\Pi)\ \Big|\ 
\text{$\lambda|_{T(\p_h E)}$ is contact} \right\}, \\
\Giroux(\p\Pi) &:= \left\{ \lambda \in \GirouxFib(\p\Pi)\ \Big|\ 
\text{$\lambda|_{T(\p_h E)}$ and $\lambda|_{T(\p_v E)}$ are both contact} \right\}.
\end{split}
\end{equation*}
\end{defn}

The following variation on Theorem~\ref{thm:GirouxForms} follows from a
simpler version of the same argument, implementing the Thurston trick
via Propositions~\ref{prop:GirouxVertical} and~\ref{prop:GirouxHorizontal}.
It implies in particular that both
$\Giroux(\p_h \Pi)$ and $\Giroux(\p\Pi)$ are nonempty and contractible.

\begin{prop}
\label{prop:GirouxNearBoundary}
The spaces $\GirouxFib(\p_h\Pi)$ and $\GirouxFib(\p\Pi)$ are each
nonempty.  Moreover, fixing a Liouville form $\sigma$ on $\Sigma$,
for any $\lambda \in \GirouxFib(\p_h\Pi)$
or $\GirouxFib(\p\Pi)$, there exists a constant $K_0 \ge 0$ depending
continuously on~$\lambda$ such that for every constant $K \ge K_0$,
$\lambda + K\, \Pi^*\sigma$ belongs to $\Giroux(\p_h\Pi)$ or
$\Giroux(\p\Pi)$ respectively, and we can take $K_0 = 0$ if
$\lambda$ is already in $\Giroux(\p_h\Pi)$ or $\Giroux(\p\Pi)$. \qed
\end{prop}

\begin{defn}
\label{defn:fiberwise}
The space of \defin{weakly convex fiberwise symplectic} structures
$\weakFib(\Pi)$ consists of all smooth closed $2$-forms $\omega$ on~$E$ such that
\begin{enumerate}
\item $\omega$ is positive on all fibers in $E \setminus E\crit$;
\item At $E\crit$, $\omega$ is nondegenerate and tames $J\crit$;
\item Near $\p_h E$, $\omega = d\lambda$ for some 
$\lambda \in \GirouxFib(\p_h\Pi)$.
\end{enumerate}
The space of supported weakly convex symplectic structures is then
\begin{equation*}
\begin{split}
\weak(\Pi) := \big\{ \omega \in \weakFib(\Pi)\ |\ 
&\text{$\omega^2 > 0$ and $\omega = d\lambda$ near $\p_h E$} \\
&\text{for some $\lambda \in \Giroux(\p_h\Pi)$} \big\}.
\end{split}
\end{equation*}
The space of \defin{strongly convex fiberwise symplectic} structures will be
$$
\strongFib(\Pi) := \left\{ \omega \in \weakFib(\Pi)\ \Big|\ 
\text{$\omega = d\lambda$ near $\p E$ for some $\lambda \in \GirouxFib(\p\Pi)$} 
\right\},
$$
so that the space of supported strongly convex symplectic structures is
\begin{equation*}
\begin{split}
\strong(\Pi) := \big\{ \omega \in \strongFib(\Pi)\ |\ 
&\text{$\omega^2 > 0$ and $\omega = d\lambda$ near $\p E$} \\
&\text{for some $\lambda \in \Giroux(\p\Pi)$} \big\}.
\end{split}
\end{equation*}
We similarly define the space of \defin{fiberwise Liouville structures}
$\LiouvilleFib(\Pi)$ to consist of all $\omega \in \strongFib(\Pi)$ for
which the primitive $\lambda \in \GirouxFib(\p\Pi)$ extends to a global
primitive of $\omega$ (i.e.~a \defin{fiberwise Liouville form}) on~$E$.  
The space of supported Liouville structures is then
\begin{equation*}
\begin{split}
\Liouville(\Pi) := \big\{ \omega \in \LiouvilleFib(\Pi)\ |\ 
&\text{$\omega^2 > 0$ and $\omega = d\lambda$ on $E$} \\
&\text{for some $\lambda$ with $\lambda|_{\p E} \in \Giroux(\p\Pi)$} \big\}.
\end{split}
\end{equation*}
\end{defn}

Observe that there are natural inclusions
$$
\LiouvilleFib(\Pi) \hookrightarrow \strongFib(\Pi) \hookrightarrow \weakFib(\Pi),
$$
and all three spaces are convex.

To handle the Stein case, we shall consider a special space of almost
complex structures.  Given any almost complex structure $J$ on $E$,
denote the maximal $J$-complex subbundle in $T(\p_h E)$ by
$$
\xi_J := T(\p_h E) \cap J T(\p_h E) \subset T(\p_h E).
$$

\begin{defn}
\label{defn:AC}
Let $\AC(\Pi)$ denote the space of pairs $(J,\p_\theta)$ where $J$ is an
almost complex structure on $E$ compatible with its orientation,
$\p_\theta$ is a nowhere zero vertical vector field on $\p_h E$, oriented
in the positive direction of the fibers,
and the following properties are satisfied:
\begin{enumerate}
\item
There exists a complex structure $j$ on $\Sigma$ for which
$\Pi : (E,J) \to (\Sigma,j)$ is pseudoholomorphic;
\item
The flow of $\p_\theta$ is $1$-periodic and preserves $\xi_J$.
\end{enumerate}
\end{defn}

Note that any $(J,\p_\theta) \in \AC(\Pi)$
uniquely determines $j$ on~$\Sigma$. The choice of vector field $\p_\theta$ 
is equivalent to a choice of principal $S^1$-bundle structure on~$\p_h E$, so it 
defines a fiber-preserving $S^1$-action that preserves both $\xi_J$ and
(due to the first condition)~$J|_{\xi_J}$.  We do not require $J$ to match
$J\crit$ near~$E\crit$, though they automatically match at $E\crit$ due to
Proposition~\ref{prop:J1J2}.
Using the fact that the space of positively oriented complex structures 
on any oriented real vector bundle of rank~$2$ is nonempty and contractible, 
it follows that the same is true for~$\AC(\Pi)$.

\begin{remark}
\label{remark:S1action}
In contact geometric terms, defining a principal $S^1$-bundle structure on 
$\p_h E$ is equivalent to giving it the structure of a \emph{strict}
contact fiber bundle (cf.~Remark~\ref{remark:higherDims}), i.e.~each fiber
is identified with the contact manifold $(S^1,dt)$ so that the vector field 
$\p_\theta$ generating the $S^1$-action satisfies $dt(\p_\theta) \equiv 1$.
Any fiberwise Giroux form $\lambda \in \GirouxFib(\p_h \Pi)$ near $\p_h E$
canonically determines a strict contact fiber bundle structure, with a
positive constant multiple of $\lambda$ as the contact form on each fiber;
here the condition $d\lambda(\p_\theta,\cdot)|_{T(\p_h E)} \equiv 0$ ensures
that all fibers are strictly contactomorphic since $\lambda$ has the same
integral on all of them, by Stokes' theorem.
\end{remark}

\begin{defn}
\label{defn:fiberwiseJconvex}
Given any $(J,\p_\theta) \in \AC(\Pi)$, we will say that a smooth 
function $f : E \to \RR$ 
is \defin{fiberwise $J$-convex} if, writing $\lambda_J := -df \circ J$,
the following conditions are satisfied:
\begin{enumerate}
\item $f$ is constant on each boundary component of each fiber 
$E_z \subset E$;
\item $d\lambda_J \in \LiouvilleFib(\Pi)$;
\item $\lambda_J|_{\p E} \in \GirouxFib(\p\Pi)$;
\item $\lambda_J(\p_\theta)$ is constant.
\end{enumerate}
The space of fiberwise $J$-convex functions for a fixed 
$(J,\p_\theta) \in \AC(\Pi)$ will be denoted by
$\JconvexFib_{(J,\p_\theta)}(\Pi)$.
\end{defn}

Observe that $\JconvexFib_{(J,\p_\theta)}(\Pi)$ is convex for each
$(J,\p_\theta) \in \AC(\Pi)$, and 
there is a natural map
$$
\JconvexFib_{(J,\p_\theta)}(\Pi) \to \LiouvilleFib(\Pi) : f \mapsto
-d(d f \circ J).
$$ 

\begin{defn}
\label{defn:JconvexPi}
Let $\Jconvex_{(J,\p_\theta)}(\Pi) \subset \JconvexFib_{(J,\p_\theta)}(\Pi)$
denote the subspace for which $d\lambda_J$ is also a symplectic form
taming $J$ and $\lambda_J|_{\p E} \in \Giroux(\p\Pi)$.
\end{defn}

The space of supported almost Stein structures is now precisely
$$
\ACStein(\Pi) = \{ (J,f) \ |\ \text{$(J,\p_\theta) \in \AC(\Pi)$ for
some $\p_\theta$, and
$f \in \Jconvex_{(J,\p_\theta)}(\Pi)$} \}.
$$
Note that for any $(J,f) \in \ACStein(\Pi)$, the vector field $\p_\theta$
is canonically determined via Remark~\ref{remark:S1action}, hence there is
a well-defined projection
\begin{equation}
\label{eqn:SteinProj}
\ACStein(\Pi) \to \AC(\Pi) : (J,f) \mapsto (J,\p_\theta),
\end{equation}
whose fiber over any $(J,\p_\theta) \in \AC(\Pi)$ is
$\Jconvex_{(J,\p_\theta)}(\Pi)$.
Since $\AC(\Pi)$ is homotopy equivalent to a point,
the almost Stein part of Theorem~\ref{thm:Gompf} will then be a consequence
of the following statement, to be proved at the very end of this
subsection:

\begin{prop}
\label{prop:almostSteinCtrbl}
If $\Pi : E \to \Sigma$ is allowable, then
the projection \eqref{eqn:SteinProj} is a Serre fibration with contractible
fibers; in particular, it is a homotopy equivalence.
\end{prop}

\begin{remark}
\label{remark:notConvex}
The contractibility of $\Jconvex_{(J,\p_\theta)}(\Pi)$ for each
$(J,\p_\theta) \in \AC(\Pi)$ is not as obvious as it may at first appear,
e.g.~since functions in $\Jconvex_{(J,\p_\theta)}(\Pi)$ are not constant
at the boundary, $\Jconvex_{(J,\p_\theta)}(\Pi)$ is not generally convex
(cf.~the discussion of almost Stein structures preceding 
Definition~\ref{defn:almostSteinFibration}).  The proof that
$\Jconvex_{(J,\p_\theta)}(\Pi)$ is contractible will instead require the
Thurston trick.
\end{remark}

We will frequently need to use the following standard lemma in constructions
of $J$-convex functions.  Recall that a hypersurface $V$ in an almost 
complex manifold $(W,J)$ is called
\defin{$J$-convex} whenever the maximal $J$-complex subbundle in $TV$ is a
contact structure whose canonical conformal symplectic structure tames~$J$.

\begin{lemma}[see e.g.~\cite{CieliebakEliashberg}*{Lemma~2.7} or \cite{LatschevWendl}*{Lemma~4.1}]
\label{lemma:sufficientlyConvex}
Suppose $(W,J)$ is a smooth almost complex manifold and $f : W \to \RR$ is a
smooth function such that $f$ is $J$-convex near all its critical points
and all level sets of~$f$ are $J$-convex hypersurfaces wherever they are regular.  
Then if $h : \RR \to \RR$ is any smooth function with $h' > 0$ and $h''$ 
everywhere sufficiently large, $h \circ f$ is a $J$-convex function. \qed
\end{lemma} 
\begin{remark}
\label{remark:dimension2}
It will sometimes be useful to note that the $J$-convexity hypothesis on
hypersurfaces is vacuous when $\dim_\RR W = 2$.
\end{remark}

In order to construct fiberwise symplectic structures in the nonexact case, 
we will need first
to be able to pick a cohomology class that evaluates positively on every
irreducible component of every fiber.  For this we will make use of the 
following linear algebraic lemma due to Gompf.

\begin{lemma}[\cite{Gompf:locallyHolomorphic}*{Lemma~3.3}]
\label{lemma:GompfMatrix}
For a real $n$-by-$n$ symmetric matrix $A = ( a_{ij} )$, let $G_A$
denote the graph with $n$ vertices $v_1,\ldots,v_n$, and an edge
between any two distinct vertices $v_i, v_j$ whenever $a_{ij} \ne 0$.
Suppose that (a)~$G_A$ is connected, (b)~$a_{ij} \ge 0$ whenever
$i \ne j$, and (c)~there are positive real numbers $m_1,\ldots,m_n$
such that $\sum_{i=1}^m m_i a_{ij} \le 0$ for all~$j$.  Fix a choice
of such numbers~$m_i$.  Then the hypothesis (d), that the inequality
in~(c) is strict for some~$j$, implies $\rank A = n$.  If (d) is not
satisfied, then $\rank A = n-1$.
\qed
\end{lemma}

\begin{lemma}
\label{lemma:positive}
Given any $\lambda \in \GirouxFib(\p_h\Pi)$ or $\GirouxFib(\p\Pi)$,
there exists a closed
$2$-form $\eta$ on~$E$ such that $\eta = d\lambda$ near $\p_h E$ or $\p E$
respectively
and $\int_C \eta > 0$ for every irreducible component~$C$ of every fiber.
\end{lemma}
\begin{proof}
Extending $\lambda$ arbitrarily to a smooth $1$-form on~$E$, Stokes'
theorem implies $\int_C d\lambda \ge 0$ for all irreducible components
$C$ of fibers, with strict inequality if and only if $\p C \ne \emptyset$.
Our main task will be to find a closed $2$-form $\omega$ supported in the
interior such that
$\int_C \omega > 0$ for every component 
$C$ with $\p C = \emptyset$, as we can then
set $\eta := d\lambda + \epsilon \omega$ for sufficiently small $\epsilon > 0$.

We construct $\omega$ as follows.  The collection of all closed irreducible
components of singular fibers defines a graph $\Gamma$, with vertices 
corresponding to closed irreducible components and edges corresponding to
critical points at which two such components intersect each other.  Pick any
connected component of $\Gamma$ and denote the corresponding closed irreducible
components of fibers by $C_1,\ldots,C_n$.  Pick closed $2$-forms
$\omega_1,\ldots,\omega_n$ such that for $i = 1,\ldots,n$, $\omega_i$
represents the Poincar\'e dual of $[C_i]$ and 
is supported in a neighborhood of $C_i$ disjoint from $\p E$.
Let $\hat{n}_i$ denote the number of critical points at which $C_i$
intersects other irreducible components (i.e.~not counting intersections of
$C_i$ with itself), and let $n_i \le \hat{n}_i$ denote the number of these
at which $C_i$ intersects other \emph{closed} components (this is the
number of edges touching the corresponding vertex in~$\Gamma$).
For $i,j \in 1,\ldots,n$, let $n_{ij}$ denote the number of critical points
at which $C_i$ and $C_j$ intersect, i.e.~the number of edges of $\Gamma$
connecting the two corresponding vertices.  The algebraic 
intersections numbers $[C_i] \cdot [C_j] \in \ZZ$ then satisfy
\begin{equation*}
\begin{split}
[C_i] \cdot [C_j] &= n_{ij} \text{ for $i \ne j$},\\
[C_i] \cdot [C_i] &= -\hat{n}_i,
\end{split}
\end{equation*}
thus for $i=1,\ldots,n$,
\begin{equation}
\label{eqn:whocares}
\sum_{j=1}^n [C_i] \cdot [C_j] = \sum_{j \ne i} n_{ij} - \hat{n}_i =
n_i - \hat{n}_i \le 0.
\end{equation}
Since no fiber consists exclusively of closed components, the
inequality $n_i - \hat{n}_i \le 0$ must be strict for some $i=1,\ldots,n$.

Define now an $n$-by-$n$ symmetric
matrix $A = ( a_{ij} )$ with entries $a_{ij} = [C_i] \cdot [C_j]$.
By \eqref{eqn:whocares}, $A$ satisfies the conditions of
Lemma~\ref{lemma:GompfMatrix} with $m_1 = \ldots = m_n = 1$, including
hypothesis~(d), hence $\rank A = n$.  It follows that one can find
coefficients $b_1,\ldots,b_n \in \RR$ such that
$$
\int_{C_i} \sum_{j=1}^n b_j \omega_j = \sum_{j=1}^n a_{ij} b_j > 0
$$
for all $i = 1,\ldots,n$.  The desired $2$-form $\omega$ can thus be defined
as a sum of $2$-forms of this type for each connected component of the
graph~$\Gamma$.
\end{proof}

The next proposition is the main existence result for fiberwise symplectic
structures.

\begin{prop}
\label{prop:nonempty}
Given any $\lambda \in \GirouxFib(\p_h \Pi)$ or $\GirouxFib(\p\Pi)$, there
exists $\omega \in \strongFib(\Pi)$ such that $\omega = d\lambda$ near
$\p_h E$ or $\p E$ respectively.
In particular, the spaces $\weakFib(\Pi)$ and $\strongFib(\Pi)$ are 
always nonempty.
Moreover, for any $(J,\p_\theta) \in \AC(\Pi)$,
$\JconvexFib_{(J,\p_\theta)}(\Pi)$ (and hence also $\LiouvilleFib(\Pi)$)
is nonempty if and only if $\Pi$ is allowable.
\end{prop}
\begin{proof}
If $\Pi$ is not allowable, then there is a closed component in some singular
fiber, thus Stokes' theorem implies there can be no exact $2$-form
that is positive on every fiber.  Consequently,
$\LiouvilleFib(\Pi)$ 
(and therefore also $\JconvexFib_{(J,\p_\theta)}(\Pi)$) must be empty.

In the following, we shall handle the construction of 
$\omega \in \strongFib(\Pi)$ and $f \in \JconvexFib_{(J,\p_\theta)}(\Pi)$ 
in parallel at each step.  The construction of
$f \in \JconvexFib_{(J,\p_\theta)}(\Pi)$ depends on an arbitrary choice of
$(J,\p_\theta) \in \AC(\Pi)$, which we shall assume fixed throughout.
Note that while $\omega \in \strongFib(\Pi)$ is required to match
$d\lambda$ for a prescribed primitive $\lambda$ near the boundary,
the statement for the almost Stein case does not require this.

Given $\lambda \in \GirouxFib(\p_h\Pi)$ or $\GirouxFib(\p \Pi)$, let
$\eta$ denote the closed $2$-form
guaranteed by Lemma~\ref{lemma:positive}.  Observe that by the fiberwise
Giroux form condition, the integrals
of $\lambda$ over boundary components of fibers $E_z$ are locally constant
functions of~$z$.  We proceed in three steps:

\textsl{Step~1: Neighborhoods of regular fibers.}
For each $z \in \Sigma \setminus \Sigma\crit$, there exists an open
neighborhood $z \in \uU_z \subset \Sigma\setminus\Sigma\crit$
and a $1$-form $\lambda_z$ on $E|_{\uU_z}$ which restricts to $\lambda$
near $\p_h E$ such that $\omega_z := d\lambda_z$ is an area form on
every fiber in~$E|_{\uU_z}$.  If $\lambda \in \GirouxFib(\p\Pi)$ then
$d\lambda$ is already positive on the fibers near $\p_v E$, thus we
can also arrange $\lambda_z = \lambda$ near $\p_v E$.

For the almost Stein case, observe first that
the vector field $-J\p_\theta$ along $\p_h E$ is necessarily 
vertical and points transversely outward.
Choose a smooth function
$f_z : E_z \to \RR$ such that $-d(d f_z \circ J) > 0$ on $E_z$,
while at $\p E_z$, $f_z \equiv c_z$ and $d f_z(-J \p_\theta) \equiv \nu_z$
for some constants $c_z, \nu_z > 0$. This can be achieved by starting
with any smooth function that satisfies these conditions at the boundary
and has only Morse critical points of index~$0$ and~$1$, then modifying
it by local diffeomorphisms to make it $J$-convex near the critical points,
and postcomposing it with a function with very large second derivative;
cf.~Lemma~\ref{lemma:sufficientlyConvex}.
We can then find a neighborhood $\uU_z \subset \Sigma \setminus \Sigma\crit$
of $z$ such that $f_z$ admits an extension to a function
$$
f_z : E|_{\uU_z} \to \RR
$$
satisfying these same properties on every fiber in $E|_{\uU_z}$.
Note that the constants $c_z$ and $\nu_z$ can always be made larger without 
changing the neighborhood~$\uU_z$. The $1$-form $\lambda_z := -df_z \circ J$
on $E|_{\uU_z}$ now satisfies $d\lambda_z > 0$ on every fiber, and 
its restriction to the horizontal boundary
$$
\alpha_z^h := \lambda_z|_{T(\p_h E)}
$$
satisfies $\alpha_z^h(\p_\theta) = \nu_z$, $\alpha_z^h|_{\xi_J} = 0$.
The invariance of $\xi_J$ under the flow of $\p_\theta$ then implies
$\Lie_{\p_\theta} \alpha_z^h = d\alpha_z^h(\p_\theta,\cdot) = 0$.

\textsl{Step~2: Neighborhoods of singular fibers.}
For each $z \in \Sigma\crit$, let $E_z\crit$ denote the finite set of
critical points in~$E_z$.  For each $p \in E_z\crit$, choose
$J\crit$-holomorphic Morse coordinates $(z_1,z_2)$ on a neighborhood
$\uU_p \subset E$ of~$p$, and let $\omega\crit$ denote the symplectic form
on $\uU_p$ which looks like the standard symplectic form on $\CC^2$ in
these coordinates.  Choose an area form $\omega_z$ on $E_z \setminus
E_z\crit$ satisfying the following conditions:
\begin{itemize}
\item $\omega_z$ restricts to $d\lambda$ near $\p_h E$;
\item $\omega_z = \omega\crit$ near $E_z\crit$;
\item For each irreducible component $C \subset E_z$,
$\int_C \omega_z = \int_C \eta$.
\end{itemize}
This can be extended to a
closed $2$-form on $E|_{\uU_z}$ for some open neighborhood
$z \in \uU_z \subset \Sigma$ with $\overline{\uU}_z \subset \mathring{\Sigma}$, 
such that the extended $\omega_z$
also matches $d\lambda$ near $\p_h E$ and is positive on fibers.

For the almost Stein case, we must assume explicitly at this step that
$\Pi : E \to \Sigma$ is allowable, so in particular, the connected
components of $E_z \setminus E_z\crit$ are all compact oriented surfaces
with \emph{nonempty boundary} and finitely many punctures. Using the same
$J\crit$-holomorphic coordinates $(z_1,z_2)$ as above near any
$p \in E_z\crit$, define a function $f_z : \uU_p \to \RR$ by
$$
f_z(z_1,z_2) = \frac{1}{2}\left( |z_1|^2 + |z_2|^2 \right).
$$
This function is $J\crit$-convex, and we claim that it is also
$J$-convex on a sufficiently small neighborhood of~$p$. To see this,
recall that $J$ and $J\crit$ match at $p$ due to 
Proposition~\ref{prop:J1J2}.  Since $d f_z(p) = 0$, the $1$-forms
$-d f_z \circ J$ and $-d f_z \circ J\crit$ have the same $1$-jet at~$p$,
so their exterior derivatives match at that point,
and the claim follows.  By shrinking $\uU_p$ if necessary, we
can therefore assume $-df_z \circ J$ is the primitive of a positive symplectic
form in $\uU_p$ that tames $J$ and restricts symplectically to the vertical
subspaces.  Now since every connected component of $E_z \setminus E_z\crit$
has nonempty boundary, we can extend $f_z$ over $E_z$ so that it is 
$J$-convex on $E_z$ and satisfies $f_z \equiv c_z$, $df_z(-J\p_\theta) \equiv
\nu_z$ at $\p E_z$.  Using the fact that $J$-convexity is an open condition,
we can then extend $f_z$ over $E|_{\uU_z}$ for some neighborhood
$z \in \uU_z \subset \Sigma$ so that it has these same properties on each
fiber. The constants $c_z$ and $\nu_z$ can again be made larger if desired
without changing the neighborhood~$\uU_z$.

\textsl{Step~3: Partition of unity.}
Since $\Sigma$ is compact, there is a finite subset $I \subset \Sigma$
such that the open sets $\{ \uU_z \}_{z \in I}$ cover $\Sigma$.
Choose a partition of unity $\{ \rho_z : \uU_z \to [0,1] \}_{z \in I}$
subordinate to this cover.  For each $z \in I$, the $2$-form
$\omega_z - \eta$ on $E|_{\uU_z}$ is exact by construction, thus we can
pick a $1$-form $\theta_z$ on $E|_{\uU_z}$ with
$$
\omega_z = \eta + d\theta_z,
$$
and since $\omega_z$ and $\eta$ both match $d\lambda$ on a neighborhood
of $\p_h E$ or $\p E$ respectively,
we can choose $\theta_z$ such that $\theta_z = 0$ on such a neighborhood.
We can then define $\omega \in \strongFib(\Pi)$ by
$$
\omega = \eta + d \left( \sum_{z \in I} (\rho_z \circ \Pi) \theta_z \right).
$$

For the almost Stein case, consider the same partition of unity with the
functions $f_z : E|_{\uU_z} \to \RR$ constructed in the
first two steps, for $z \in I$. By making these functions more convex near $\p_h E$,
we can increase the constants $c_z > 0$ for all $z \in I$ so that they
match a single constant $c > 0$, and likewise increase $\nu_z$ for
$z \in I$ to match some large number $\nu > 0$.  The function
$$
f: = \sum_{z \in I} (\rho_z \circ \Pi) f_z
$$
is then constant at $\p_h E$.
Writing $\lambda_J = -df \circ J$, we also have $d\lambda_J > 0$ on
all fibers, while $d\lambda_J$ is symplectic and tames $J$ near $E\crit$, 
and the
$1$-form $\alpha^h := \lambda_J|_{T(\p_h E)}$ satisfies
$$
\alpha^h(\p_\theta) \equiv \nu > 0, \quad\text{ and }\quad
\alpha^h|_{\xi_J} \equiv 0,
$$
thus the invariance of $\xi_J$ under the flow of $\p_\theta$ implies
$$
d\alpha^h(\p_\theta,\cdot) \equiv \Lie_{\p_\theta} \alpha^h \equiv 0.
$$
\end{proof}

\begin{remark}
\label{remark:notHolomorphic}
It will occasionally (e.g.~in Lemma~\ref{lemma:asyetunspecified})
be useful to observe that in the almost Stein case, the
above proof did not make any use of the assumption that $\Pi : (E,J) \to
(\Sigma,j)$ is pseudoholomorphic.  The conditions on $(J,\p_\theta)$ we used
were merely that every fiber is $J$-holomorphic and the $S^1$-action
defined by $\p_\theta$ on $\p_h E$ preserves $\xi_J := T(\p_h E) \cap J T(\p_h E)$
and~$J|_{\xi_J}$.
\end{remark}

To move from fiberwise structures to honest symplectic structures, we apply
the Thurston trick.  Fix a Liouville form $\sigma$ on~$\Sigma$.
For the almost Stein case, we may also assume
$$
\sigma = -d\varphi \circ j,
$$
where $\varphi : \Sigma \to \RR$ is a smooth function constant at the
boundary and $j$ is the unique complex structure on $\Sigma$ for which
$\Pi : (E,J) \to (\Sigma,j)$ is pseudoholomorphic.

\begin{prop}
\label{prop:Thurston}
Given $\omega$ in $\weakFib(\Pi)$, $\strongFib(\Pi)$ or $\LiouvilleFib(\Pi)$, there
exists a constant $K_0 \ge 0$, depending continuously on~$\omega$, such that
for every $K \ge K_0$,
$$
\omega_K := \omega + K \, \Pi^* d\sigma
$$
belongs to $\weak(\Pi)$, $\strong(\Pi)$ or $\Liouville(\Pi)$ respectively.

Similarly, given $(J,\p_\theta) \in \AC(\Pi)$ and 
$f \in \JconvexFib_{(J,\p_\theta)}(\Pi)$, there exists $K_0 \ge 0$, depending
continuously on $J$ and $f$, such that for every $K \ge K_0$,
$$
f_K := f + K (\varphi \circ \Pi)
$$
belongs to $\Jconvex_{(J,\p_\theta)}(\Pi)$.

Moreover, if $\omega$ is already in $\weak(\Pi)$, $\strong(\Pi)$ or 
$\Liouville(\Pi)$, or $f$ is already in $\Jconvex_{(J,\p_\theta)}(\Pi)$ 
respectively, then for both statements it suffices to set $K_0 = 0$.
\end{prop}
\begin{proof}
Let $\uU\crit \subset E$ denote a neighborhood of $E\crit$ on which the
integrable complex structure $J\crit$ is defined and
$\Pi|_{\uU\crit}$ is holomorphic; more precisely for each $p \in E\crit$,
a neighborhood of $\Pi(p)$ in $\Sigma$ admits a complex structure $j_p$
such that the restriction of $\Pi$ to the connected component $\uU_p$ of 
$\uU\crit$
containing $p$ is a holomorphic map $(\uU_p,J\crit) \to (\Pi(\uU_p),j_p)$.
Assume to start with that $\omega \in \weakFib(\Pi)$.
By shrinking $\uU\crit$ if necessary,
we may assume $\omega|_{\uU\crit}$ is symplectic and tames~$J\crit$.
Now for any nonzero vector $v \in TE|_{\uU_p}$ for $p \in E\crit$, we have
$$
\omega_K(v,J\crit v) = \omega(v,J\crit v) + 
K \, d\sigma(\Pi_*v, j_p \Pi_* v),
$$
in which the first term is positive and the second is nonnegative for any
$K \ge 0$, hence $\omega_K|_{\uU\crit}$ is symplectic and positive on the
fibers.  In the almost Stein case, we write $\lambda_J := - d f \circ J$
and $\sigma := - d\varphi \circ j$ and observe that
the holomorphicity of $\Pi$ implies
$- d (\varphi \circ \Pi) \circ J = \Pi^*( - d\varphi \circ j) =
\Pi^*\sigma$, hence $\lambda_J^K := - d f_K \circ J = \lambda + K \, \Pi^*\sigma$.
We then have
$$
d\lambda_J^K(v,J v) = d\lambda_J(v,Jv) + K\, d\sigma(\Pi_*v,j \Pi_*v),
$$
and for any $v \ne 0$ near $E\crit$ this is again positive since
$d\lambda_J$ tames $J\crit$ and, by Prop.~\ref{prop:J1J2},
the latter matches $J$ at~$E\crit$.

Outside a neighborhood of $E\crit$, the rest follows by direct application
of the results in \S\ref{sec:Thurston}.
\end{proof}

Applying Whitehead's theorem as in the proof of Theorem~\ref{thm:GirouxForms},
Propositions~\ref{prop:nonempty} and~\ref{prop:Thurston} together imply
that the various spaces of supported symplectic structures in
Theorem~\ref{thm:Gompf} are nonempty and contractible as claimed.
They also imply that the fibers of the projection
$\ACStein(\Pi) \to \AC(\Pi) : (J,f) \mapsto (J,\p_\theta)$ are nonempty
and contractible. To see that this projection is also a Serre fibration,
it suffices to observe that due to the continuous dependence on $J$ and $f$,
the construction in Proposition~\ref{prop:Thurston} of the 
$J$-convex function $f_K$ can be done parametrically.
This completes the proof of Proposition~\ref{prop:almostSteinCtrbl}.

\subsection{Smoothing corners}
\label{subsec:smoothing}

To finish the
proof of Theorem~\ref{thm:Gompf}, we must show that the corners of $\p E$
can be smoothed in a way that yields a symplectic filling canonically up
to deformation.  For strong fillings this is mostly obvious because we have
a Liouville vector field transverse to both smooth faces of $\p E$, but the
case of weak fillings requires a bit more thought since there is no Liouville
vector field.  We will consider a specific class of smoothings defined
as follows.

Fix a collar neighborhood $\nN(\p\Sigma) = (-1,0] \times \p\Sigma
\subset \Sigma$ and a corresponding collar neighborhood
$E|_{\nN(\p\Sigma)} =: \nN(\p_v E) = (-1,0] \times \p_v E \subset E$ such that
$$
\Pi|_{\nN(\p_v E)} : \nN(\p_v E) \to \nN(\p\Sigma) : (s,p) \mapsto (s,\Pi(p)).
$$
Fix also a collar $\nN(\p_h E) = (-1,0] \times \p_h E \subset E$
such that
$$
\Pi|_{\nN(\p_h E)} : \nN(\p_h E) \to \Sigma : (t,p) \mapsto \Pi(p).
$$
The intersection of these two collars is then a neighborhood of the
corner
$$
\nN(\p_v E \cap \p_h E) := \nN(\p_h E) \cap \nN(\p_v E) =
(-1,0] \times (-1,0] \times \left( \p_h E \cap \p_v E \right),
$$
and in coordinates $(s,t,p)$ on this neighborhood we have
$\Pi(s,t,p) = (s,\Pi(p))$.
Given a constant $\epsilon \in (0,1)$, choose a pair of
smooth functions $f_\epsilon , g_\epsilon : (-1,1) \to (-1,1)$ satisfying the following
conditions:
\begin{itemize}
\item For $\tau \le -\epsilon$, $f_\epsilon(\tau) = \tau$ and $g_\epsilon(\tau) = 0$,
\item For $\tau \in (-\epsilon,\epsilon)$, $f_\epsilon'(\tau) > 0$ and
$g_\epsilon'(\tau) < 0$,
\item For $\tau \ge \epsilon$, $f_\epsilon(\tau) = 0$ and $g_\epsilon(\tau) = -\tau$.
\end{itemize}
Denote by $\gamma_\epsilon \subset (-1,0] \times (-1,0]$ the image of the smooth
path $(f_\epsilon(\tau),g_\epsilon(\tau))$ for $\tau \in (-1,1)$; this divides
$(-1,0] \times (-1,0]$ into two connected components.  We shall
denote the component of $\left((-1,0] \times (-1,0]\right) \setminus \gamma_\epsilon$
containing $(0,0)$ by~$\Gamma_\epsilon$ (see Figure~\ref{fig:smoothing}), 
and then define the compact domain
$$
W_\epsilon =  E \setminus \left(\Gamma_\epsilon \times \nN(\p_v E \cap \p_h E)\right).
$$
This is a smooth manifold with boundary $M_\epsilon := \p W_\epsilon$,
and the latter can be identified with
$\p E$ canonically up to a continuous isotopy which is
smooth outside the corner.

\begin{figure}
\begin{minipage}[t]{3in}
\psfrag{f}{$f$}
\psfrag{g}{$g$}
\psfrag{gmmaeps}{$\gamma_\epsilon$}
\psfrag{Gammaeps}{$\Gamma_\epsilon$}
\psfrag{-eps}{$-\epsilon$}
\psfrag{-1}{$-1$}
\includegraphics[scale=1.5]{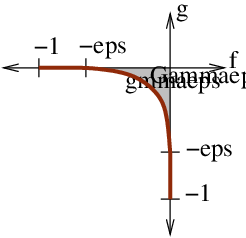}
\caption{\label{fig:smoothing} 
The path $\gamma_\epsilon$ and region $\Gamma_\epsilon \subset (-1,0] \times
(-1,0]$ used for smoothing corners.}
\end{minipage}
\hfill
\begin{minipage}[t]{3in}
\psfrag{s}{$s$}
\psfrag{t}{$t$}
\psfrag{Meps}{$M_\epsilon$}
\psfrag{Weps}{$W_\epsilon$}
\psfrag{-eps}{$-\epsilon$}
\includegraphics[scale=1.5]{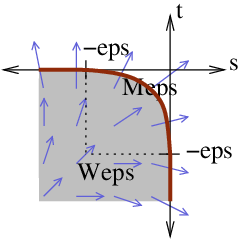}
\caption{\label{fig:smoothing2} 
The smoothed corner of $M_\epsilon = \p W_\epsilon$ with the transverse vector field~$V_K$.}
\end{minipage}
\end{figure}

\begin{prop}
\label{prop:smoothing}
Suppose $\omega \in \weak(\Pi)$, $\strong(\Pi)$ or $\Liouville(\Pi)$,
or $\omega = -d (d f \circ J)$ for some $(J,f) \in \ACStein(\Pi)$.
Then for sufficiently small
$\epsilon > 0$, the domain $W_\epsilon$ with its symplectic or almost Stein
data is a weak, strong, exact or almost Stein filling respectively of
$(M_\epsilon,\xi_\epsilon)$, where
$\xi_\epsilon$ is a contact structure supported by a spinal open
book with smooth overlap that is isotopic 
(in the sense of Remark~\ref{remark:smoothingSOB})
to~$\p\Pi$.  Moreover, any two fillings obtained in this way by different
choices of smoothing are deformation equivalent.
\end{prop}
\begin{proof}
Assume $\omega \in \weak(\Pi)$, so $\omega = d\lambda$ near $\p_h E$ for
some $\lambda \in \Giroux(\p_h\Pi)$.  One can extend $\lambda$ to a
neighborhood of $\p E$ so that $\lambda \in \GirouxFib(\p E)$; this follows
from Proposition~\ref{prop:nonempty} (or a simpler variant focusing only on a
neighborhood of~$\p_v E$).  Choosing a Liouville form $\sigma$ on $\Sigma$,
Proposition~\ref{prop:GirouxNearBoundary} then implies that for
sufficiently large constants $K > 0$, the $1$-form
$$
\lambda_K := \lambda + K\, \Pi^*\sigma
$$
defines a Giroux form near $\p E$.  Further, we claim that if
$K > 0$ is sufficiently large, then
$\lambda_K \wedge \omega > 0$ restricts positively to both
$\p_h E$ and $\p_v E$.  On $\p_h E$ this is immediate since
$\omega = d\lambda$ near $\p_h E$ and $\lambda|_{T(\p_h E)}$ is contact, hence
$$
\lambda_K \wedge \omega|_{T(\p_h E)} = (\lambda + K\, \Pi^*\sigma)
\wedge d\lambda|_{T(\p_h E)} = \lambda \wedge d\lambda|_{T(\p_h E)} > 0,
$$
where the term $(\Pi^*\sigma \wedge d\lambda)|_{T(\p_h E)}$ vanishes
because both $\Pi^*\sigma$ and $d\lambda|_{T(\p_h E)}$ kill $V(\p_h E)$.
On $\p_v E$, we have
$$
\lambda_K \wedge \omega|_{T(\p_v E)} = (\lambda + K\, \Pi^*\sigma)
\wedge \omega|_{T(\p_v E)} = K\, \Pi^*\sigma \wedge \omega|_{T(\p_v E)}
+ \lambda \wedge \omega|_{T(\p_v E)},
$$
in which the first term is positive since $\omega$ is positive on fibers,
hence the sum is positive for $K \gg 0$.

Since $\omega$ is symplectic, there is a vector field $V_K$ defined 
near $\p E$ by the condition $\omega(V_K,\cdot) = \lambda_K$, and 
$\lambda_K \wedge \omega$ is then positive on any given oriented
hypersurface if and only if $V_K$ is positively transverse to that
hypersurface.  It follows that $V_K$ is everywhere positively transverse
to both $\p_h E$ and $\p_v E$, so if $\epsilon > 0$ is chosen
sufficiently small, then $V_K$ has positive $\p_s$ and $\p_t$
components (in the coordinates $(s,t,p)$) everywhere on
$$
(-\epsilon,0] \times (-\epsilon,0] \times (\p_h E \cap \p_v E) \subset
\nN(\p_v E \cap \p_h E).
$$
For this choice of $\epsilon$, $V_K$ is then positively transverse to
$\p W_\epsilon$ everywhere (Figure~\ref{fig:smoothing2}), and it follows that
$$
\lambda_K \wedge \omega|_{TM_\epsilon} > 0.
$$
Thus $(W_\epsilon,\omega)$ is a weak filling of $(M_\epsilon,\xi_\epsilon)$, where
$\xi_\epsilon := \ker \left( \lambda_K|_{TM_\epsilon} \right)$.

To see that this filling is unique up to symplectic deformation, note
first that by the results of the previous subsection, 
$\omega \in \weak(\Pi)$ is unique up to homotopy through~$\weak(\Pi)$.
Given any such
homotopy $\omega_\tau \in \weak(\Pi)$, $\tau \in [0,1]$, one can choose 
a continuous family of primitives $\lambda_\tau \in \Giroux(\p_h\Pi)$,
then extend these to $\lambda_\tau \in \GirouxFib(\p\Pi)$ and choose
$K > 0$ large enough so that 
$\lambda_\tau + K\, \Pi^*\sigma$ defines a continuous family of
Giroux forms near $\p E$ with $(\lambda_\tau + K\, \Pi^*\sigma) \wedge \omega_\tau$
positive on both $\p_h E$ and $\p_v E$.  Then for some continuous deformation
of the parameter $\epsilon_\tau > 0$, shrinking it as small as necessary
for $\tau \in (0,1)$, we can
arrange for $(W_{\epsilon_\tau},\omega_\tau)$ to be a weak filling of
$(M_{\epsilon_\tau}, \ker (\lambda_\tau + K\, \Pi^*\sigma) )$ for all $\tau \in [0,1]$.

The corresponding statements for strong, exact or almost Stein fillings
are proved by a simplification of the above arguments:
if $\omega$ is strongly convex, we may assume $\omega = d\lambda$ with
$\lambda \in \Giroux(\p\Pi)$,
thus $\lambda \wedge d\lambda$ is positive on both boundary faces and
the corresponding Liouville vector field plays the role that $V_K$ played
above.

It remains to show that the contact structure induced on $M_\epsilon$ is
supported by a spinal open book isotopic to $\p \Pi$.  It will suffice
to show this for a particular choice of $\omega \in \strong(\Pi)$.
Choose a coordinate $\phi \in S^1$ for each connected component 
of $\p \Sigma$, so the collar $\nN(\p\Sigma)$ can be viewed as a disjoint
union of components $(-1,0] \times S^1$ with coordinates $(s,\phi)$,
and we can choose $\sigma = e^s\, d\phi$ in these collars.  Choose also
a trivialization of the $S^1$-bundle $\p_h E \cap \p_v E \to \p\Sigma$
and denote the fiber coordinate by $\theta \in S^1$, so each component of
$\p_h E \cap \p_v E$ now has coordinates $(\phi,\theta) \in T^2$, and the
components of $\nN(\p_v E \cap \p_h E)$ inherit coordinates
$(s,t,\phi,\theta) \in (-1,0] \times (-1,0] \times T^2$ with
$$
\Pi(s,t,\phi,\theta) = (s,\phi).
$$
One can then construct a fiberwise Giroux form $\lambda$ near $\p E$
that takes the form $e^t\, d\theta$ in $\nN(\p_v E \cap \p_h E)$,
and extend $d\lambda$ by Proposition~\ref{prop:nonempty} 
to a fiberwise symplectic structure $\omega \in \strongFib(\Pi)$.
Applying Proposition~\ref{prop:Thurston},
this yields a supported strongly convex symplectic structure with a
primitive of the form
$$
\lambda_K := e^t\, d\theta + K e^s\, d\phi
$$
on $\nN(\p_v E \cap \p_h E) = (-1,0] \times (-1,0] \times S^1 \times S^1$.
Defining $W_\epsilon$ as above for any choice of $\epsilon \in (0,1)$, 
$M := \p W_\epsilon$ inherits a spinal open book defined as follows:
the spine $M\spine$ is the complement in $\p_h E$ of the region
$(-\epsilon,0] \times \{0\} \times T^2 \subset \nN(\p_v E \cap \p_h E)$,
with the fibration
$$
\pi\spine : M\spine \to \Sigma'
$$
induced by $\Pi$, where $\Sigma'$ is the complement of the collars
$(-\epsilon,0] \times S^1$ in~$\Sigma$.  The closure of $M \setminus M\spine$
then constitutes the paper $M\paper$, with the $\phi$-coordinate defining
the fibration $\pi\paper : M\paper \to S^1$.  The restriction of
$\lambda_K$ to $M$ is now a Giroux form for this spinal open book.
\end{proof}

\section{A criterion for the canonical Stein homotopy type}
\label{sec:SteinHomotopy}

The characterization of supported almost Stein structures given in
Definition~\ref{defn:almostSteinFibration} is natural, but not general
enough to be useful in classifying fillings up to Stein homotopy.
In particular, the proof of Theorem~\ref{thma:quasi} stated in the
introduction will require us to consider bordered Lefschetz fibrations
$\Pi : E \to \Sigma$ with almost Stein structures $(J,f)$ for which
the fibers are almost complex submanifolds but the projection $\Pi$ is not pseudoholomorphic.
The more general characterization given by Theorem~\ref{thma:SteinHomotopy}
will therefore be useful, and it can be restated as follows.

\begin{thm}
\label{thm:SteinHomotopy}
Suppose $\Pi : E \to \Sigma$ is an allowable bordered Lefschetz fibration,
$j$ is a complex structure on $\Sigma$
and $(J,f)$ is an almost Stein structure on $E$ with the following properties:
\begin{enumerate}
\item $J$ restricts to a positively oriented complex structure on the smooth part of
every fiber;
\item $f$ is constant on the boundary components of every fiber;
\item The restriction of $-df \circ J$ to $\p E$ is a Giroux form for $\p\Pi$
(cf.~Remark~\ref{remark:corner}).
\item There exists an open neighborhood $\uU \subset \Sigma$ of $\p\Sigma$
such that the map
$$
(E|_\uU,J) \stackrel{\Pi}{\longrightarrow} (\uU,j)
$$
is pseudoholomorphic;
\item The maximal $J$-complex subbundle $\xi_J \subset T(\p_h E)$ is preserved
under the Reeb flow defined via $-df \circ J|_{T(\p_h E)}$, and
$J|_{\xi_J} = \Pi^*j$.
\end{enumerate}
Then $(J,f)$ is almost Stein homotopic to an almost Stein structure supported by~$\Pi$.
\end{thm}

\begin{remark}
We will not use this fact, but one can show that whenever the first 
and fifth conditions
in Theorem~\ref{thm:SteinHomotopy} hold, the projection $\Pi$ is pseudoholomorphic
(for a suitable choice of complex structure on the base) whenever the
Nijenhuis tensor takes values in the vertical subbundle. In particular, this is always
true if $J$ is integrable.
\end{remark}
%

We begin by generalizing the space $\AC(\Pi)$ from Definition~\ref{defn:AC}.

\begin{defn}
\label{defn:ACweaker}
Given an open subset $\uU \subset \Sigma$,
let $\ACweaker(\Pi;\uU)$ denote the space of pairs $(J,\p_\theta)$, where $J$ is an
almost complex structure on $E$ defining the correct orientation, $\p_\theta$ is a
positively oriented nowhere zero vertical vector field on $\p_h E$,
and $\Sigma$ admits a complex structure $j$ so that the following conditions are satisfied:
\begin{enumerate}
\item
$J$ restricts to a positively oriented complex structure on the smooth part of every
fiber;
\item
The equation $T\Pi \circ J = j \circ T\Pi$ is satisfied in
$E|_{\uU}$ and along~$\p_h E$;
\item
The flow of $\p_\theta$ on $\p_h E$ is $1$-periodic and preserves
$\xi_J := T(\p_h E) \cap J T(\p_h E)$.
\end{enumerate}
\end{defn}

Observe that $\ACweaker(\Pi;\Sigma)=\AC(\Pi)$, and 
$\ACweaker(\Pi;\uU') \subset \ACweaker(\Pi;\uU)$ whenever $\uU \subset\uU'$;
moreover, $\ACweaker(\Pi;\uU)$ is contractible for every choice of
$\uU \subset \Sigma$.
Theorem~\ref{thm:SteinHomotopy} would thus follow immediately if we could
show that every $(J,\p_\theta) \in \ACweaker(\Pi;\uU)$ admits a suitable
$J$-convex function, but this is probably not true in general---we at least have been
unable to prove it except when $\uU = \Sigma$.  What we will
show instead is that if $\uU$ contains~$\p\Sigma$, then every 
$(J,\p_\theta) \in \ACweaker(\Pi;\uU)$ has a
\emph{perturbation} that admits a suitable $J$-convex function, and this
perturbation can be arranged to depend continuously on parameters.
Here is the more technical result that implies Theorem~\ref{thm:SteinHomotopy}:

\begin{prop}
\label{prop:parametricJ}
Assume $\Pi : E \to \Sigma$ is allowable, $\uU \subset \Sigma$ is an
open neighborhood of~$\p\Sigma$, $X$ is a compact cell complex,
$A \subset X$ is a subcomplex, and
\begin{equation*}
\begin{split}
X &\to \ACweaker(\Pi;\uU) : \tau \mapsto (J_\tau,\p_\theta^\tau), \\
A &\to C^\infty(E) : \tau \mapsto f_\tau
\end{split}
\end{equation*}
are continuous maps such that for every $\tau \in A$,
$(J_\tau,f_\tau)$ is an almost Stein structure, $f_\tau$ is constant on all
boundary components of fibers, and $\lambda_\tau := -d f_\tau \circ J_\tau$
restricts to $\p E$ as a Giroux form for~$\p\Pi$ (in the sense of Remark~\ref{remark:corner})
satisfying $\lambda_\tau(\p_\theta^\tau) \equiv \operatorname{const}$.
Then there exists a continuous (with respect to the $C^\infty$-topology)
family of almost Stein structures 
$\left\{ (J'_\tau,f'_\tau) \right\}_{\tau \in X}$ matching
$(J_\tau,f_\tau)$ for all $\tau \in A$ such that $J'_\tau$ is $C^\infty$-close
to $J_\tau$ for all $\tau \in X$.
\end{prop}

The proof of Proposition~\ref{prop:parametricJ} requires several steps and
will occupy the remainder of this section, so here is an initial sketch.  
Let $\{j_\tau\}_{\tau \in X}$ 
denote the uniquely determined family of complex structures on $\Sigma$ such 
that $\Pi : (E,J_\tau) \to (\Sigma,j_\tau)$ is holomorphic in $E|_{\uU}$
and along~$\p_h E$ for all~$\tau$.  Since $\uU$ is open, we can choose a function
$\varphi : \Sigma \to \RR$ which has all its critical points in $\uU$ and
is $j_\tau$-convex for every~$\tau$.  Holomorphicity of $\Pi$ then allows the construction of
$J_\tau$-convex functions on $E|_\uU$ using the Thurston trick as
in Prop.~\ref{prop:ThurstonStein}.  Outside of
$E|_\uU$, the function $\varphi \circ \Pi$ has level sets that are
unions of $J_\tau$-holomorphic fibers and are thus Levi-flat, i.e.~the maximal complex
subbundle in each level set is a foliation.  A suitable choice of
fiberwise Liouville structure then allows us to perturb these foliations to contact
structures as in the Thurston-Winkelnkemper construction \cite{ThurstonWinkelnkemper} of contact forms
supported by open books (cf.~Prop.~\ref{prop:GirouxVertical}).
The almost complex structures admit corresponding perturbations $J_\tau'$ that
preserve these contact structures, so that
the function $\varphi \circ \Pi$,
after modifying $\varphi$ to make $\varphi''$ sufficiently large,
becomes $J_\tau'$-convex.  This makes use of Lemma~\ref{lemma:sufficientlyConvex},
and it produces $J_\tau'$-convex functions $f_\tau'$ that match
$\varphi \circ\Pi$ away from $E|_{\Crit(\varphi)}$ and take the form
$\varphi \circ \Pi + \epsilon f_\tau$ near $E|_{\Crit(\varphi)}$.
Actually proving that $(J_\tau',f_\tau')$ are almost Stein structures requires
also showing that the Liouville forms $- d f_\tau' \circ J_\tau'$ restrict
to the smooth faces of $\p E$ as contact forms.  Moreover, we need to be
able to keep this condition under linear interpolations between our constructed
functions $f_\tau'$ and the original $f_\tau$ in order solve the extension
problem.  Both steps will make essential use of the holomorphicity of
$\Pi$ at $\p E$, as well as the Thurston trick: a crucial detail for the
latter is that the original family of $J_\tau$-convex functions 
$\{f_\tau\}_{\tau \in A}$ can easily be extended to $\tau \in X$ as a family of
\emph{fiberwise} $J_\tau$-convex functions, which we use in the construction
of $J_\tau'$ and~$f_\tau'$.

We now proceed with the details of the argument sketched above.  
As in the statement of 
Proposition~\ref{prop:parametricJ}, all families of objects parametrized by
$X$ will be assumed in the following to be continuous in the $C^\infty$-topology,
and $\uU \subset \Sigma$ will be an open neighborhood of~$\p\Sigma$.
We will sometimes find it convenient to replace $\uU$ with a smaller neighborhood
of~$\p\Sigma$, which is not a loss of generality since it enlarges the space
$\ACweaker(\Pi;\uU)$.  In particular, since all critical values of $\Pi$ are 
in the interior, let us start by assuming
$$
\widebar{\uU} \cap \Sigma\crit = \emptyset.
$$

\subsection{Weinstein structures on the base}
\label{sec:step1}

\begin{lemma}
\label{lemma:phi}
There exists a smooth function $\varphi : \Sigma \to \RR$ which is
$j_\tau$-convex for all $\tau$ and constant on $\p\Sigma$, and has all
its critical points in $\uU \setminus \p\Sigma$.
\end{lemma}
\begin{proof}
Start by choosing a Morse function $\varphi : \Sigma \to \RR$ that is regular 
and constant on the boundary and has no local maxima.  By composing
with a suitable diffeomorphism of $\Sigma$, we can arrange that
$\Crit(\varphi) \subset \uU \setminus \p\Sigma$.
Now since every critical point has Morse index~$0$ or~$1$, we can fix
local coordinates $(x,y)$ near each critical point so that, up to addition
of constants, $\varphi(x,y)$
takes the form $x^2 + y^2$ or $x^2 - y^2$.  Given any constant $c > 0$, we
can further modify $\varphi$ by composing with a diffeomorphism supported
near the index~$1$ critical points so that these (in the same coordinates!)
now take the form $c x^2 - y^2$.  Since the parameter space $X$ is compact,
Lemma~\ref{lemma:jconvexNearCritical} below now implies that by selecting $c$
sufficiently large, we can assume $\varphi$ is $j_\tau$-convex near
$\Crit(\varphi)$ for all $\tau \in X$.  Lemma~\ref{lemma:sufficientlyConvex}
(with Remark~\ref{remark:dimension2})
can then be applied to make $\varphi$ into a globally $j_\tau$-convex function
for all $\tau \in X$ by postcomposing it with a sufficiently convex
function $\RR \to \RR$.
\end{proof}

The above proof required the following lemma:

\begin{lemma}
\label{lemma:jconvexNearCritical}
Suppose $j$ is a smooth almost complex structure on a neighborhood of~$0$
in $\CC$, compatible with the canonical orientation,
and let $\varphi_0, \varphi_1 : \CC \to \RR$ denote the functions
$$
\varphi_0(x + iy) = x^2 + y^2, \qquad
\varphi_1(x + iy) = c x^2 - y^2,
$$
where $c > 0$ is a constant.  Then $\varphi_0$ is $j$-convex near~$0$,
and $\varphi_1$ is also $j$-convex near~$0$ whenever $c$ is sufficiently
large.
\end{lemma}
\begin{proof}
Let $j_0$ denote the ``constant'' complex structure on $\CC$ that
matches $j$ at the origin, in other words $j_0(z) := j(0)$ for all $z \in \CC$.
We claim first that the statement of the lemma is true if $j$ is replaced 
by~$j_0$.  Indeed, $j_0$ can be written as the matrix
$$
j_0 = \begin{pmatrix}
a & -\frac{1 + a^2}{b} \\
b & -a
\end{pmatrix},
$$
where $a$ and $b$ are real constants with $b > 0$ (due to the orientation
assumption).  Then we compute:
\begin{equation*}
\begin{split}
- d (d\varphi_0 \circ j_0) = 2 \left( \frac{1 + a^2}{b} + b \right) \, dx \wedge dy, \\
- d (d\varphi_1 \circ j_0) = 2 \left( c \frac{1 + a^2}{b} - b \right) \, 
dx \wedge dy.
\end{split}
\end{equation*}
The first is always positive, and the second is positive if and only if
$c > b^2 / (1 + a^2)$, so this proves the claim about $j_0$.  To
generalize this to~$j$, it suffices to observe that since $d\varphi_0(0) =
d\varphi_1(0) = 0$, the $1$-jets of $-d\varphi_0 \circ j$ and 
$-d\varphi_1 \circ j$ at~$0$ (and hence also the question of $j$-convexity 
on some neighborhood of that point)
depend on $j(0)$ but not on the derivatives of~$j$, so the fact that
$j(0) = j_0(0)$ implies the result.
\end{proof}

For the remainder of this section, we fix a function $\varphi : \Sigma \to \RR$ as given
by Lemma~\ref{lemma:phi} and define the family of $1$-forms
$$
\sigma_\tau = -d\varphi \circ j_\tau.
$$
By construction, $d\sigma_\tau > 0$ everywhere and
$d\varphi \wedge \sigma_\tau > 0$ away from $\Crit(\varphi)$, for all
$\tau \in X$.  In particular, this means that $d\sigma_\tau$, together
with $\varphi$ and the family of Liouville vector fields $d\sigma_\tau$-dual
to $\sigma_\tau$, define a family of Weinstein structures on~$\Sigma$.

\subsection{Perturbing $J$ near Lefschetz critical points}
\label{sec:step2}

Define the function
$$
F = \varphi \circ \Pi : E \to \RR.
$$
We shall now define a family of perturbations of $J_\tau$ near $E\crit$
that make $F$ plurisubharmonic on this neighborhood.
For any $p \in E\crit$, let $\nN(p)$ denote an open neighborhood of~$p$,
which we will always assume is arbitrarily small in order to satisfy various
conditions.  The first such condition is that $\nN(p)$ admits
complex coordinates
$(z_1,z_2)$ identifying $p$ with $(0,0) \in \CC^2$
so that $\Pi(z_1,z_2) = z_1^2 + z_2^2$ for a suitable choice
of complex coordinate $z$ on a neighborhood $\nN(\Pi(p)) \subset \Sigma$
of $\Pi(p)$, identifying $\Pi(p)$ with $0 \in \CC$.
We shall abbreviate the pair of coordinates on $\nN(p)$ together as
$\zeta = (z_1,z_2)$, and write the real and imaginary parts as
$$
\zeta = (z_1,z_2) = (x_1 + i y_1, x_2 + i y_2) \in \nN(p), \qquad
z = x + iy \in \nN(\Pi(p)).
$$
Note that the formula for
$\Pi(z_1,z_2)$ is invariant under simultaneous coordinate changes of the
form
$$
(z_1,z_2) \mapsto (a z_1, a z_2), \qquad z \mapsto a^2 z
$$
for any $a \in \CC$, thus we can choose a suitable constant~$a$ and make such
a transformation such that without loss of generality, the local coordinate
expression for $\varphi$ near $\Pi(p)$ satisfies
$$
d\varphi(0) = dx.
$$
This is possible because we have already arranged for all critical points
of $\varphi$ to be separate from $\Sigma\crit$; indeed,
$\Crit(\varphi) \subset \uU$ and $\uU \cap \Sigma\crit
= \emptyset$, where the latter can always be achieved by making $\uU$ a
smaller neighborhood of~$\p\Sigma$.
In particular, $\varphi$ then has the same $1$-jet at $\Pi(p)$
as the locally defined function
$$
\varphi_0(x + iy) := x + \varphi(0).
$$
We shall repeatedly make use of this fact via the following lemma,
which is an easy consequence of the fact that 
$d\varphi_0(0) = d\varphi(0)$ and $d\Pi(0,0) = 0$.

\begin{lemma}
\label{lemma:matchingJets}
The functions $F = \varphi \circ \Pi : E \to \RR$ and 
$F_0 := \varphi_0 \circ \Pi : \nN(p) \to \RR$
have the same $2$-jet at~$p$.  Moreover, for any smooth bundle endomorphism
$A : \nN(p) \to \End\left(TE|_{\nN(p)}\right)$, the $1$-forms
$dF \circ A$ and $dF_0 \circ A$ have matching $1$-jets at~$p$, which depend on
$A(p)$ but not on the derivatives of $A$ at~$p$. \qed
\end{lemma}

Denote by $i$ the standard complex structure on $\CC^2$ and
identify this with an integrable complex structure on $\nN(p)$ via the
coordinates $(z_1,z_2)$.
For any $i$-antilinear map $Y$ on $\CC^2$ sufficiently close to~$0$,
one can define another complex structure close to~$i$ by
$$
\Phi(Y) := \left( \1 + \frac{1}{2} i Y \right) i \left( \1 + \frac{1}{2} i Y
\right)^{-1}.
$$
Indeed, $\Phi$ can be regarded as the inverse of a local chart for the
manifold of complex structures $\jJ(\CC^2)$, identifying a neighborhood of~$i$
in $\jJ(\CC^2)$ with a neighborhood of~$0$ in $T_i \jJ(\CC^2)$ such that
$d\Phi(0)$ is the identity on the space of $i$-antilinear maps.
By Proposition~\ref{prop:J1J2}, $J_\tau(p) = i$ for all $\tau \in X$, thus
there is a family of smooth maps $Y_\tau : \nN(p) \to T_i\jJ(\CC^2)$ such 
that for all $\zeta \in \nN(p)$,
$$
J_\tau(\zeta) = \Phi(Y_\tau(\zeta)),
$$
and $Y_\tau(0) = 0$.

Working in real coordinates $(x_1,y_1,x_2,y_2)$, define the $i$-antilinear 
matrix
$$
Y' = \begin{pmatrix}
0 & -1 & 0 & 0 \\
-1 & 0 & 0 & 0 \\
0 & 0 & 0 & -1 \\
0 & 0 & -1 & 0
\end{pmatrix}.
$$
We use this to define for all $\epsilon \ge 0$ sufficiently small a family 
of perturbed almost complex structures on $\nN(p)$ by
$$
J_\tau^\epsilon(\zeta) = \Phi(Y_\tau(\zeta) + \epsilon Y').
$$
Let
$$
Y'_\tau(\zeta) = \left. \frac{\p}{\p \epsilon} J_\tau^\epsilon(\zeta)\right|_{\epsilon=0}.
$$
Then since $Y_\tau(0,0) = 0$ and $d\Phi(0)$ is the identity, we have
$Y'_\tau(0,0) = Y'$.  Let
$$
\Lambda_\tau^\epsilon = - dF \circ J_\tau^\epsilon,
$$
and for $\epsilon > 0$, define smooth families of $1$-forms $\widehat{\eta}_\tau^\epsilon$
via the formula
\begin{equation}
\label{eqn:LambdaTauEps}
\Lambda_\tau^\epsilon = \Lambda_\tau^0 + \epsilon \widehat{\eta}_\tau^\epsilon.
\end{equation}
There is a smooth extension of $\widehat{\eta}_\tau^\epsilon$ to $\epsilon=0$, namely
$$
\widehat{\eta}_\tau^0 := \lim_{\epsilon \to 0} \frac{\Lambda_\tau^\epsilon - \Lambda_\tau^0}{\epsilon}
= \left. \frac{\p}{\p\epsilon} \Lambda_\tau^\epsilon \right|_{\epsilon=0} =
- d F \circ Y'_\tau.
$$

\begin{lemma}
\label{lemma:perturbedJcrit}
There exists a constant $\epsilon_0 > 0$ such that
for all $\epsilon \in
(0,\epsilon_0]$ and $\tau \in X$, $d\Lambda_\tau^\epsilon$ is symplectic 
on $\nN(p)$ and tames both $i$ and~$J_\tau^\epsilon$.  Moreover,
$d\widehat{\eta}_\tau^0$ is also symplectic on $\nN(p)$ and tames~$J_\tau$
for all $\tau \in X$.
\end{lemma}
\begin{proof}
We first prove the claim about $d\widehat{\eta}_\tau^0$, for which it suffices to show
that $d\widehat{\eta}_\tau^0|_p$ tames~$i$ since $J_\tau(p) = i$ for all $\tau \in X$
and the taming condition is open.  Consider the slightly simpler $1$-form
$$
\widehat{\eta}_0 := -d(\varphi_0 \circ \Pi) \circ Y',
$$
where we recall $\varphi_0(x+iy) = x + \varphi(0)$.
We then have $\varphi_0 \circ \Pi(z_1,z_2) - \varphi(0) = \Re\left( z_1^2 + z_2^2 \right)
= \sum_{j=1}^2 (x_j^2 - y_j^2)$ and
$$
dx_j \circ Y' = -dy_j, \qquad dy_j \circ Y' = -dx_j \quad \text{ for $j=1,2$},
$$
thus
$$
\widehat{\eta}_0 = 2 \sum_{j=1}^2 \left( x_j\, dy_j - y_j\, dx_j \right),
$$
giving $\displaystyle d\widehat{\eta}_0 = 4\sum_{j=1}^2 dx_j \wedge dy_j$,
which clearly tames~$i$.  Now Lemma~\ref{lemma:matchingJets} implies
that $\widehat{\eta}_\tau^0$ and $\widehat{\eta}_0$ have the same $1$-jet at~$p$,
hence $d\widehat{\eta}_\tau^0|_p = d\widehat{\eta}_0|_p$,  and the claim follows.

Next we show that $d\Lambda_\tau^\epsilon$ tames both~$i$ and $J_\tau^\epsilon$
on $\nN(p)$ when $\epsilon$ is positive but small.  Observe first that
$d\Lambda_\tau^0|_p = 0$: indeed, by Lemma~\ref{lemma:matchingJets} this holds
if the $1$-form $\lambda := -d(\varphi_0 \circ \Pi) \circ i$ satisfies
$d\lambda|_p = 0$, and since $\varphi_0 \circ \Pi(z_1,z_2) =
\sum_{j=1}^2 (x_j^2 - y_j^2) + \varphi(0)$,
an explicit computation shows
$$
\lambda = - d(\varphi_0 \circ \Pi) \circ i = 2 \sum_{j=1}^2 \left(
x_j\, dy_j + y_j \, dx_j \right),
$$
which is everywhere closed.  It will now suffice to show that for any
$\tau \in X$ and any nonzero vector $v \in T_p E$, the derivatives
\begin{equation}
\label{eqn:2derivatives}
\left. \frac{d}{d\epsilon} d\Lambda_\tau^\epsilon(v,iv)\right|_{\epsilon=0},
\qquad
\left. \frac{d}{d\epsilon} 
d\Lambda_\tau^\epsilon(v,J_\tau^\epsilon v)\right|_{\epsilon=0}
\end{equation}
are both positive.  Since 
$\left.\frac{\p}{\p\epsilon} \Lambda_\tau^\epsilon\right|_{\epsilon=0} =
\widehat{\eta}_\tau^0$, both of these are equal to
$d\widehat{\eta}_\tau^0(v,iv)$, which is positive by the first claim proved above.
Note that in computing the expression on the right in \eqref{eqn:2derivatives},
derivative of $J_\tau^\epsilon$ with respect to $\epsilon$ does not appear
since $d\Lambda_\tau^0|_p = 0$.
\end{proof}

To summarize this step so far, we have defined a family of almost complex 
structures $\left\{ J_\tau^\epsilon \right\}_{\epsilon \in [0,\epsilon_0],\tau \in X}$
near $E\crit$ such that $J_\tau^0 = J_\tau$ and $F = \varphi \circ \Pi$
is $J_\tau^\epsilon$-convex for $\epsilon > 0$.
Moreover, the Liouville forms
$\Lambda_\tau^\epsilon = -d F \circ J_\tau^\epsilon$ for $\epsilon > 0$ 
can be written as
$\Lambda_\tau^\epsilon = \Lambda_\tau^0 + \epsilon \widehat{\eta}_\tau^\epsilon$,
where $\widehat{\eta}_\tau^\epsilon$ is a smooth family of $1$-forms that define
fiberwise Liouville forms near $E\crit$ and converge as $\epsilon \to 0$ to
a fiberwise Liouville form~$\widehat{\eta}_\tau^0$ such that
$d\widehat{\eta}_\tau^0$ tames both~$i$ and~$J_\tau$.  The main point of this
construction was that it gives rise to a family of contact structures on the
level sets of~$F$: indeed, define on $\nN(p) \setminus \{p\}$ a family of 
co-oriented $2$-plane distributions
$$
\xi_\tau^\epsilon := \ker dF \cap \ker \Lambda_\tau^\epsilon.
$$
These are $J_\tau^\epsilon$-invariant, so the fact that $F$ is
$J_\tau^\epsilon$-convex for $\epsilon > 0$ implies that they are contact
on each level set of~$F$ whenever $\epsilon > 0$.  For $\epsilon=0$ this is
not the case, as
$$
\xi_\tau^0 = VE
$$
is the vertical subbundle of the Lefschetz fibration and thus defines
foliations on the level sets of~$F$.
In the next step, we will use the Liouville forms $\sigma_\tau$ from \S\ref{sec:step1}
to extend $\xi_\tau^\epsilon$ over the rest of $E \setminus 
\left(E\crit \cup E|_{\Crit(\varphi)} \right)$.  To do this we will need
Lemma~\ref{lemma:fibLiouvilleEcrit} below, for which the next two
lemmas are preparation.  In the following, we use the coordinates $\zeta = (z_1,z_2)$
to define Euclidean norms $|v|$ of vectors $v \in TE|_{\nN(p)}$, and keep in mind
that $\nN(p)$ can always be made smaller if necessary.

\begin{lemma}
\label{lemma:Reeb}
There exists a constant $c_1 > 0$ and a family of smooth vector fields 
$R_\tau$ on $\nN(p) \setminus \{p\}$
such that $|R_\tau| \equiv 1$, $dF(R_\tau) \equiv 0$, and 
$$
\Lambda_\tau^0(R_\tau(\zeta)) \ge c_1 |\zeta| \quad
\text{ for all $\zeta \in \nN(p) \setminus \{p\}$.}
$$
\end{lemma}
\begin{proof}
Choose a family of $J_\tau$-invariant Riemannian metrics $g_\tau$
on $\nN(p)$ and let $\nabla^\tau F$ denote the corresponding gradient vector 
fields of~$F$.  Note that since $\tau$ lives in a compact parameter space,
the norms defined via $g_\tau$ are uniformly (with respect to~$\tau$) equivalent
to the Euclidean norm.  
By Lemma~\ref{lemma:matchingJets}, the Hessian of $F$ at~$p$
matches that of $\varphi_0 \circ \Pi(\zeta) = \sum_{j=1}^2 (x_j^2 - y_j^2)
+ \varphi(0)$, thus the critical point of $F$ at $\zeta=0$ is nondegenerate.
It follows that one can find a constant $k > 0$ such that
$$
| \nabla^\tau F(\zeta) | \ge k |\zeta| \quad\text{ for all $\zeta \in \nN(p)$,
$\tau \in X$.}
$$
A family of vector fields with the desired properties can then be defined by
$R_\tau = \frac{J_\tau \nabla^\tau F}{| J_\tau \nabla^\tau F|}$.
\end{proof}

\begin{lemma}
\label{lemma:contactDerivative}
On $\nN(p) \setminus \{p\}$, $d\Lambda_\tau^0|_{\xi_\tau^0} = 0$,
$\left.\Lambda_\tau^0 \wedge d\Lambda_\tau^0\right|_{\ker dF} = 0$, and
$$
\left. \frac{\p}{\p\epsilon} \left(\Lambda_\tau^\epsilon \wedge
d\Lambda_\tau^\epsilon|_{\ker dF} \right)\right|_{\epsilon=0} > 0.
$$
\end{lemma}
\begin{proof}
Since $\ker \left(\Lambda_\tau^0|_{\ker dF}\right) = 
\xi_\tau^0 = VE$, the first two statements are both equivalent to the fact
that $VE$ defines a foliation on every level set of~$F$.
We will now prove that the third claim holds after shrinking the neighborhood
$\nN(p)$ sufficiently.  Recall from the proof of Lemma~\ref{lemma:perturbedJcrit}
that $d\Lambda_\tau^0|_p = 0$, and similarly, $\widehat{\eta}_\tau^0 = -dF \circ Y'_\tau$
vanishes at $p$. Since both are smooth, this implies there is a constant 
$c_2 > 0$ such that
\begin{equation}
\label{eqn:bound1}
\left\| d\Lambda_\tau^0|_\zeta\right\| \le c_2 |\zeta|, \quad
\left\| \widehat{\eta}_\tau^0|_\zeta\right\| \le c_2 |\zeta| \qquad
\text{ for $\zeta \in \nN(p)$,}
\end{equation}
where we denote by $\|\cdot\|$ the natural norm induced on tensors from the
Euclidean norm in the coordinates.
Now for any $\zeta \in \nN(p) \setminus \{p\}$, fix $v \in T_\zeta E$ with $v \in VE$
and $|v| = 1$, so the vector $iv \in T_\zeta E$ is also vertical and also has norm~$1$.
Denote the value at $\zeta$ of the vector field from Lemma~\ref{lemma:Reeb} by
$R := R_\tau(\zeta) \in \ker dF|_\zeta$, which according to the lemma, satisfies
\begin{equation}
\label{eqn:bound2}
\Lambda_\tau^0(R) \ge c_1 |\zeta|
\end{equation}
for some constant $c_1 > 0$ independent of $\zeta$ and~$\tau$.  The triple $(R,v,iv)$ now form a
positively oriented basis of $\ker dF|_\zeta$, and
$\Lambda_\tau^\epsilon \wedge d\Lambda_\tau^\epsilon(R,v,iv)$ is proportional to
$$
\Lambda_\tau^\epsilon(R)\, d\Lambda_\tau^\epsilon(v,iv) + 
\Lambda_\tau^\epsilon(v)\, d\Lambda_\tau^\epsilon(iv,R) +
\Lambda_\tau^\epsilon(iv)\, d\Lambda_\tau^\epsilon(R,v).
$$
Differentiating this with respect to $\epsilon$ and setting $\epsilon=0$, three terms
drop out since $\Lambda_\tau^0(v) = \Lambda_\tau^0(iv) = d\Lambda_\tau^0(v,iv) = 0$, 
and we are left with
\begin{multline*}
\Lambda_\tau^0(R)\, d\widehat{\eta}_\tau^0(v,iv) + \widehat{\eta}_\tau^0(v)\, 
d\Lambda_\tau^0(iv,R) + \widehat{\eta}_\tau^0(iv)\, d\Lambda_\tau^0(R,v) \\
 \ge c_1 |\zeta| \, d\widehat{\eta}_\tau^0(v,iv) - 2 c_2^2 |\zeta|^2
 = |\zeta| \cdot \left( c_1 \, d\widehat{\eta}_\tau^0(v,iv) - 2 c_2^2 |\zeta| \right),
\end{multline*}
where we've used \eqref{eqn:bound1} to bound the magnitude of the
last two terms from above
and \eqref{eqn:bound2} to bound the first from below.  Since
$d\widehat{\eta}_\tau^0$ tames~$i$ and $|v|=1$, the term 
$d\widehat{\eta}_\tau^0(v,iv)$ satisfies a uniform positive lower bound on
$\nN(p)$, thus the entire expression becomes positive as soon as
$|\zeta|$ is sufficiently small.
\end{proof}

\begin{lemma}
\label{lemma:fibLiouvilleEcrit}
There exists a family of smooth $1$-forms $\eta_\tau^\epsilon$ on
$\nN(p) \setminus \{p\}$, for $\tau \in X$ and $\epsilon \in [0,\epsilon_0]$,
such that $d\eta_\tau^0|_{VE} > 0$
and
$$
\xi_\tau^\epsilon = 
\ker dF \cap \ker \left( \Pi^*\sigma_\tau + \epsilon \eta_\tau^\epsilon \right).
$$
Moreover, the $\eta_\tau^\epsilon$ decay (uniformly in $\tau$ and $\epsilon$)
to zero at~$p$.
\end{lemma}
\begin{proof}
On $\nN(p) \setminus \{p\}$, the kernels of $\Lambda_\tau^0$ and
$\Pi^*\sigma_\tau$ restricted to level sets of $F$ are both
$\xi_\tau^0 = VE$, thus there is a family of smooth positive functions
$g_\tau : \nN(p) \setminus \{p\} \to (0,\infty)$ such that
\begin{equation}
\label{eqn:gtau}
\left.\Pi^*\sigma_\tau\right|_{\ker dF} = \left. g_\tau \Lambda_\tau^0
\right|_{\ker dF}.
\end{equation}
We can plug in the unit vector field $R_\tau$ from Lemma~\ref{lemma:Reeb}
to compute $g_\tau$ in coordinates, and since $\Pi^*\sigma_\tau$ 
vanishes at~$p$, the estimate in the lemma gives rise to an estimate
$$
|g_\tau(\zeta)| = \frac{\left| \Pi^*\sigma_\tau(R_\tau(\zeta))\right|}{|\Lambda^0_\tau(R_\tau(\zeta))|}
\le \frac{c_2 |\zeta|}{c_1 |\zeta|} = \frac{c_2}{c_1}
$$
for some constants $c_1,c_2 > 0$, so that the functions $g_\tau$ are uniformly
bounded near~$p$.

The relation \eqref{eqn:gtau} together with \eqref{eqn:LambdaTauEps} now implies
$$
\left. g_\tau \Lambda_\tau^\epsilon\right|_{\ker dF} =
\left.\left(\Pi^*\sigma_\tau + \epsilon g_\tau \widehat{\eta}_\tau^\epsilon
\right)\right|_{\ker dF},
$$
thus we can set
$$
\eta_\tau^\epsilon := g_\tau \widehat{\eta}_\tau^\epsilon
\quad\text{ for $\epsilon \in [0,\epsilon_0]$, $\tau \in X$.}
$$
Observe that $\widehat{\eta}_\tau^\epsilon|_p = 0$ for all $\tau$
and $\epsilon$ by definition, so the boundedness of $g_\tau$ implies
that $\eta_\tau^\epsilon$ also decays uniformly to zero at~$p$.

Our remaining task is to show that $d\eta_\tau^0|_{VE} > 0$ in some
(possibly smaller) open set of the form $\nN(p) \setminus \{p\}$.
Since $\left. g_\tau \Lambda_\tau^\epsilon \wedge d\left( g_\tau
\Lambda_\tau^\epsilon \right)\right|_{\ker dF} =
\left. g_\tau^2 \, \Lambda_\tau^\epsilon \wedge
d\Lambda_\tau^\epsilon\right|_{\ker dF}$, Lemma~\ref{lemma:contactDerivative}
implies that on a sufficiently small neighborhood $\nN(p) \setminus \{p\}$,
\begin{equation*}
\begin{split}
0 &< \left. g_\tau^2 \frac{\p}{\p\epsilon}\left(\left.
\Lambda_\tau^\epsilon \wedge d\Lambda_\tau^\epsilon\right|_{\ker dF}\right)\right|_{\epsilon = 0} \\
&= \left. \frac{\p}{\p\epsilon} \left[ \left. (\Pi^*\sigma_\tau + \epsilon \eta_\tau^\epsilon)
\wedge d(\Pi^*\sigma_\tau + \epsilon \eta_\tau^\epsilon)\right|_{\ker dF} \right]\right|_{\epsilon=0} \\
&= \left. \frac{\p}{\p\epsilon} \left[ \epsilon \cdot \left. (\Pi^*\sigma_\tau + \epsilon \eta_\tau^\epsilon)
\wedge d\eta_\tau^\epsilon \right|_{\ker dF} \right]\right|_{\epsilon=0} 
= \left.\Pi^*\sigma_\tau \wedge d\eta_\tau^0\right|_{\ker dF},
\end{split}
\end{equation*}
where we've used the fact that $\Pi^*\sigma_\tau$ is closed on the level 
sets of $F$ since $\sigma_\tau$ is (obviously) closed on the level sets of~$\varphi$.
The kernel of $\Pi^*\sigma|_{\ker dF}$ is $VE$, so this last relation is 
equivalent to $d\eta_\tau^0|_{VE} > 0$.
\end{proof}

\subsection{Perturbing from Levi flat to contact}
\label{sec:step3}

By assumption, there is a subcomplex $A \subset X$ and a family of
smooth functions $\{f_\tau : E \to \RR \}_{\tau \in A}$ such that
$\lambda_\tau := -d f_\tau \circ J_\tau$
are Liouville forms and restrict to $\p E$ as 
Giroux forms in the sense of Remark~\ref{remark:corner}.  Recall from
Definition~\ref{defn:fiberwiseJconvex} the convex space
$\JconvexFib_{(J,\p_\theta)}(\Pi)$ of fiberwise $J$-convex functions
associated to each $(J,\p_\theta) \in \AC(\Pi)$.  It will be convenient to
observe that this definition still makes sense and 
$\JconvexFib_{(J,\p_\theta)}(\Pi)$ is
still convex if $(J,\p_\theta)$ is only assumed to belong
to $\ACweaker(\Pi;\uU)$.  For example, $f_\tau \in \JconvexFib_{(J_\tau,\p_\theta^\tau)}(\Pi)$
for each $\tau \in A$.

\begin{lemma}
\label{lemma:asyetunspecified}
The family of functions $\{f_\tau : E \to \RR\}_{\tau \in A}$ can be extended
to a family $\{f_\tau : E \to \RR\}_{\tau \in X}$ such that
$f_\tau \in \JconvexFib_{(J_\tau,\p_\theta^\tau)}(\Pi)$ for every $\tau \in X$.
\end{lemma}
\begin{proof}
Independently of the given functions~$f_\tau$, we first observe that there exists
a family $\{g_\tau \in \JconvexFib_{(J_\tau,\p_\theta^\tau)}(\Pi)\}_{\tau \in X}$.
In light of Remark~\ref{remark:notHolomorphic}, this follows from the 
partition of unity argument in the proof of Proposition~\ref{prop:nonempty};
the only meaningful difference is that one needs to consider families depending
continuously on $\tau$ at every step, though since $X$ is compact, one can
also use Lemmas~\ref{lemma:sufficientlyConvex} and~\ref{lemma:jconvexNearCritical}
to construct $g_\tau$ so that it is independent of $\tau$ away from~$\p_h E$.
This establishes the lemma in the case $A = \emptyset$.

To solve the extension
problem in general, it suffices to consider the case where $X$ is a disk
$\DD^k$ and $A = \p\DD^k = S^{k-1}$ for some $k \in \NN$.  We start by
extending the given family $\{f_\tau\}_{\tau \in A}$ arbitrarily to a 
family of smooth functions
$\widehat{f}_\tau : E \to \RR$ for $\tau \in X$ such that each
$\widehat{f}_\tau|_{\p_h E}$ is invariant under the $S^1$-action defined by
the flow of $\p_\theta^\tau$ and the normal derivatives 
$d \widehat{f}_\tau(-J_\tau \p_\theta^\tau)$ are locally constant for each~$\tau$.
Since each $J_\tau$ has an $S^1$-invariant restriction to $\p_h E$, the 
$1$-forms $\alpha_\tau := - d\widehat{f}_\tau \circ J_\tau|_{T(\p_h E)}$ 
are also $S^1$-invariant and thus satisfy
$$
0 = \Lie_{\p_\theta^\tau} \alpha_\tau = d\left( \alpha_\tau(\p_\theta^\tau)\right) +
d\alpha_\tau(\p_\theta^\tau,\cdot).
$$
In this expression, the first term at the right vanishes since 
$\alpha_\tau(\p_\theta^\tau) = - d\widehat{f}_\tau(J_\tau \p_\theta^\tau)$
is locally constant, thus $d\alpha_\tau(\p_\theta^\tau,\cdot) \equiv 0$.  The remaining
conditions in the definition of a fiberwise $J_\tau$-convex function are all
open, so it follows that $\widehat{f}_\tau$ is also fiberwise $J_\tau$-convex for every
$\tau$ in some open neighborhood $A' \subset X$ of~$A$.  Finally, choose a
cutoff function $\beta : X \to [0,1]$ that is supported in $A'$ and satisfies
$\beta|_A \equiv 1$, and set
$f_\tau := \beta(\tau) \widehat{f}_\tau + [1 - \beta(\tau)] g_\tau$.
This is fiberwise $J_\tau$-convex for every $\tau \in X$ since 
the space $\JconvexFib_{(J_\tau,\p_\theta^\tau)}(\Pi)$ is convex.
\end{proof}

From now on, denote by
$$
\lambda_\tau = - d f_\tau \circ J_\tau, \qquad \tau \in X
$$
the family of fiberwise Liouville forms on $E$ that arise from the above lemma.
Choose a neighborhood $\nN(E\crit) \subset E$ of $E\crit$ such that
the $1$-forms $\eta_\tau^\epsilon$ from Lemma~\ref{lemma:fibLiouvilleEcrit}
are defined and satisfy $d\eta_\tau^0|_{VE} > 0$ on an open neighborhood
of $\overline{\nN(E\crit)} \setminus E\crit$.  
For any smaller neighborhood $\nN'(E\crit)$ of $E\crit$ with compact closure in
$\nN(E\crit)$, we can choose $\epsilon_0 > 0$ small enough so that
$$
d\eta_\tau^\epsilon|_{VE} > 0 \quad
\text{ on $\nN(E\crit) \setminus \nN'(E\crit)$ 
for all $\epsilon \in [0,\epsilon_0]$}.
$$
The following lemma should be understood to be true after a possible further
shrinking of the neighborhoods $\nN(E\crit)$ and $\nN'(E\crit)$ together with
the number $\epsilon_0 > 0$.

\begin{lemma}
\label{lemma:LiouvilleExtension}
The family of $1$-forms $\eta_\tau^\epsilon$ on $\nN(E\crit) \setminus
\nN'(E\crit)$ for $\tau \in X$ and $\epsilon \in [0,\epsilon_0]$
can be extended over $E \setminus \nN'(E\crit)$ so that
$d\eta_\tau^\epsilon|_{VE} > 0$ everywhere and $\eta_\tau^\epsilon = \lambda_\tau$
near $\p_h E$ and in $E|_{\uU'}$ for some neighborhood $\uU' \subset \Sigma$ of 
$\Crit(\varphi) \cup \p\Sigma$.
\end{lemma}
\begin{proof}
We again use a variant of the partition of unity argument from
Proposition~\ref{prop:nonempty}.  
For step~1, choose an open neighborhood $\uU' \subset \Sigma$ of
$\Crit(\varphi) \cup \p\Sigma$ with closure in~$\uU$, and for each
regular value $z \in \Sigma \setminus \uU'$ of~$\Pi$, choose a neighborhood
$\uU_z \subset \Sigma \setminus \Sigma\crit$ of $z$ together with a
family of $1$-forms $\eta_{\tau,z}^\epsilon$ on $E|_{\uU_z}$ such that
$d\eta_{\tau,z}^\epsilon$ is positive on fibers and 
$\eta_{\tau,z}^\epsilon = \lambda_\tau$ near~$\p_h E$.
We can arrange this moreover so that $\uU_z$ is disjoint from
$\widebar{\uU'}$ whenever $z \not\in \widebar{\uU'}$ and $\uU_z \subset \uU$ for
$z \in \widebar{\uU'}$, which permits the choice
$\eta_{\tau,z}^\epsilon := \lambda_\tau$ in the latter case.

Step~2 is to define $\eta_{\tau,z}^\epsilon$ on $E|_{\uU_z}$ for a neighborhood
$\uU_z \subset \Sigma$ of each $z \in \Sigma\crit$, matching
the given $\eta_\tau^\epsilon$ near~$E\crit$.  We start by extending
$\eta_\tau^\epsilon$ over each component of $E_z \setminus E_z\crit$ as a
Liouville form, which is possible since $\Pi$ is allowable, though there is
a slightly subtle point if we want to arrange $\eta_{\tau,z}^\epsilon = \lambda_\tau$
near~$\p_h E$: Stokes' theorem may make this impossible if there are vanishing
cycles $C \subset E_z$ near $E\crit$ on which $\int_C \eta_{\tau,z}^\epsilon$
is too large.  Recall however
that while $\eta_\tau^\epsilon$ may fail to be smooth at~$E\crit$, it does
have a (uniformly in $\tau$ and $\epsilon$) continuous extension that
vanishes at~$E\crit$, so its integrals along cycles in $\nN(E\crit)$ can be
assumed arbitrarily small if we replace $\nN(E\crit)$ by a suitably smaller
neighborhood (which may necessitate making $\nN'(E\crit)$ and $\epsilon_0$
smaller as well).  With this understood, the required extension of
$\eta_\tau^\epsilon$ from $\nN(E\crit)$ to a fiberwise Liouville form 
$\eta_{\tau,z}^\epsilon$ on $E|_{\uU_z}$ exists for some neighborhood
$\uU_z \subset \Sigma \setminus \widebar{\uU'}$ of~$z$.

Step~3 is then to choose a finite subcover $\{\uU_z\}_{z \in I}$ of
$\Sigma \setminus \uU'$ with a subordinate partition of unity
$\{\rho_z : \uU_z \to [0,1]\}_{z \in I}$ and define the desired extension by
$\eta_\tau^\epsilon = \sum_{z \in I} (\rho_z \circ \Pi) \eta_{\tau,z}^\epsilon$.
\end{proof}

Since the smaller neighborhood $\uU' \subset \uU$ in the above lemma contains
$\Crit(\varphi) \cup \p\Sigma$, we can now relabel $\uU'$ as $\uU$ without
loss of generality, and let $\uU' \subset \uU$ denote a still smaller 
neighborhood of $\Crit(\varphi) \cup \p\Sigma$ with closure in~$\uU$.
Choose a smooth cutoff function 
$$
\beta : \Sigma \to [0,1]
$$
with compact support in $\uU$ such that $\beta|_{\uU'} \equiv 1$, and define
from this a family of smooth functions 
$$
F_\tau^\epsilon =  : \varphi \circ \Pi + \epsilon(\beta \circ \Pi) f_\tau :
E \to \RR
$$
for $\epsilon \in [0,\epsilon_0]$ and $\tau \in X$.  Observe that
$F_\tau^\epsilon = F = \varphi \circ \Pi$ outside of $E|_{\uU}$.
On $E \setminus \nN'(E\crit)$, we can also define the family of smooth $1$-forms
$$
\Theta_\tau^\epsilon = \Pi^*\sigma_\tau + \epsilon \eta_\tau^\epsilon,
$$
which are Liouville for $\epsilon > 0$ sufficiently small
(cf.~Proposition~\ref{prop:Thurston}).

\begin{lemma}
\label{lemma:WeinsteinFamily}
For all $\epsilon > 0$ sufficiently small and all $\tau \in X$,
$F_\tau^\epsilon$ is $J_\tau$-convex on~$E|_{\uU'}$, and
$$
d F_\tau^\epsilon \wedge \Theta_\tau^\epsilon \wedge d\Theta_\tau^\epsilon > 0
\quad\text{ on $E \setminus \left(\nN'(E\crit) \cup E|_{\Crit(\varphi)}\right)$}.
$$
\end{lemma}
\begin{proof}
Consider first the region $E|_{\uU'}$.  Here
$\Pi$ is $J_\tau$-$j_\tau$-holomorphic and $\beta \circ \Pi \equiv 1$, thus
$$
- d F_\tau^\epsilon \circ J_\tau = \Pi^*(-d\varphi \circ j_\tau) +
\epsilon (-d f_\tau \circ J_\tau) = \Pi^*\sigma_\tau + \epsilon \lambda_\tau =
\Theta_\tau^\epsilon.
$$
Proposition~\ref{prop:ThurstonStein} then implies that $F_\tau^\epsilon$
is $J_\tau$-convex on $E|_{\uU'}$ for $\epsilon > 0$ sufficiently small,
and consequently that $\Theta_\tau^\epsilon$ is
contact on all the regular level sets of $F_\tau^\epsilon$ in this region.

On $E \setminus \left(\nN'(E\crit) \cup E|_{\uU'}\right)$,
we compute
\begin{equation*}
\begin{split}
d F_\tau^\epsilon \wedge \Theta_\tau^\epsilon \wedge d\Theta_\tau^\epsilon &=
\left(\Pi^*d\varphi + \epsilon \, d\left[ (\beta \circ \Pi) f \right] \right)
\wedge
(\Pi^*\sigma_\tau + \epsilon \eta_\tau^\epsilon) \wedge
(\Pi^*d\sigma_\tau + \epsilon\, d\eta_\tau^\epsilon) \\
&= \epsilon \, \Pi^*(d\varphi \wedge \sigma_\tau) \wedge d\eta_\tau^\epsilon
+ \oO(\epsilon^2).
\end{split}
\end{equation*}
This is positive for all $\epsilon > 0$ sufficiently small since
outside of any neighborhood of $\Crit(\varphi)$ and of
$E\crit$ respectively, $d\varphi \wedge \sigma_\tau$ and
$d\eta_\tau^\epsilon|_{VE}$ can each be assumed to satisfy uniform positive
lower bounds; note that the latter depends on the fact that
$d\eta_\tau^\epsilon|_{VE} > 0$ holds even for $\epsilon=0$
(cf.~Lemma~\ref{lemma:fibLiouvilleEcrit}).
\end{proof}

In light of Lemma~\ref{lemma:fibLiouvilleEcrit}, the family of 
$2$-plane distributions $\xi_\tau^\epsilon$ can now be extended from
$\nN(E\crit)$ over the entirety of 
$E \setminus \left(E\crit \cup E|_{\Crit(\varphi)}\right)$ by setting
$$
\xi_\tau^\epsilon = \ker dF_\tau^\epsilon \cap \ker \Theta_\tau^\epsilon.
$$
Lemma~\ref{lemma:WeinsteinFamily} then implies that for all
$\epsilon > 0$ sufficiently small, $\xi_\tau^\epsilon$ defines a family of
contact structures on the level sets of~$F_\tau^\epsilon$.
Observe that by construction, $\xi_\tau^\epsilon$ is also preserved by
$J_\tau$ on the neighborhood $E|_{\uU'}$ of $E|_{\Crit(\varphi)} \cup \p_v E$,
and it is preserved by $J_\tau^\epsilon$ in $\nN(E\crit)$.

\begin{lemma}
\label{lemma:perturbedJ}
After possibly shrinking $\epsilon_0 > 0$,
the family $J_\tau^\epsilon$ defined near $E\crit$ in \S\ref{sec:step2} for 
$\epsilon \in [0,\epsilon_0]$ and $\tau \in X$
can be extended
to a family of global almost complex structures on $E$ that depend smoothly
on~$\epsilon$, preserve $\xi_\tau^\epsilon$, and satisfy $J_\tau^\epsilon = J_\tau$
in some fixed neighborhood of $E|_{\Crit(\varphi)} \cup \p_v E$ for all $\epsilon$
and $J_\tau^0 \equiv J_\tau$.  Moreover,
\begin{equation}
\label{eqn:Gtaueps}
- d F_\tau^\epsilon \circ J_\tau^\epsilon = G_\tau^\epsilon \Theta_\tau^\epsilon
\quad \text{ along } \quad \p_h E,
\end{equation}
for a (uniquely determined) family of functions 
$G_\tau^\epsilon : \p_h E \to (0,\infty)$ which depend smoothly on
$\epsilon$ and satisfy $G_\tau^0 \equiv 1$.
\end{lemma}
\begin{proof}
Pick an open set $E^{\reg} \subset E$ with closure disjoint from
$E\crit \cup E|_{\Crit(\varphi)} \cup \p_v E$ such that
$$
E = E^{\reg} \cup E|_{\uU'} \cup \nN(E\crit).
$$
Choose also a family $g_\tau$ of $J_\tau$-invariant Riemannian metrics on~$E$
and let $H_\tau E^{\reg} \subset TE^{\reg}$ denote the $g_\tau$-orthogonal complement
of~$VE|_{E^{\reg}}$.  By construction, $H_\tau E^{\reg}$ is $J_\tau$-invariant, and
since $\xi_\tau^0 = VE$, we are free to assume
$H_\tau E^{\reg} \pitchfork \xi_\tau^\epsilon$ whenever $\epsilon \ge 0$ is
sufficiently small.  Then the projections $TE^{\reg} \to VE$
along $H_\tau E$ restrict to a family of bundle isomorphisms 
$\Psi_\tau : \xi_\tau^\epsilon \to VE$, and
there is a unique family of almost complex structures 
$\widehat{J}_\tau^\epsilon$ on $E^{\reg}$ defined by the conditions
$$
\widehat{J}_\tau^\epsilon|_{\xi_\tau^\epsilon} = 
\Psi_\tau^*J_\tau|_{VE}, \qquad
\widehat{J}_\tau^\epsilon|_{H_\tau E} = J_\tau|_{H_\tau E}.
$$
These preserve $\xi_\tau^\epsilon$ and match $J_\tau$ for
$\epsilon=0$.

We next splice $\widehat{J}_\tau^\epsilon$ together with the existing
families $J_\tau^\epsilon$ on $\nN(E\crit)$ and $J_\tau^\epsilon := J_\tau$
on~$E|_{\uU'}$.  For any point $p \in E$ with
a complex structure $J$ on $T_p E$ and sufficiently small $J$-antilinear
map $Y : T_p E \to T_p E$, define
$$
\Phi_J(Y) = \left( \1 + \frac{1}{2} J Y \right) J \left( \1 + \frac{1}{2} J Y
\right)^{-1}.
$$
This identifies a neighborhood of~$0$ in the space of $J$-antilinear maps
on $T_p E$ with a neighborhood of~$J$ in the manifold of complex structures
on~$T_p E$, and moreover, if $J$ preserves some subspace $V \subset T_p E$,
then $\Phi_J(Y)$ also preserves $V$ if and only if $Y$ preserves~$V$.
On $E^{\reg} \cap \left(\nN(E\crit) \cup E|_{\uU'}\right)$, we may assume
for sufficiently small $\epsilon \ge 0$ that
$\widehat{J}_\tau^\epsilon$ and $J_\tau^\epsilon$ are each $C^0$-close
to $J_\tau$ and therefore also to each other, so there exists a family
of $\xi_\tau^\epsilon$-preserving $J_\tau^\epsilon$-antilinear bundle 
endomorphisms $Y_\tau^\epsilon$ such that
$$
\widehat{J}_\tau^\epsilon = \Phi_{J_\tau^\epsilon}\left(Y_\tau^\epsilon\right),
$$
and $Y_\tau^0 \equiv 0$.
Now for any choice of smooth function 
$\psi : E \to [0,1]$ 
that equals $1$ outside $E|_{\uU'} \cup \nN(E\crit)$ and has compact support in~$E^{\reg}$,
a family of almost complex structures satisfying most of the desired properties
can be defined by
\begin{equation}
\label{eqn:cutoffJ}
J_\tau^\epsilon :=
\begin{cases}
\widehat{J}_\tau^\epsilon & \text{ on $E^{\reg} \setminus \left(\nN(E\crit) \cup
E|_{\uU'} \right)$},\\
J_\tau^\epsilon & \text{ on $\left(\nN(E\crit) \cup E|_{\uU'}\right) \setminus
E^{\reg}$}, \\
\Phi_{J_\tau^\epsilon}\left(\psi Y_\tau^\epsilon\right) & \text{ on
$E^{\reg} \cap \left(\nN(E\crit) \cup E|_{\uU'} \right)$}.
\end{cases}
\end{equation}
Notice that on the region where $\Pi$ is holomorphic, $\beta = 1$ and $J_\tau^\epsilon = J_\tau$,
we have
$$
- d F_\tau^\epsilon \circ J_\tau^\epsilon = - d (\varphi \circ \Pi + \epsilon f_\tau) \circ J_\tau =
\Pi^*\sigma_\tau + \epsilon \lambda_\tau = \Theta_\tau^\epsilon.
$$
This applies in particular on a neighborhood of $E|_{\Crit(\varphi)} \cup \p_v E$,
so that \eqref{eqn:Gtaueps} is already established with $G_\tau^\epsilon = 1$
near $E|_{\Crit(\varphi)}$.  In order to achieve \eqref{eqn:Gtaueps} everywhere
else, we can modify the definition of $J_\tau^\epsilon$ near $\p_h E$ on
a subbundle transverse to~$\xi_\tau^\epsilon$.  Indeed, observe first that
on $\p_h E$, the relation $T\Pi \circ J_\tau = j_\tau \circ T\Pi$ implies
$-dF_\tau^0 \circ J_\tau^0 = \Pi^*\sigma_\tau$, and the latter is nowhere 
zero away from $E|_{\Crit(\varphi)}$, hence so are both $-d F_\tau^\epsilon \circ J_\tau^\epsilon|_{T(\p_h E)}$
and $\Theta_\tau^\epsilon$ for $\epsilon \ge 0$ sufficiently small. Our
goal will thus be to achieve
$$
\ker (-d F_\tau^\epsilon \circ J_\tau^\epsilon) = \ker \Theta_\tau^\epsilon \quad
\text{ along }\quad \p_h E.
$$
Since $\Theta_\tau^\epsilon$
and $-d F_\tau^\epsilon \circ J_\tau^\epsilon$ both annihilate $\xi_\tau^\epsilon$,
it suffices to find a $1$-dimensional subbundle of $TE|_{\p_h E}$ that
intersects $\xi_\tau^\epsilon$ trivially and is also annihilated by both.
For $\Theta_\tau^\epsilon$ there is a clear choice: this $1$-form is Liouville
for sufficiently small $\epsilon > 0$, so its dual Liouville vector field
$V_\tau^\epsilon$ satisfies
$$
\Theta_\tau^\epsilon(V_\tau^\epsilon) = d\Theta_\tau^\epsilon(V_\tau^\epsilon,V_\tau^\epsilon) = 0,
$$
and we will see presently that it is not contained in~$VE$, and therefore also not
in~$\xi_\tau^\epsilon$ for $\epsilon > 0$ small, outside a neighborhood of~$E|_{\Crit(\varphi)}$.
Indeed, working in a 
neighborhood of $\p_h E$ where $\Pi$ has no critical points and $\eta_\tau^\epsilon = \lambda_\tau$, let 
$H_\tau E \subset TE$
denote the $d\lambda_\tau$-symplectic complement of~$VE$.
With respect to this splitting, write $V_\tau^\epsilon = v_\tau^\epsilon + h_\tau^\epsilon$
for $v_\tau^\epsilon \in VE$ and $h_\tau^\epsilon \in H_\tau E$.
Writing $\Theta_\tau^\epsilon = \Pi^*\sigma_\tau + \epsilon \lambda_\tau$ and
restricting the relation $\Theta_\tau^\epsilon = d\Theta_\tau^\epsilon(V_\tau^\epsilon,\cdot)$ 
to the subbundles $VE$ and $H_\tau E$ then gives
$$
\left.\lambda_\tau\right|_{VE} = \left.d\lambda_\tau(v_\tau^\epsilon,\cdot)\right|_{VE}
\quad\text{ and }\quad
\left.(\Pi^*\sigma_\tau + \epsilon \lambda_\tau)\right|_{H_\tau E} =
\left.(\Pi^*d\sigma_\tau + \epsilon \, d\lambda_\tau)(h_\tau^\epsilon,\cdot)\right|_{H_\tau E}.
$$
The first relation identifies $v_\tau^\epsilon$ as the ``vertical Liouville
vector field'' $V_{\lambda_\tau}$, defined on the smooth part of each 
fiber $E_z$ as the Liouville vector field dual to $\lambda_\tau|_{TE_z}$.
In particular, the vertical term does not depend on~$\epsilon$.
The horizontal term $h_\tau^\epsilon$ also has a well-behaved limit 
as $\epsilon \to 0$, determined by
$$
\Pi^*\sigma_\tau|_{H_\tau E} = \Pi^*d\sigma_\tau(h_\tau^0,\cdot)|_{H_\tau E},
$$
which means $h_\tau^0$ is the horizontal lift $V_{\sigma_\tau}^\#$ of the 
Liouville vector field $V_{\sigma_\tau}$ on $\Sigma$, 
defined by $d\sigma_\tau(V_{\sigma_\tau},\cdot) = \sigma_\tau$.  
The latter is nowhere zero away from $\Crit(\varphi)$,
implying that
$$
V_\tau^0 := \lim_{\epsilon \to 0} V_\tau^\epsilon = V_{\lambda_\tau} + V_{\sigma_\tau}^\#
$$
always has a nontrivial horizontal part on the region of interest.  This establishes
the claim that $V_\tau^\epsilon \not\in \xi_\tau^\epsilon$ on this region for all
$\epsilon \ge 0$ sufficiently small.

Now observe that at $\p_h E$ for $\epsilon=0$,
$$
- d F_\tau^0(J_\tau V_\tau^0) = - d \varphi(\Pi_* J_\tau V_{\lambda_\tau} + \Pi_* J_\tau V_{\sigma_\tau}^\#)
= - d\varphi(j_\tau \Pi_*V_{\sigma_\tau}^\#) = \sigma_\tau(V_{\sigma_\tau}) =
d\sigma_\tau(V_{\sigma_\tau},V_{\sigma_\tau}) = 0,
$$
where we've again used the assumption that $T\Pi \circ J_\tau = j_\tau \circ T\Pi$
along~$\p_h E$.  In other words, on $\p_h E$ away from $E|_{\Crit(\varphi)}$, 
$V_\tau^0$ and $J_\tau V_\tau^0$ span a
$J_\tau$-complex subbundle of $TE$ that is transverse to $VE$
and intersects $\ker dF_\tau^0$ transversely in the subspace spanned
by~$J_\tau V_\tau^0$.  At $E|_{\Crit(\varphi)}$, the transversality fails
because $V_{\sigma_\tau}^\#$ vanishes, but since $-d F_\tau^\epsilon\circ J_\tau^\epsilon
= \Theta_\tau^\epsilon$ along $\p_h E$ in this region, we also have
$$
- d F_\tau^\epsilon(J_\tau^\epsilon V_\tau^\epsilon) =
\Theta_\tau^\epsilon(V_\tau^\epsilon) = 0
$$
here.  It is therefore possible to modify the family 
$J_\tau^\epsilon$ near $\p_h E$ without changing it near $E|_{\Crit(\varphi)}$
or changing its action on $\xi_\tau^\epsilon$ anywhere so that it satisfies
$$
J_\tau^\epsilon V_\tau^\epsilon \in \ker d F_\tau^\epsilon
$$
everywhere along $\p_h E$ for $\epsilon \ge 0$ sufficiently small.
This identifies the kernels of $-d F_\tau^\epsilon \circ J_\tau^\epsilon$
and $\Theta_\tau^\epsilon$ along $\p_h E$ and thus establishes
\eqref{eqn:Gtaueps} for a uniquely determined family of functions
$G_\tau^\epsilon : \p_h E \to (0,\infty)$ which necessarily equal $1$
near~$E|_{\Crit(\varphi)}$.  Since both families of $1$-forms match
$\Pi^*\sigma_\tau$ for $\epsilon=0$, we also have $G_\tau^0 \equiv 1$.
The modified family $J_\tau^\epsilon$ can now be spliced together with the
previously constructed family away from $\p_h E$ using the same trick
as in \eqref{eqn:cutoffJ}.
\end{proof}

\begin{lemma}
\label{lemma:almostStein}
After replacing $\varphi : \Sigma \to \RR$ by a function of the form
$h \circ \varphi$ with $h' > 0$ and $h'' \gg 0$,
the pairs $(J_\tau^\epsilon,F_\tau^\epsilon)$ become almost Stein structures
for all $\tau \in X$ and all $\epsilon > 0$ sufficiently small.
\end{lemma}
\begin{proof}
The functions $F_\tau^\epsilon$ have critical points at $E\crit$ and
in $E|_{\Crit(\varphi)}$, but are $J_\tau^\epsilon$-convex near both
due to Lemmas~\ref{lemma:perturbedJcrit} and~\ref{lemma:WeinsteinFamily}.
Outside these neighborhoods, the maximal $J_\tau^\epsilon$-complex subbundles
on the level sets of $F_\tau^\epsilon$ are the contact 
structures~$\xi_\tau^\epsilon$, so $F_\tau^\epsilon$ becomes $J_\tau^\epsilon$-convex
after postcomposition with a sufficiently convex function, using
Lemma~\ref{lemma:sufficientlyConvex}.

It remains to check that $-d F_\tau^\epsilon \circ J_\tau^\epsilon$ restricts
to contact forms on both $\p_v E$ and~$\p_h E$.  The former lies in the
region where $-d F_\tau^\epsilon \circ J_\tau^\epsilon = \Theta_\tau^\epsilon
= \Pi^*\sigma_\tau + \epsilon \lambda_\tau$, and Proposition~\ref{prop:GirouxVertical}
proves that the latter is contact on $\p_v E$ for sufficiently small $\epsilon > 0$
since $\sigma_\tau|_{T(\p\Sigma)} > 0$ and $\lambda_\tau$ is fiberwise
Liouville.  Using Proposition~\ref{prop:GirouxHorizontal} similarly, 
$\Theta_\tau^\epsilon$ is also contact
on $\p_h E$ for small $\epsilon > 0$, so the contact condition on $\p_h E$
follows from \eqref{eqn:Gtaueps}.
\end{proof}

\subsection{Interpolation of almost Stein structures}
\label{sec:step4}

To complete the proof of Proposition~\ref{prop:parametricJ}, we need to relate
the family of almost Stein structures 
$\{(J_\tau^\epsilon,F_\tau^\epsilon)\}_{\tau \in X,\, \epsilon \in (0,\epsilon_0]}$
constructed above to the given family $\{(J_\tau,f_\tau)\}_{\tau \in A}$.
The functions $f_\tau$ where extended to all $\tau \in X$ in Lemma~\ref{lemma:asyetunspecified}
but are only \emph{fiberwise} $J_\tau$-convex in general for $\tau \not\in A$;
on the other hand, all conditions that distinguish $J_\tau$-convexity from its
fiberwise counterpart are open, thus we can assume $(J_\tau,f_\tau)$
are almost Stein structures for all $\tau$ in some open neighborhood
$A' \subset X$ of~$A$.  The same can also be assumed for 
$(J_\tau^\epsilon,f_\tau)$ for any $\epsilon \in [0,\epsilon_0]$ if
$\epsilon_0 > 0$ is sufficiently small.
Now choose a cutoff function $\rho : X \to [0,1]$
with support in $A'$ and $\rho|_A \equiv 1$, and consider the family
of interpolated functions
$$
f_\tau^\epsilon := \rho(\tau) f_\tau + [1 - \rho(\tau)] F_\tau^\epsilon
$$
for $\tau \in X$ and $\epsilon \in [0,\epsilon_0]$.  These functions are
$J_\tau^\epsilon$-convex everywhere when $\epsilon > 0$, but we still need
to check that the remaining conditions of an almost Stein structure are
satisfied for $\tau \in A' \setminus A$, i.e.~that the interpolated
Liouville forms
$$
-d f_\tau^\epsilon \circ J_\tau^\epsilon = \rho(\tau) 
\left( -d f_\tau \circ J_\tau^\epsilon\right) + [1-\rho(\tau)]
\left( -d F_\tau^\epsilon \circ J_\tau^\epsilon \right)
$$
are contact on both faces of~$\p E$.  After shrinking $\epsilon_0 > 0$
further if necessary, this will follow from the next two lemmas.

\begin{lemma}
\label{lemma:interpVertical}
For all $\tau \in A'$, $\epsilon > 0$ sufficiently small and $\rho \in [0,1]$,
the $1$-forms $\rho (-d f_\tau \circ J_\tau^\epsilon) + (1-\rho) (-d F_\tau^\epsilon \circ J_\tau^\epsilon)$
restrict to contact forms on~$\p_v E$.
\end{lemma}
\begin{proof}
In a neighborhood of $\p_v E$, we have $J_\tau^\epsilon = J_\tau$ and thus
$-d f_\tau \circ J_\tau^\epsilon = \lambda_\tau$, and similarly,
$F_\tau^\epsilon = \varphi \circ \Pi + \epsilon f_\tau$ and
$T\Pi \circ J_\tau = j_\tau \circ T\Pi$ imply $-d F_\tau^\epsilon \circ J_\tau^\epsilon
= \Pi^*\sigma_\tau + \epsilon \lambda_\tau$.  The $1$-form in question is thus
$$
\rho \lambda_\tau + (1-\rho)(\Pi^*\sigma_\tau + \epsilon \lambda_\tau) 
= [\rho + \epsilon(1-\rho)] \left( \lambda_\tau + \frac{1-\rho}{\rho + \epsilon(1-\rho)} \Pi^*\sigma_\tau\right),
$$
assuming $\epsilon > 0$ so that $\rho + \epsilon(1-\rho) > 0$ for all~$\rho$.
Since $\lambda_\tau$ is fiberwise Liouville and defines a contact form on $\p_v E$ and
$\sigma_\tau|_{T(\p\Sigma)} > 0$, the expression in parentheses is contact for
all $\rho \in [0,1]$ by Proposition~\ref{prop:GirouxVertical}.
\end{proof}

\begin{lemma}
\label{lemma:interpHorizontal}
The statement of Lemma~\ref{lemma:interpVertical} also holds for the restriction
to~$\p_h E$.
\end{lemma}
\begin{proof}
Near $\p_h E$, Lemma~\ref{lemma:perturbedJ} gives $-d F_\tau^\epsilon \circ 
J_\tau^\epsilon = G_\tau^\epsilon \Theta_\tau^\epsilon$ for a family of
functions $G_\tau^\epsilon : \p_h E \to (0,\infty)$ satisfying 
$G_\tau^0 \equiv 1$, while the $1$-form $\Theta_\tau^\epsilon =
\Pi^*\sigma_\tau + \epsilon \lambda_\tau$ is contact for $\epsilon > 0$
sufficiently small due to Prop.~\ref{prop:GirouxHorizontal}.  
For $\epsilon=0$, the interpolated $1$-forms in question are thus
$$
\rho \lambda_\tau + (1-\rho) \Pi^*\sigma_\tau = \rho \left( \lambda_\tau +
\frac{1-\rho}{\rho} \Pi^*\sigma_\tau \right)
$$
along $\p_h E$, and these are contact for all $\rho > 0$ by another application
of Prop.~\ref{prop:GirouxHorizontal} since $\sigma_\tau$ is
Liouville and $\lambda_\tau|_{T(\p_h E)}$ satisfies the conditions of a Giroux 
form.  The contact condition ceases to hold for $\rho=\epsilon=0$, but since the
condition is open, the lemma will follow from the claim that
$\rho (-d f_\tau \circ J_\tau^\epsilon) + (1-\rho) (-d F_\tau^\epsilon \circ J_\tau^\epsilon)$
restricted to $\p_h E$ is contact for every $(\rho,\epsilon)$ in a neighborhood
of $(0,0)$ excluding~$(0,0)$ itself.  To see this, note first that since
$\frac{1}{G_\tau^0} (-d f_\tau \circ J_\tau^0) = \lambda_\tau$ and everything
depends smoothly on $\epsilon$, we can write
$$
\frac{1}{G_\tau^\epsilon} (-d f_\tau \circ J_\tau^\epsilon) = \lambda_\tau +
\epsilon \gamma_\tau^\epsilon
$$
for some family of smooth $1$-forms 
$\{\gamma_\tau^\epsilon\}_{\tau \in X,\, \epsilon \in [0,\epsilon_0]}$.
Our family of interpolated $1$-forms on $\p_h E$ can then be rewritten whenever
$(\rho,\epsilon) \ne (0,0)$ as
\begin{equation*}
\begin{split}
\rho\left( -d f_\tau \circ J_\tau^\epsilon\right) + (1-\rho) \left( - d F_\tau^\epsilon
\circ J_\tau^\epsilon \right) &= \rho G_\tau^\epsilon \left( \lambda_\tau +
\epsilon \gamma_\tau^\epsilon \right) + (1-\rho) G_\tau^\epsilon \left(
\Pi^*\sigma_\tau + \epsilon \lambda_\tau \right) \\
&= G_\tau^\epsilon \left( c \,\Pi^*\sigma_\tau + a \lambda_\tau + \epsilon \rho \gamma_\tau^\epsilon \right)
=: G_\tau^\epsilon \mu,
\end{split}
\end{equation*}
where we are abbreviating $c = c(\rho) := 1-\rho$ and $a = a(\rho,\epsilon) :=
\rho + \epsilon(1-\rho)$.  Notice that while $c(\rho)$ approaches~$1$, 
$a(\rho,\epsilon)$ and $\epsilon \rho / a(\rho,\epsilon)$ each decay to $0$ as
$(\rho,\epsilon) \to (0,0)$.
Since $\Pi^*\sigma_\tau \wedge d\lambda_\tau = 0$ and
$\lambda_\tau \wedge \Pi^*d\sigma_\tau > 0$ by the fiberwise Giroux condition
on~$\lambda_\tau$, we then find that
\begin{equation*}
\begin{split}
\mu \wedge d\mu &= \left(c\, \Pi^*\sigma_\tau + a \lambda_\tau +
\epsilon \rho \gamma_\tau^\epsilon \right) \wedge \left( c\, \Pi^*d\sigma_\tau +
a\, d\lambda_\tau + \epsilon \rho \, d\gamma_\tau^\epsilon \right) \\
&= a c \left[ \lambda_\tau \wedge \Pi^*d\sigma_\tau + \frac{a}{c} \lambda_\tau
\wedge d\lambda_\tau + \frac{\epsilon \rho}{c} \left( \lambda_\tau \wedge d\gamma_\tau^\epsilon
+ \gamma_\tau^\epsilon \wedge d\lambda_\tau + \frac{\epsilon \rho}{a}
\gamma_\tau^\epsilon \wedge d\gamma_\tau^\epsilon \right) \right]
\end{split}
\end{equation*}
is positive as soon as $(\rho,\epsilon)$ gets close enough to~$(0,0)$.
\end{proof}

With this, the pairs $(J_\tau^\epsilon,f_\tau^\epsilon)$ for all $\tau \in X$
and $\epsilon > 0$ sufficiently small are seen to be
almost Stein structures that match $(J_\tau,f_\tau)$ for $\tau \in A$,
so the proof of Proposition~\ref{prop:parametricJ} (and therefore also of
Theorem~\ref{thm:SteinHomotopy}) is now complete.

\section{A symplectic model of a collar neighborhood with corners}
\label{sec:model}

Throughout this section, assume $(M',\xi)$ is a
closed connected contact $3$-manifold, and $M \subset M'$ is a
compact connected $3$-dimensional submanifold $M \subset M'$, possibly 
with boundary, on which $\xi$ is supported by a spinal open book
$$
\boldsymbol{\pi} := \Big(\pi\spine : M\spine \to \Sigma,
\pi\paper : M\paper \to S^1, \{m_T\}_{T \subset \p M}\Big).
$$
The immediate purpose of this section is to construct a precise symplectic 
model of a collar neighborhood of the form $(-\epsilon,0] \times M'$ in
the symplectization of $(M',\xi)$, designed such that spine removal cobordisms
can be defined via an easy modification of the model.  The intuition for
the construction comes from the neighborhood of $\p E$ when $\Pi : E \to \Sigma$
is a bordered Lefschetz fibration that fills a spinal open book---however,
it will not be necessary to assume in the following that $(M',\xi)$ is 
symplectically fillable, as we will
instead make use of the trivial observation that \emph{every} closed
contact manifold arises as the convex boundary of a noncompact subset of
its own symplectization. The model we construct
will thus be a noncompact $4$-manifold
$E'$ whose boundary has two smooth faces 
$$
\p E' = \p_v E' \cup \p_h E',
$$
interpreted as the vertical and horizontal boundaries respectively of a
(locally defined) symplectic fibration, such that the smoothed contact boundary
of $E'$ can be identified with~$(M',\xi)$.
We will elaborate further on this
model in \cite{LisiVanhornWendl2} by attaching cylindrical ends to both
its fibers and its base, producing the so-called \emph{double completion}
of~$E'$, which will admit an abundance of holomorphic curves modeled after
the pages of~$\boldsymbol{\pi}$.  These curves generate the moduli space
needed for classifying fillings as in Theorems~\ref{thma:main} 
and~\ref{thma:quasi}.

\subsection{The Liouville collar}
\label{sec:LiouvilleCollar}

\subsubsection{The collar and its boundary}
\label{sec:collarBoundary}

We shall denote the union of the paper with the ``rest'' of $M'$ by
$$
M\paper' := M\paper \cup (M' \setminus M) = M' \setminus \mathring{M}\spine \subset M',
$$
hence $M'$ is the union of $M\spine$ with $M\paper'$ along their common
boundary $\p M\spine = \p M\paper'$, a disjoint union of $2$-tori.
Recall from \S\ref{sec:coordinates} the collar neighborhoods
$\nN(\p \Sigma)$, $\nN(\p M\spine)$ and $\nN(\p M\paper)$ with their
coordinate systems $(s,\phi)$, $(s,\phi,\theta)$ and $(\phi,t,\theta)$
respectively.  We will denote by
$$
\nN(\p M) \subset M\paper
$$
the neighborhood of $\p M$ in~$M$ defined as the union of all components
of $\nN(\p M\paper)$ that touch~$\p M$.  Similarly, the union of
components of $\nN(\p M\paper)$ that are disjoint from~$\p M$ will be denoted
by
$$
\nN(\p M\paper') \subset M\paper',
$$
as this forms a collar neighborhood of
$\p M\paper'$ in~$M\paper'$.  For assistance in keeping track of this
notation, see Figure~\ref{fig:Mprime}.

\begin{figure}
\psfrag{MP}{$M\paper$}
\psfrag{MSigma}{$M\spine$}
\psfrag{M'minusM}{$M' \setminus M$}
\psfrag{N(pM)}{$\nN(\p M)$}
\psfrag{N(pMP')}{$\nN(\p M\paper')$}
\psfrag{pM}{$\p M$}
\psfrag{M}{$M$}
\psfrag{MP'}{$M\paper'$}
\includegraphics{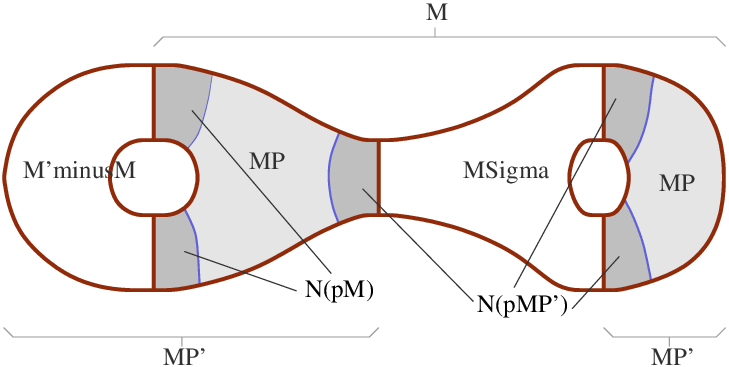}
\caption{\label{fig:Mprime} 
A schematic picture of the closed manifold $M'$ with compact subdomains
$M = M\paper \cup M\spine \subset M'$ and $M\paper' \subset M'$, together with
the various collar neighborhoods
$\nN(\p M)$ and $\nN(\p M\paper')$ of~$\p M\paper$.}
\end{figure}

Now since $\p M\spine = \p M\paper'$, we can use the collars
$\nN(\p M\spine) = (-1,0] \times \p M\spine$ and
$\nN(\p M\paper') = (-1,0] \times \p M\paper'$ to define a diffeomorphism
\begin{equation*}
\begin{split}
\Phi : (-1,0] \times \nN(\p M\spine) &\to (-1,0] \times \nN(\p M\paper') \\
\left( t , (s,x) \right) &\mapsto \left(s , (t,x) \right),
\end{split}
\end{equation*}
and then use this as a gluing map to define (see Figure~\ref{fig:corners})
$$
E' := \big( (-1,0] \times M\spine \big) \cup_\Phi
\big( (-1,0] \times M\paper' \big),
$$
along with the distinguished subdomain
$$
E := \big( (-1,0] \times M\spine \big) \cup_\Phi
\big( (-1,0] \times M\paper \big) \subset E'.
$$
This construction makes $E'$ and $E$ into smooth noncompact $4$-manifolds 
with boundary and codimension~$2$ corners.  The boundary of $E'$ consists of 
two smooth faces
$$
\p E' = \p_v E' \cup \p_h E'
$$
defined as follows:
\begin{itemize}
\item The \defin{vertical boundary} $\p_v E'$ is $\{0\} \times M\paper'$,
so it is a copy of~$M\paper'$.  We will denote the resulting collar
neighborhood of the vertical boundary by
$$
\nN(\p_v E') := (-1,0] \times M\paper' \subset E',
$$
and denote the coordinate on the first factor by~$s$.
We will also want to consider the distinguished subset
$$
\p_v E := \{0\} \times M\paper \subset \p_v E'
$$
and the corresponding collar
$$
\nN(\p_v E) := (-1,0] \times M\paper \subset
\nN(\p_v E'),
$$
which are the same as $\p_v E'$ and $\nN(\p_v E')$ respectively 
if $\p M = \emptyset$.
\item The \defin{horizontal boundary} $\p_h E'$ is $\{0\} \times M\spine$,
a copy of~$M\spine$, and it can also be denoted by $\p_h E := \p_h E'$
since it lies in the subdomain~$E$.  The resulting collar neighborhood of 
this face will be denoted by
$$
\nN(\p_h E) := \nN(\p_h E') := (-1,0] \times M\spine \subset E,
$$
with the coordinate on the first factor denoted by~$t$.
\end{itemize}
Notice that $\p_v E' \cup \p_h E'$ is naturally homeomorphic to~$M'$, 
and similarly $\p_v E \cup \p_h E$ is homoemorphic to~$M$, in both cases by
a homeomorphism that identifies the corner 
$\p_v E \cap \p_h E = \p_v E' \cap \p_h E'$ with
$\p M\spine = \p M\paper' = M\spine \cap M\paper$.  Each connected 
component of the neighborhood 
$$
\nN(\p_v E \cap \p_h E) := \nN(\p_v E) \cap \nN(\p_h E) \subset E
$$
of this corner carries coordinates
$$
(s,\phi,t,\theta) \in (-1,0] \times S^1 \times (-1,0] \times S^1 \subset
\nN(\p_v E \cap \p_h E),
$$
as the construction of the gluing map guarantees that each of these
coordinates is unambiguously defined.  We assign to $E'$ and $E$
the orientation determined by this coordinate system.  A similar coordinate
system exists on each connected component of
$$
\nN(\p\p_v E) := (-1,0] \times \nN(\p M) \subset \nN(\p_v E) \subset E.
$$

\begin{figure}
\psfrag{...}{$\ldots$}
\psfrag{phE}{$\p_h E$}
\psfrag{-1}{$-1$}
\psfrag{s}{$s$}
\psfrag{t}{$t$}
\psfrag{E'minusE}{$E' \setminus E$}
\psfrag{N(ppvE)}{$\nN(\p\p_v E)$}
\psfrag{N(phE)}{$\nN(\p_h E)$}
\psfrag{N(pvEintphE)}{$\nN(\p_v E \cap \p_h E)$}
\psfrag{N(pvE)}{$\nN(\p_v E)$}
\psfrag{pvE}{$\p_v E$}
\psfrag{pvE'}{$\p_v E'$}
\psfrag{pM}{$\p M$}
\psfrag{MPintMS}{$M\spine \cap M\paper$}
\psfrag{N(pM)}{$\nN(\p M)$}
\psfrag{N(pMP')}{$\nN(\p M\paper')$}
\psfrag{N(MSigma)}{$\nN(\p M\spine)$}
\psfrag{Sigma}{$\Sigma$}
\psfrag{pSigma}{$\p\Sigma$}
\psfrag{N(pSigma)}{$\nN(\p\Sigma)$}
\psfrag{Pih}{$\Pi_h$}
\includegraphics{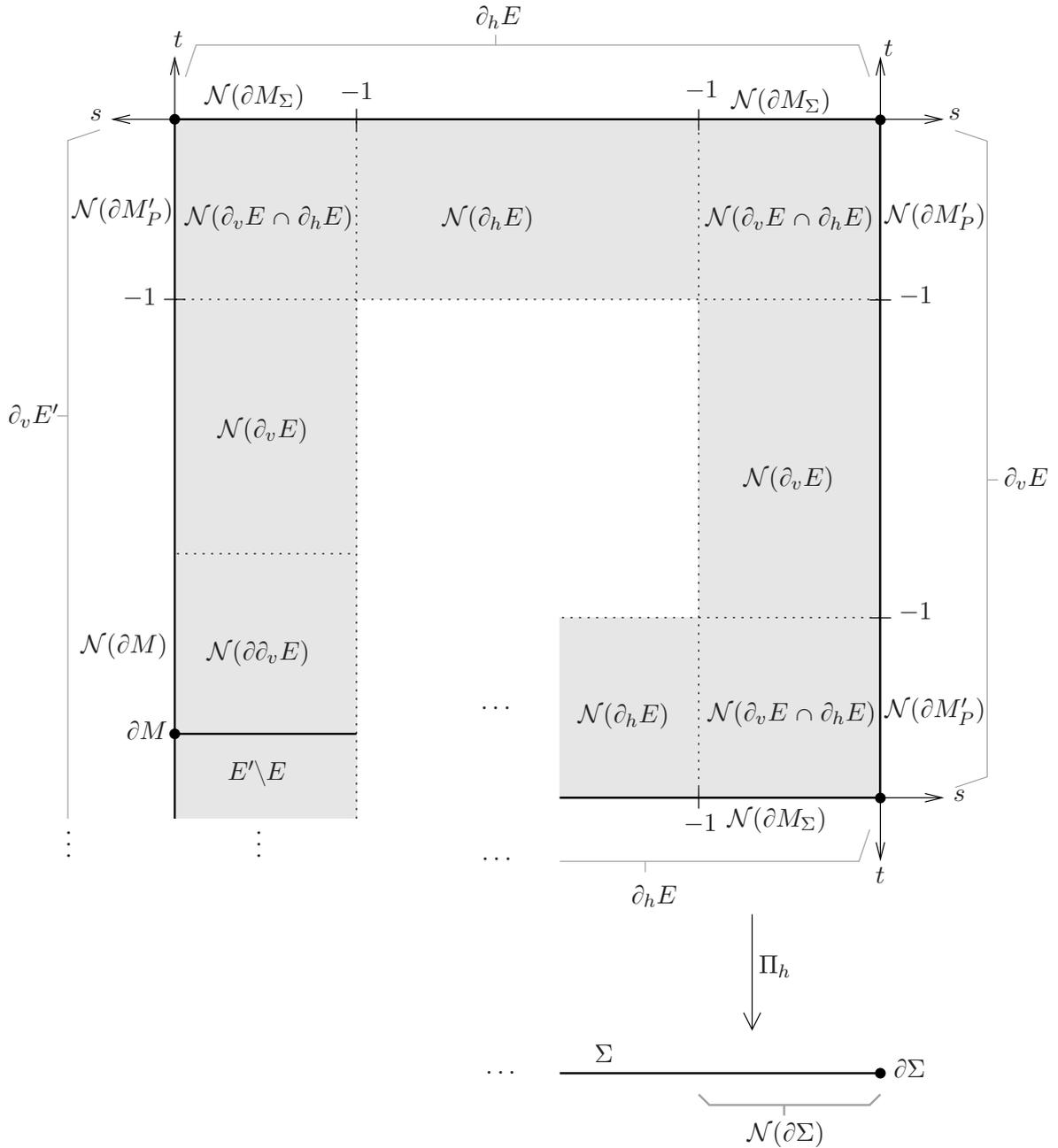}
\caption{\label{fig:corners}
The domain $E'$ with its boundary faces and collar neighborhoods, shown
together with a portion of the fibration $\Pi_h : \nN(\p_h E) \to \Sigma$.
In this example, $M\paper$ contains at least two connected components,
one (shown at the right) that touches two separate spinal components but
not the boundary, and
another (at the left) that does touch~$\p M$.}
\end{figure}

On the collars $\nN(\p_h E)$ and $\nN(\p_v E)$,
one can separately define fibrations
$$
\Pi_h : \nN(\p_h E) = (-1,0] \times (\Sigma \times S^1) \to \Sigma : 
\big(t,(z,\theta)\big) \mapsto \pi\spine(z,\theta) = z,
$$
and
$$
\Pi_v : \nN(\p_v E) = (-1,0] \times M\paper 
\to (-1,0] \times S^1 : (s,x) \mapsto (s,\pi\paper(x)).
$$
On the region where the domains of these two fibrations overlap, we can
write them in $(s,\phi,t,\theta)$-coordinates as
\begin{equation}
\label{eqn:fibrationCorner}
\Pi_h(s,\phi,t,\theta) = (s,\phi), \qquad
\Pi_v(s,\phi,t,\theta) = (s,m\phi).
\end{equation}
While it may not be true in general that $\Pi_h$ and $\Pi_v$
can be fit together to define a global fibration on~$E$, they have the same
fibers on the region of overlap and thus give rise to
a well-defined \defin{vertical subbundle}
$$
V E := \ker T\Pi_h \text{ or } \ker T\Pi_v \subset T E,
$$
which on $\nN(\p_h E)$ is spanned by the vector fields
$\p_t$ and~$\p_\theta$.  Figure~\ref{fig:corners} has been drawn so that
the fibers can be represented as vertical lines in the picture.

\subsubsection{The Liouville structure on~$E'$}
\label{sec:LiouvilleData}

We will use the fibrations $\Pi_v$ and $\Pi_h$ to construct
a Liouville structure on $E$ via the Thurston trick
as in \S\ref{sec:Thurston}, and then
extend it to~$E'$ using the given contact structure on~$M'$.

Fix a Liouville form $\sigma$ on $\Sigma$ that takes the form
$$
\sigma = m e^s \, d\phi \quad \text{ on $\nN(\p\Sigma)$},
$$
where $m \in \NN$ is the \emph{multiplicity} of $\pi\paper : M\paper \to S^1$
at its boundary component adjacent to the relevant component of
$\nN(\p M\spine)$; recall that this number may differ on
distinct connected components of $\nN(\p\Sigma)$,
cf.~\S\ref{sec:coordinates}.  
We will also use $\sigma$ to denote the pullback of this Liouville form
under the trivial bundle projection 
$\Pi_h : \nN(\p_h E) \to \Sigma$, and since 
$\pi\paper(\phi,t,\theta) = m\phi$ on $\nN(\p M\paper)$, $\sigma$ extends 
globally to a $1$-form on $E$ satisfying
$$
\sigma = e^s \, d\pi\paper \quad \text{ on $\nN(\p_v E)$},
$$
where we are abusing notation slightly by using $\pi\paper : 
\nN(\p_v E) \to S^1$ to denote the composition of the 
fibration $\pi\paper : M\paper \to S^1$ with the obvious projection
$\nN(\p_v E) = (-1,0] \times M\paper \to M\paper$,
hence defining $d\pi\paper$ as a real-valued $1$-form on~$\nN(\p_v E)$.

We next define a $1$-form on~$E$ that can be regarded as a \emph{fiberwise
Liouville} structure with respect to the fibrations $\Pi_h$ and~$\Pi_v$.
By Lemma~\ref{lemma:fiberLiouville}, there exists
a $1$-form $\lambda$ on $M\paper$ such that $d\lambda$ is
positive on all fibers of $\pi\paper : M\paper \to S^1$ and
$$
\lambda = e^t \, d\theta \quad \text{ on $\nN(\p M\paper)$}.
$$
Using the same symbol to denote the pullback of $\lambda$ via the 
projection $\nN(\p_v E) = (-1,0] \times M\paper \to
M\paper$, we can then extend $\lambda$ to a global $1$-form 
on $E$ satisfying
$$
\lambda = e^t \, d\theta \quad \text{ on $\nN(\p_h E)$}.
$$
It is fiberwise Liouville in the sense that $d\lambda|_{V E} > 0$
everywhere on~$E$, and since $\lambda|_{T(\p_h E)} = d\theta$, 
the boundaries of the fibers of $\Pi_h$ are positive with respect
to $\lambda$ and are annihilated by $d\lambda|_{T(\p_h E)}$.

We can now apply the Thurston trick: for any constant $K \ge 0$, we define a
$1$-form $\lambda_K$ by
$$
\lambda_K := K\sigma + \lambda.
$$
Corollary~\ref{cor:ThurstonLiouville} in conjunction with
Remark~\ref{remark:Enoncompact} below then provides a constant $K_0 > 0$ such that
$d\lambda_K$ is symplectic everywhere on~$E$ for each $K \ge K_0$.  Near
the boundary, we have
\begin{equation}
\label{eqn:lambdaKboundary}
\lambda_K = K\, \sigma + e^t \, d\theta
\text{ on $\nN(\p_h E)$},\qquad\text{ and }\qquad
\lambda_K = K e^s\, d\pi\paper + \lambda
\text{ on $\nN(\p_v E)$},
\end{equation}
so in particular
\begin{equation}
\label{eqn:lambdaKcorner}
\lambda_K = K m e^s\, d\phi + e^t\, d\theta
\quad \text{ on $\nN(\p_v E \cup \p_h E) \cup \nN(\p\p_v E)$}.
\end{equation}

\begin{remark}
\label{remark:Enoncompact}
The noncompactness of $E$ does not pose any problem in the above use of the
Thurston trick: the reason is that if we fix on $E$ any Riemannian
metric that is independent of the $s$- and/or $t$-coordinates wherever
these are defined, then $|d\lambda \wedge d\lambda|$ is bounded above
and $d\sigma \wedge d\lambda$ is bounded away from zero.  This observation
will be even more useful when we discuss the double completion
in~\cite{LisiVanhornWendl2}.
\end{remark}

We will always assume from 
now on that $K \ge K_0$ so that $d\lambda_K$ is symplectic,
and we will occasionally require further increases in the
value of $K_0$ for convenience.  There is now a Liouville vector 
field $V_K$ on $(E,d\lambda_K)$ defined via the condition
$$
d\lambda_K(V_K,\cdot) \equiv \lambda_K.
$$
From \eqref{eqn:lambdaKboundary} we compute
\begin{equation}
\label{eqn:LiouvilleFormula}
V_K = V_\sigma + \p_t \quad \text{ on $\nN(\p_h E)$},
\end{equation}
where $V_\sigma$ denotes the Liouville vector field on $\Sigma$
dual to~$\sigma$, and from \eqref{eqn:lambdaKcorner}, 
$$
V_K = \p_s + \p_t \quad \text{ on $\nN(\p_v E \cap \p_h E) \cup \nN(\p\p_v E)$}.
$$

\begin{lemma}
\label{lemma:LiouvilleTransverse}
For all $K > 0$ sufficiently large, $ds(V_K) > 0$ on $\nN(\p_v E)$.
\end{lemma}
\begin{proof}
It is equivalent to show that the restriction of 
$\lambda_K = K e^s \, d\pi\paper + \lambda$ to 
$\{s\} \times M\paper$ for each $s \in (-1,0]$ is a positive contact form.
Since $e^s \, d\pi\paper$ is the pullback via $\pi\paper : M\paper \to S^1$ 
of a volume form on~$S^1$ for each fixed $s \in [-1,0]$, the result follows
from Proposition~\ref{prop:GirouxVertical}.
\end{proof}

In light of the lemma, we shall assume from now on that $K_0 > 0$ is large 
enough to ensure $ds(V_K) > 0$ for all $K \ge K_0$.
Before extending $\lambda_K$ to the rest of $E'$, we must make a minor
adjustment in the neighborhood of $(-1,0] \times \p M \subset \nN(\p_v E)$.

\begin{lemma}
\label{lemma:bndryAdjustment}
There exists a smooth homotopy of Liouville forms 
$\{\lambda_K^\tau\}_{\tau \in [0,1]}$ on~$E$ with the following properties:
\begin{enumerate}
\item $\lambda_K^0 \equiv \lambda_K$;
\item The restrictions of $\lambda_K^\tau$ to $T(\p_v E)$ are identical
for every $\tau \in [0,1]$;
\item $\lambda_K^\tau \equiv \lambda_K$ outside a small open neighborhood
of $\nN(\p \p_v E)$ for all $\tau \in [0,1]$;
\item For each $\tau \in [0,1]$, the Liouville vector 
field $V_K^\tau$ determined by $\lambda_K^\tau$ satisfies $ds(V_K^\tau) > 0$
on~$\nN(\p_v E)$;
\item $\lambda_K^1 = e^s \left( K m \, d\phi + e^t \, d\theta \right)$
near $(-1,0] \times \p M \subset \nN(\p_v E)$,
where $m \in \NN$ is the multiplicity of $\pi\paper$ at the relevant
component of~$\p M$.
\end{enumerate}
\end{lemma}
\begin{proof}
Working in $(\phi,t,\theta)$-coordinates on a connected component of 
$\nN(\p M)$, let $\widehat{\nN}(\p M) \subset M$ denote a slightly expanded 
collar neighborhood in which the $t$-coordinate takes values in 
$(-1-\epsilon,0]$ for some $\epsilon > 0$ small.  
Let us similarly extend the
$s$-coordinate to the interval $(-1-\delta,0]$ and consider the expanded
domain
$$
\widehat{\nN}(\p\p_v E) := (-1-\delta,0] \times \widehat{\nN}(\p M)
$$
for some $\delta > 0$ small enough so that $\lambda_K = K m e^s\, d\phi + \lambda$
is still a Liouville form on this domain and its Liouville vector field
$V_K$ is still transverse to all hypersurfaces of the form
$\{s=\operatorname{const}\}$.  Notice that in the region $\{t \ge -1\}
\subset \widehat{\nN}(\p\p_v E)$, we have $\lambda_K = K m e^s\, d\phi + e^t\, d\theta$
and thus $V_K = \p_s + \p_t$.  Now if $\epsilon > 0$ is sufficiently small, 
we can assume that the flow $\Phi^\rho_{V_K}$ of $V_K$ in $\widehat{\nN}(\p\p_v E)$
for times
$\rho \in [-1,0]$ is well defined on the small collar
$$
\{ t\ge  -\epsilon/2 \} \subset \nN(\p M) \subset \p_v E.
$$
Choose a smooth vector field $V$ on $\widehat{\nN}(\p\p_v E)$ with the
following properties:
\begin{enumerate}
\item $V \equiv V_K$ throughout $\nN(\p\p_v E)$ and also in
the region obtained by flowing
$\{ t \ge -\epsilon/2 \} \subset \nN(\p M)$ backwards from time $0$
to time~$-1$;
\item $ds(V) > 0$ is close to~$1$ everywhere;
\item $V \equiv \p_s$ in a neighborhood of $\{t = -1 - \epsilon$\}.
\end{enumerate}
Using the flow $\Phi^\bullet_V$ of $V$, define the embedding
(see Figure~\ref{fig:adjustment})
$$
(-1,0] \times S^1 \times (-1-\epsilon,0] \times S^1 
\stackrel{\Psi}{\hookrightarrow}
\widehat{\nN}(\p\p_v E) : (s,\phi,t,\theta) \mapsto \Phi^s_{V}(\phi,t,\theta).
$$
Identifying the domain of $\Psi$ with the
obvious collar neighborhood in $\nN(\p_v E)$, this map equals the identity
near $\{t=-1-\epsilon\}$ and at $\{s=0\}$, and by deforming the vector field
$V$ we can also find a smooth isotopy
of embeddings $\{\Psi_\tau\}_{\tau \in [0,1]}$ with both of these properties
such that $\Psi_1 = \Psi$ and $\Psi_0 = \Id$.  The desired family of
Liouville forms can then be defined on this collar by
$$
\lambda_K^\tau = \Psi_\tau^*\lambda_K
$$
and extended to the rest of $E$ as~$\lambda_K$.
In particular, we have $\lambda_K^1 = e^s \left( K m\, d\phi + e^t \, d\theta\right)$
for $t \ge -\epsilon/2$ since $\Psi$ redefines the $s$-coordinate via the 
flow of the Liouville vector field.  Since $V_K = \p_s + \p_t$ on
$\nN(\p\p_v E)$, the condition $ds(V_K^\tau) > 0$ is
easily achieved as long as $\delta$ and $\epsilon$ are both sufficiently small.
\end{proof}

\begin{figure}
\psfrag{...}{$\ldots$}
\psfrag{-1}{$-1$}
\psfrag{s}{$s$}
\psfrag{t}{$t$}
\psfrag{N(ppvE)}{$\nN(\p\p_v E)$}
\psfrag{N(pvE)}{$\nN(\p_v E)$}
\psfrag{pM}{$\p M$}
\psfrag{pE}{$\p E$}
\psfrag{-eps/2}{$-\epsilon/2$}
\psfrag{-1-eps}{$-1 - \epsilon$}
\psfrag{-1-delta}{$-1 - \delta$}
\psfrag{N(pM)}{$\nN(\p M)$}
\psfrag{E'minusE}{$E' \setminus E$}
\psfrag{Psi}{$\Psi$}
\includegraphics{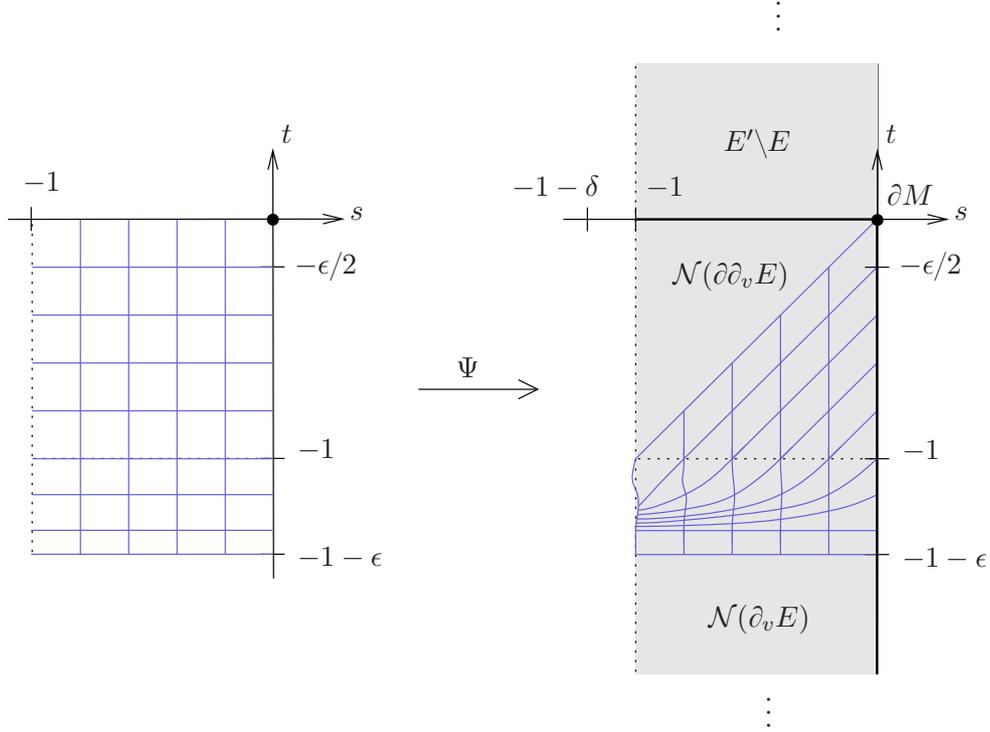}
\caption{\label{fig:adjustment}
The embedding $\Psi$ in the proof of Lemma~\ref{lemma:bndryAdjustment}.}
\end{figure}

Let us now replace $\lambda_K$ with $\lambda_K^1$ from the lemma, so as
to assume without loss of generality that
$\lambda_K = e^s \left( K m\, d\phi + e^t \, d\theta\right)$ on
$(-1,0] \times S^1 \times (-\delta,0] \times S^1 \subset \nN(\p \p_v E)$
for some $\delta > 0$.  We then make one further modification on the same
region and redefine $\lambda_K$ in the form
$$
\lambda_K = e^s \left[ f(t)\, d\theta + K m g(t)\, d\phi \right],
$$
where $f, g : (-\delta,0] \to [0,\infty)$ are smooth functions chosen such that
(see Figure~\ref{fig:fgBoundary})
\begin{itemize}
\item $(f(t),g(t)) = (e^t,1)$ for $t$ near~$-\delta$;
\item $f' g - f g' > 0$;
\item $f(0) = 1$ and $g(0) = 0$;
\item $f'(0) = 0$.
\end{itemize}
\begin{figure}
\psfrag{f}{$f$}
\psfrag{g}{$g$}
\psfrag{t=-delta}{$t=-\delta$}
\psfrag{t=0}{$t=0$}
\includegraphics{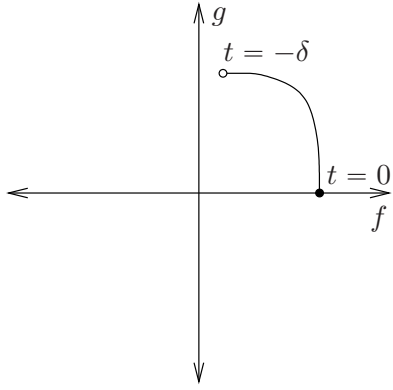}
\caption{\label{fig:fgBoundary}
The path $t \mapsto (f(t),g(t))$ for $-\delta < t \le 0$.}
\end{figure}
These conditions guarantee that $\alpha' := f(t) \, d\theta + K m g(t)\, d\phi$ 
defines a positive contact form on $S^1 \times (-\delta,0] \times S^1 
\subset \nN(\p M)$ satisfying $\alpha'(\p_\phi) = 0$ and 
$d\alpha'(\p_\theta,\cdot) = 0$ at~$\p M$.  One can now extend
$\alpha'$ smoothly beyond $\p M$ so that it defines a
contact form for $\xi$ on $M' \setminus M$.  The corresponding extension
of $\lambda_K$ is defined by
$$
\lambda_K := e^s \alpha' \quad \text{ on $(-1,0] \times (M' \setminus M)
\subset \nN(\p_v E')$}.
$$
The corresponding Liouville vector field on $(-1,0] \times (M' \setminus M)$
is simply~$\p_s$.

\subsubsection{Contact hypersurfaces and smoothing corners}
\label{sec:corners}

It is immediate from the above constructions that the Liouville vector field
$V_K$ is transverse to both $\p_v E'$ and $\p_h E'$, so smoothing the corners
makes $\p E'$ into a contact hypersurface.  Moreover, the fiberwise
Liouville condition on $\lambda_K$ and the specific way that it was
modified in $\nN(\p \p_v E)$ mean that the induced contact structure
on the smoothing of $\p_h E \cup \p_v E$ will be isotopic to one
supported by $\boldsymbol{\pi}$, hence the contact structure on $\p E'$
is isotopic to~$\xi$ after identifying the latter with~$M'$.  

To define the smoothing more precisely,
choose a pair of smooth functions
$F , G : (-1,1) \to (-1,0]$ that satisfy the following conditions:
\begin{itemize}
\item $(F(\rho),G(\rho)) = (\rho,0)$ for $\rho \le - 1 / 4$;
\item $(F(\rho),G(\rho)) = (0,-\rho)$ for $\rho \ge 1 / 4$;
\item $G'(\rho) < 0$ for $\rho > - 1 / 4$;
\item $F'(\rho) > 0$ for $\rho < 1 / 4$.
\end{itemize}
Now let
$$
M^0 \subset E'
$$
denote the smooth hypersurface obtained from $\p E'$ by replacing 
$\p E' \cap \nN(\p_v E \cap \p_h E)$ in $(s,\phi,t,\theta)$-coordinates with
\begin{equation}
\label{eqn:smoothing}
\left\{ (F(\rho),\phi, G(\rho),\theta)\ \Big|\ 
\phi,\theta \in S^1,\ -1 < \rho < 1 \right\};
\end{equation}
see Figure~\ref{fig:cornersSmoothed}.
This smoothing is transverse to $V_K = \p_s + \p_t$ by construction,
thus $M^0$ is a contact hypersurface and inherits the contact structure
$$
\xi_0 := \ker \alpha_0, \qquad \alpha_0 := \lambda_K|_{TM^0}.
$$

\begin{figure}
\psfrag{-1/4}{$-1/4$}
\psfrag{-3/4}{$-3/4$}
\psfrag{M0}{$M^0$}
\psfrag{M-}{$M^-$}
\psfrag{s}{$s$}
\psfrag{t}{$t$}
\psfrag{E'minusE}{$E' \setminus E$}
\psfrag{...}{$\ldots$}
\psfrag{pM0}{$\p M^0$}
\includegraphics{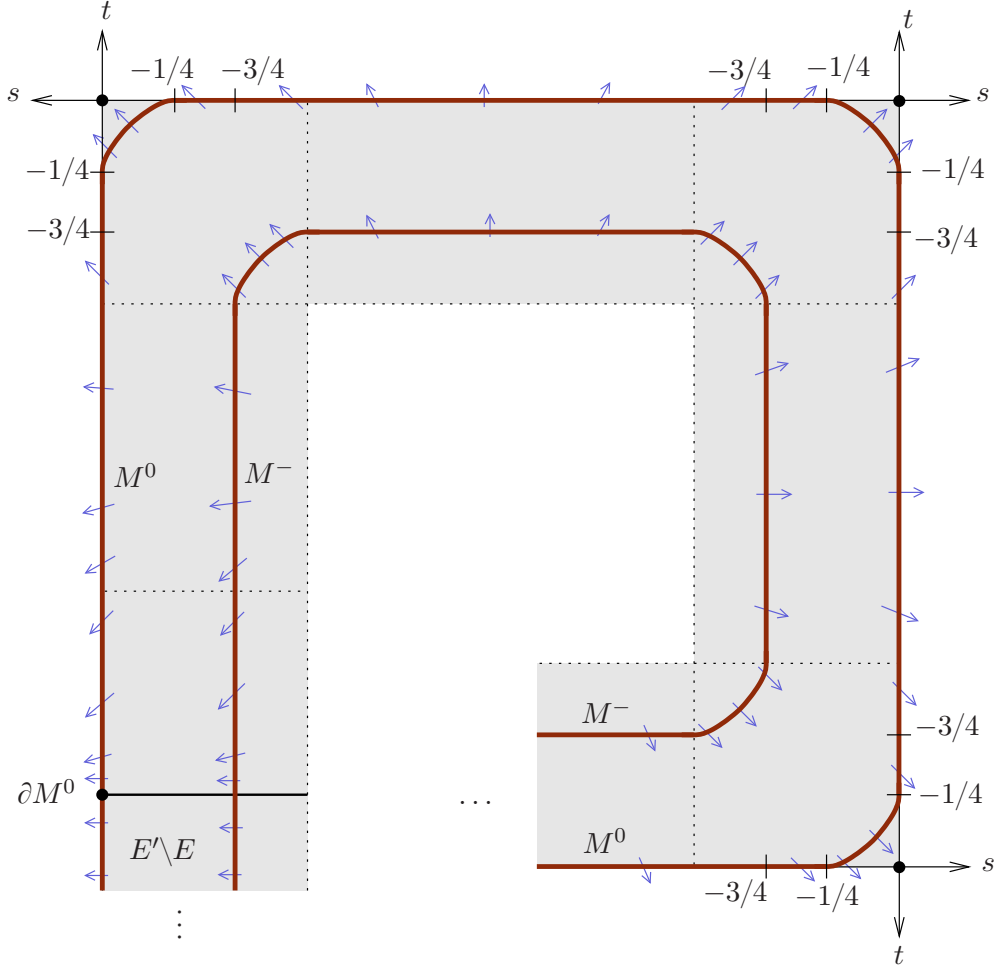}
\caption{\label{fig:cornersSmoothed}
The smoothed hypersurfaces $M^0$ and $M^-$ sitting inside the same model of
$E'$ as shown in Figure~\ref{fig:corners}, together with the
transverse Liouville vector field~$V_K$.}
\end{figure}

By translating $M^0$ a distance of $-3/4$ in both the $s$- and $t$-coordinates,
one obtains another contact hypersurface
$$
(M^-,\xi_-) \subset (E',d\lambda_K)
$$
which contains portions of the two hypersurfaces
$\{-3/4\} \times M\spine \subset \nN(\p_h E)$ and
$\{-3/4\} \times M\paper' \subset \nN(\p_v E')$ and a translated copy of
\eqref{eqn:smoothing} replacing the neighborhood of their intersection
(see the inner hypersurface in Figure~\ref{fig:cornersSmoothed}).
Since $(M^-,\xi_-)$ and $(M^0,\xi_0)$ can evidently be connected by a smooth
$1$-parameter family of contact hypersurfaces in $(E',d\lambda_K)$,
their contact structures are isotopic, so in particular $\xi_-$ is isotopic
to $\xi$ after a suitable identification of $M^-$ with~$M'$.

\subsection{Spine removal cobordisms}
\label{sec:easySpineRemoval}

In this section we use the model $(E',d\lambda_K)$ with contact hypersurfaces
$(M^-,\xi_-)$ and $(M^0,\xi_0)$ constructed in \S\ref{sec:LiouvilleCollar}
to prove Theorem~\ref{thm:spineRemoval}.  In particular, we will enlarge
$E'$ in order to construct
a symplectic spine removal cobordism whose negative weakly contact boundary 
is the contact hypersurface $(M^-,\xi_-)$.

Fix a decomposition of $\Sigma$ into open and closed subsets
$$
\Sigma = \Sigma\remove \amalg \Sigma\other,
$$
and assume $\Sigma\remove$ is nonempty.
Fix also a trivialization $M\spine = \Sigma \times S^1$, so in particular
$\pi\spine^{-1}(\Sigma\remove) = \Sigma\remove \times S^1$.  The choice of
decomposition $\Sigma = \Sigma\remove \amalg \Sigma\other$
splits the horizontal boundary $\p_h E$ into a disjoint union
$$
\p_h E = \p_h\remove E \amalg \p_h\other E := \left( \Sigma\remove \times S^1 \right)
\amalg \left( \Sigma\other \times S^1 \right),
$$
and the collar $\nN(\p_h E)$ decomposes accordingly as
$$
\nN(\p_h E) = \nN(\p_h\remove E) \amalg \nN(\p_h\other E).
$$
Recall that $\lambda_K = K\sigma + e^t\, d\theta$ in $\nN(\p_h E)$.

\begin{figure}
\psfrag{(0)xS1}{$\{0\} \times S^1$}
\psfrag{(-1/2)xS1}{$\{-1/2\} \times S^1$}
\includegraphics{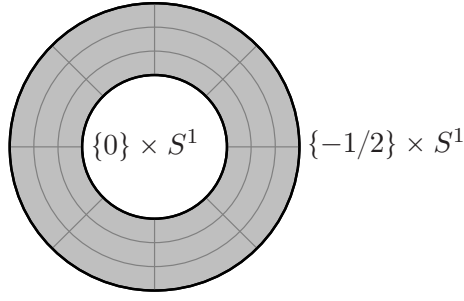}
\caption{\label{fig:annulus}
The diffeomorphism $\psi : [-1/2,0] \times S^1 \to \DD^2 \setminus 
\mathring{\DD}_{1/2}^2$.}
\end{figure}

We will now modify $E'$ by attaching a generalized notion of a
``symplectic handle'' to~$\p_h\remove E$.  Choose a diffeomorphism
$$
\psi: [-1/2,0] \times S^1 \stackrel{\cong}{\longrightarrow} 
\DD^2 \setminus \mathring{\DD^2}_{1/2},
$$
where $\DD^2_{1/2}$ denotes the closed disk of radius $1/2$ inside the unit
disk $\DD^2 \subset \CC$, and assume $\psi$
maps $\{-1/2\} \times S^1$ to $\p \DD^2$; see Figure~\ref{fig:annulus}.
Using the obvious coordinates
$(t,\theta)$ on $[-1/2,0] \times S^1$, let $\omega_\DD$ denote
any area form on $\DD^2$ that restricts to $\psi_*(e^t \, dt \wedge d\theta)$ 
outside of~$\DD^2_{1/2}$.
We can then define a new symplectic manifold with boundary and corners by
$$
(\widetilde{E}',\widetilde{\omega}_K) := (E',d\lambda_K) \cup_{\Psi} 
\left( \Sigma\remove \times \DD^2, K\, d\sigma + \omega_\DD \right),
$$
where $d\sigma$ and $\omega_\DD$ are each identified with their pullbacks
to $\Sigma\remove \times \DD^2$ via the obvious projections, and
the gluing map is defined by
$$
\Psi : \nN(\p_h\remove E) \supset 
(-1/2,0] \times \Sigma\remove \times S^1 \hookrightarrow \Sigma\remove \times \DD^2 :
(t,z,\theta) \mapsto (z,\psi(t,\theta)).
$$
Schematic pictures of this modification are shown in 
Figures~\ref{fig:spineRemoval} and~\ref{fig:spineRemoval3}, for cases
where $\Sigma\remove$ has one or two connected components respectively.
Since $\nN(\p_h\remove E)$ lies entirely in the subdomain $E \subset E'$, we
can define a corresponding subdomain 
$$
\widetilde{E} \subset \widetilde{E}'
$$
by attaching $\Sigma\remove \times \DD^2$ in this way to $E$ instead of~$E'$.
The boundary of $\widetilde{E}'$ now has two smooth faces
$\p \widetilde{E}' = \p_h \widetilde{E}' \cup \p_v \widetilde{E}'$,
where the ``horizontal'' boundary is
$$
\p_h \widetilde{E}' := \p_h \widetilde{E} := \p_h\other E,
$$
and the ``vertical'' boundary 
$$
\p_v \widetilde{E}' = \p_v E' \cup (\p \Sigma\remove \times \DD^2)
$$
is obtained from $\p_v E'$ by gluing in
$\p \Sigma\remove \times \DD^2$---a disjoint union of solid tori---along the boundary 
components of
$\p_v E'$ that touch $\p_h\remove E$.  Again this attachment has nothing to do
with the region $E' \setminus E$, so we can define
$$
\p_v \widetilde{E} := \p_v E \cup (\p \Sigma\remove \times \DD^2) \subset
\p_v \widetilde{E}'.
$$
Each of these gives rise to collars which are also subsets of
$\widetilde{E}'$: we shall denote
$$
\nN(\p_h \widetilde{E}') := \nN(\p_h \widetilde{E}) :=
(-1,0] \times \p_h\other E = (-1,0] \times \Sigma\other \times S^1
\subset \widetilde{E},
$$
with $t$ denoting the coordinate in $(-1,0]$, and
\begin{equation*}
\begin{split}
\nN(\p_v \widetilde{E}') &:= (-1,0] \times \p_v \widetilde{E}'
\subset \widetilde{E}', \\
\nN(\p_v \widetilde{E}) &:= (-1,0] \times \p_v \widetilde{E}
\subset \widetilde{E},
\end{split}
\end{equation*}
with the coordinate on $(-1,0]$ denoted by~$s$.  These collars do not
cover all of $\tilde{E}'$, as we also have
$$
\widetilde{\nN}(\p_h\remove E) := \nN(\p_h\remove E) \cup_\Psi (\Sigma\remove \times \DD^2)
\cong \Sigma\remove \times \DD^2 \subset \widetilde{E};
$$
see Figures~\ref{fig:spineRemoval} and~\ref{fig:spineRemoval3}.
Here the slightly different notational convention is meant to emphasize
the fact that $\widetilde{\nN}(\p_h\remove E)$ is not actually a collar
neighborhood of any part of the boundary---it contains the original
$\p_h\remove E$, but this now lives in the \emph{interior} of the ``handle''
$\Sigma\remove \times \DD^2$.  With this notation in place,
$\widetilde{E}'$ and $\widetilde{E}$ can be presented as the unions of
overlapping regions
\begin{equation*}
\begin{split}
\widetilde{E}' &= \nN(\p_v \widetilde{E}') \cup \nN(\p_h \widetilde{E})
\cup \widetilde{\nN}(\p_h\remove E),\\
\widetilde{E} &= \nN(\p_v \widetilde{E}) \cup \nN(\p_h \widetilde{E})
\cup \widetilde{\nN}(\p_h\remove E).
\end{split}
\end{equation*}
The fibrations $\Pi_v : \nN(\p_v E) \to (-1,0] \times S^1$ and
$\Pi_h : \nN(\p_h E) \to \Sigma$ extend in obvious ways: on the horizontal
neighborhoods we have trivial projections
\begin{equation*}
\begin{split}
&\widetilde{\Pi}_h : \nN(\p_h \widetilde{E}) = (-1,0] \times \Sigma\other \times S^1
\to \Sigma\other, \\
&\widetilde{\Pi}_h : \widetilde{\nN}(\p_h\remove E) \cong \Sigma\remove \times \DD^2 \to \Sigma\remove,
\end{split}
\end{equation*}
and on the vertical collars, the formula $\Pi_v(s,\phi,t,\theta) = (s,m\phi)$
produces an extension
$$
\widetilde{\Pi}_v : \nN(\p_v \widetilde{E}) \to (-1,0] \times S^1
$$
which is defined on each connected component of the attached region 
$(-1,0] \times \p \Sigma\remove \times \DD^2$ by
$$
(-1,0] \times S^1 \times \DD^2 \stackrel{\widetilde{\Pi}_v}{\longrightarrow}
(-1,0] \times S^1 : (s,\phi,\zeta) \mapsto (s,m\phi),
$$
with the multiplicity $m \in \NN$ as usual depending on the component
under consideration.  Denote the resulting vertical subbundle by
$V \widetilde{E} \subset T\widetilde{E}$ and observe that
$$
\widetilde{\omega}_K|_{V\widetilde{E}} > 0
$$
by construction.

\begin{figure}
\psfrag{s}{$s$}
\psfrag{t}{$t$}
\psfrag{Ntilde}{$\widetilde{\nN}(\p_h\remove E)$}
\psfrag{SigmaremxpD2}{$\Sigma\remove \times \p\DD^2$}
\psfrag{Sigmaremx(0)}{$\Sigma\remove \times \{0\}$}
\psfrag{phremE}{$\p_h\remove E$}
\psfrag{-1/2}{$-1/2$}
\psfrag{...}{$\ldots$}
\psfrag{M-}{$M^-$}
\psfrag{pvcvxEtilde'}{$\p_v\convex\widetilde{E}'$}
\psfrag{pvcvxEtilde}{$\p_v\convex\widetilde{E}$}
\psfrag{E'minusE}{$E' \setminus E$}
\psfrag{N(phothE)}{$\nN(\p_h \widetilde{E}) = \nN(\p_h\other E)$}
\psfrag{Mtildecvx}{$\widetilde{M}\convex$}
\psfrag{N(pvcvxEtilde)}{$\nN(\p_v\convex \widetilde{E})$}
\psfrag{N(pvcvxEtilde')}{$\nN(\p_v\convex \widetilde{E}')$}
\psfrag{phEtilde}{$\p_h \widetilde{E}$}
\includegraphics{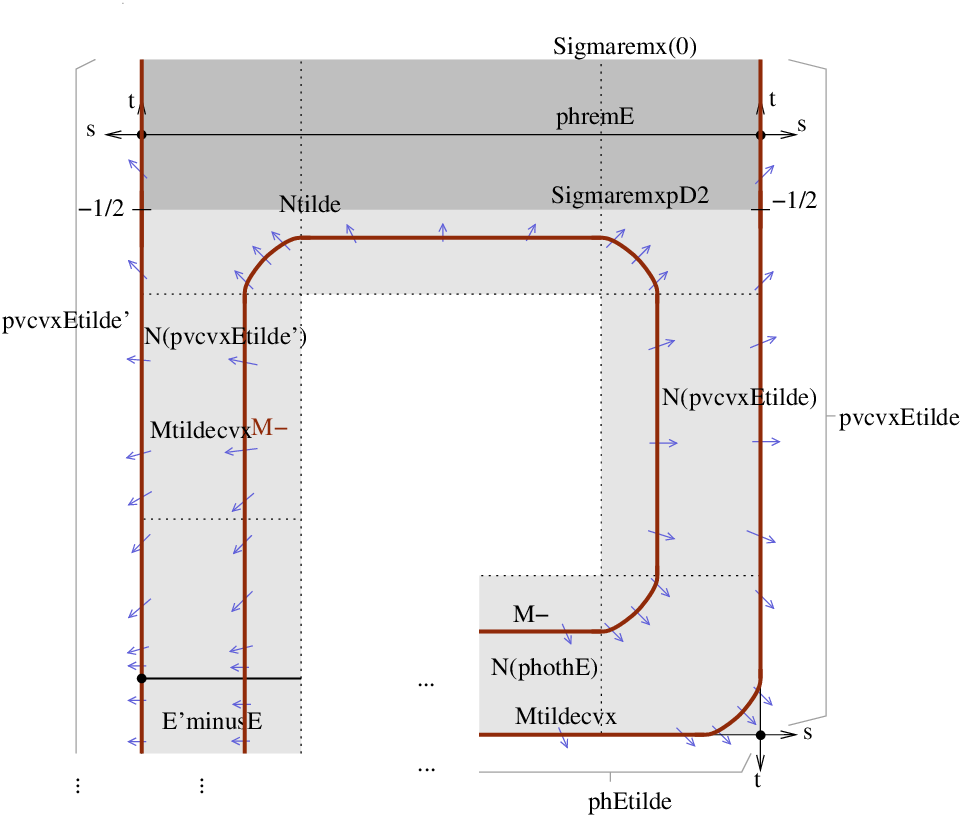}
\caption{\label{fig:spineRemoval}
The domain $\widetilde{E}'$ constructed from $E'$ of
Figure~\ref{fig:corners} by gluing $\Sigma\remove \times \DD^2$ (the darkly
shaded region) to the spinal component at the top of the picture.  The picture
is slightly misleading at its top border because there is no actual boundary
of $\widetilde{E}'$ here: one can think of this instead as the
``center'' $\Sigma\remove \times \{0\}$ of $\Sigma\remove \times \DD^2$,
and in particular, the only actual corner of $\widetilde{E}'$ shown in the
picture is the one at the bottom right.  The spine removal cobordism is
defined to be the region between the two hypersurfaces $M_-$ and
$\widetilde{M}\convex$; the former is contact type since it remains transverse
to the same Liouville vector field, but this vector does not extend over
all of $\widetilde{M}\convex$, hence the latter is in general only weakly
convex.}
\end{figure}

\begin{figure}
\psfrag{s}{$s$}
\psfrag{t}{$t$}
\psfrag{Ntilde}{$\widetilde{\nN}(\p_h\remove E)$}
\psfrag{SigmaremxpD2}{$\Sigma\remove \times \p\DD^2$}
\psfrag{Sigmaremx(0)}{$\Sigma\remove \times \{0\}$}
\psfrag{phremE}{$\p_h\remove E$}
\psfrag{-1/2}{$-1/2$}
\psfrag{...}{$\ldots$}
\psfrag{M-}{$M^-$}
\psfrag{pvcvxEtilde'}{$\p_v\convex\widetilde{E}'$}
\psfrag{pvflatEtilde}{$\p_v\flat\widetilde{E} = \widetilde{M}\flat$}
\psfrag{E'minusE}{$E' \setminus E$}
\psfrag{Mtildecvx}{$\widetilde{M}\convex$}
\psfrag{N(pvflatEtilde)}{$\nN(\p_v\flat \widetilde{E})$}
\psfrag{N(pvcvxEtilde')}{$\nN(\p_v\convex \widetilde{E}')$}
\includegraphics{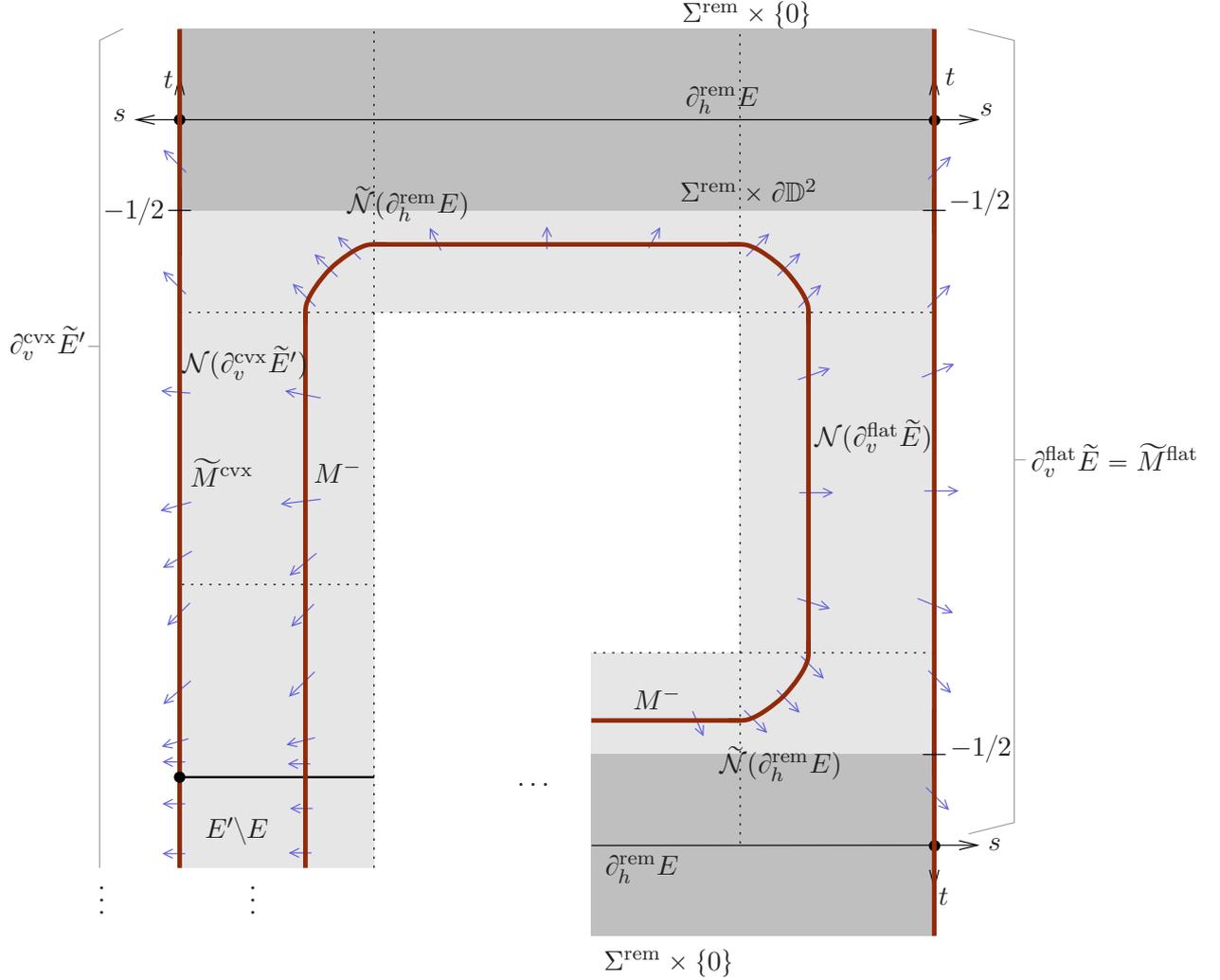}
\caption{\label{fig:spineRemoval3}
A variant of Figure~\ref{fig:spineRemoval} in which $\Sigma\remove \times \DD^2$ has two
connected components, attached at both the top and the bottom of the picture.
The upper boundary of the cobordism now includes a component $\widetilde{M}\flat$
that is not contact, as it is foliated by closed pages of a \emph{generalized}
spinal open book.  (Note that the only actual boundary of $\widetilde{E}'$ in
this picture is at the sides; the top and bottom represent two distinct connected
components of the \emph{interior} submanifold $\Sigma\remove \times \{0\} \subset
\Sigma\remove \times \DD^2$.)}
\end{figure}

It will be useful to decompose $\p_v\widetilde{E}$ 
and $\p_v\widetilde{E}'$ further into the components
$$
\p_v\widetilde{E} = \p_v\flat\widetilde{E} \amalg \p_v\convex\widetilde{E},\qquad
\p_v\widetilde{E}' = \p_v\flat\widetilde{E} \amalg \p_v\convex\widetilde{E}',
$$
with corresponding collars
$\nN(\p_v\flat\widetilde{E})$, $\nN(\p_v\convex\widetilde{E})$ and
$\nN(\p_v\convex\widetilde{E}')$,
where $\p_v\flat\widetilde{E}$ is defined as the union of all components of
$\p_v\widetilde{E}$ such that the fibers of
$\widetilde{\Pi}_v : \nN(\p_v\flat\widetilde{E}) \to (-1,0] \times S^1$
have empty boundary.  Such components arise whenever $M\paper$ has components
with boundary contained in $\pi\spine^{-1}(\Sigma\remove)$;
see Figure~\ref{fig:spineRemoval3}.
The notation is motivated by the fact that, as we'll
see below, $\p_v\convex\widetilde{E}'$ inherits a natural contact structure that
is dominated by~$\widetilde{\omega}_K$, hence making $\p_v\convex\widetilde{E}'$ weakly
convex, but $\p_v\flat\widetilde{E}$ does
not; in fact for certain natural choices of almost complex structure on 
$\widetilde{E}$, $\p_v\convex\widetilde{E}'$ is pseudoconvex while
$\p_v\flat\widetilde{E}$ is Levi flat.

The fibrations $\widetilde{\Pi}_v$ and $\widetilde{\Pi}_h$ induce on
$\p_v\convex\widetilde{E} \cup \p_h\widetilde{E}$ the structure of a
spinal open book $\widetilde{\boldsymbol{\pi}}$ with paper 
$\p_v\convex\widetilde{E}$ and spine
$\p_h\widetilde{E}$, the latter fibering over~$\Sigma\other$.

\begin{lemma}
\label{lemma:dominating}
After smoothing the corner of 
$\p_v\convex\widetilde{E} \cup \p_h\widetilde{E}$,
$\widetilde{\boldsymbol{\pi}}$ supports a contact structure that is dominated
by~$\widetilde{\omega}_K$.
\end{lemma}
\begin{proof}
We mimic the procedure
that was used in \S\ref{sec:LiouvilleCollar} to define $\lambda_K$ on~$E'$:
first choose a fiberwise Liouville form $\widetilde{\lambda}$ on
$\p_v\convex\widetilde{E}$ that equals $e^t\, d\theta$ in the collar
neighborhoods of remaining boundary components (i.e.~those that were not
capped off in the transformation from $\p_v E$ to $\p_v\convex\widetilde{E}$).
Pulling back via the obvious projection defines $\widetilde{\lambda}$
on $\nN(\p_v\convex\widetilde{E})$, and the formula $\widetilde{\lambda} =
e^t\, d\theta$ extends it over $\nN(\p_h \widetilde{E})$.  (Note that by
Stokes' theorem, $\widetilde{\lambda}$ cannot be extended
to~$\nN(\p_v\flat\widetilde{E})$.)
We can then use the Thurston trick to define a Liouville form
$$
\widetilde{\lambda}_K := K\sigma + \widetilde{\lambda}
$$
on $\nN(\p_v\convex\widetilde{E}) \cup \nN(\p_h \widetilde{E})$, after
possibly increasing the value of $K > 0$, and this Liouville form
matches $\lambda_K$ on the regions where $\widetilde{\lambda} = e^t\, d\theta$
and can thus be extended to
$\nN(\p_v \widetilde{E}') \setminus \nN(\p_v \widetilde{E})$ in the
same way as~$\lambda_K$.  We claim now that if $K > 0$ is sufficiently
large, then
$$
\left.\widetilde{\lambda}_K \wedge \widetilde{\omega}_K\right|_{T(\p_v\convex\widetilde{E}')} > 0 
\quad \text{ and }\quad
\left.\widetilde{\lambda}_K \wedge \widetilde{\omega}_K\right|_{T(\p_h \widetilde{E}')} > 0.
$$
The second relation is immediate because $\lambda_K = \widetilde{\lambda}_K$
near $\p_h \widetilde{E}'$, so we are merely rephrasing the fact that
$\p_h\other E$ is a contact hypersurface in $(E,d\lambda_K)$.  The first
relation is similarly immediate on the regions where 
$\lambda_K = \widetilde{\lambda}_K$, so we only still need to check that it
holds on $\p_v\convex\widetilde{E}$.  To see this, notice that $\widetilde{\omega}_K$
can be written in $\nN(\p_v\convex\widetilde{E})$ as
$$
\widetilde{\omega}_K = K d\left( e^s\, d\widetilde{\Pi}_v \right) + \omega\fib,
$$
where $\omega\fib$ is a closed $2$-form that satisfies 
$\omega\fib|_{V\widetilde{E}} > 0$ and is independent of~$K$, while
$e^s\, d\widetilde{\Pi}_v$ can be regarded as the pullback via
$\widetilde{\Pi}_v$ of a Liouville form on $[-1,0] \times S^1$.  The claim thus
follows via Proposition~\ref{prop:ThurstonSymp}.

Finally, the same argument used previously for $\lambda_K$ shows that 
$\widetilde{\lambda}_K$ restricts to both $\p_h\widetilde{E}'$ and
$\p_v\convex\widetilde{E}'$ as a contact form, and by construction it matches
the contact form induced by $\lambda_K$ in a neighborhood of the corners
of $\p_v\convex\widetilde{E}' \cup \p_h\widetilde{E}'$.  It follows that
we can smooth these corners by the same procedure that was used in
\S\ref{sec:corners} to define the contact hypersurface~$M^0$, giving rise in this
case to a weakly contact hypersurface
$$
(\widetilde{M}\convex,\widetilde{\xi}) \subset (\widetilde{E}',\widetilde{\omega}_K)
$$
whose contact structure $\widetilde{\xi}$ is defined by restricting
$\widetilde{\lambda}_K$ to~$\widetilde{M}\convex$.
\end{proof}

The weakly contact hypersurface $(\widetilde{M}\convex,\widetilde{\xi})$ found
in the above proof is shown in Figure~\ref{fig:spineRemoval} as the smooth
curve traversing the outer boundary of $\widetilde{E}'$ with some rounding at the
corners.  We are now in a position to define an actual spine removal cobordism:
let 
$$
X \subset \widetilde{E}'
$$
denote the region that is sandwiched in between $M^- \subset E' \subset \widetilde{E}'$
and $\p_v\flat\widetilde{E} \amalg \widetilde{M}\convex \subset \widetilde{E}'$, 
making $(X,\widetilde{\omega}_K)$ a compact symplectic manifold with
strongly concave boundary $(M^-,\xi_-)$, weakly convex
boundary $(\widetilde{M}\convex,\widetilde{\xi})$, and additional boundary components
$\p_v\flat\widetilde{E}$ which are neither concave nor convex but are
fibered by closed symplectic surfaces.

To finish the proof of Theorem~\ref{thm:spineRemoval}, we need to modify
$(X,\widetilde{\omega}_K)$ to allow symplectic forms that are not exact at the negative
boundary.  Suppose $\Omega$ is a closed $2$-form on $M'$ that satisfies
$\Omega|_{\xi} > 0$ and is exact on $\Sigma\remove \times S^1$.
We can then find a closed $2$-form $\eta$ on $M'$
with the following properties:
\begin{enumerate}
\item $[\eta] = [\Omega] \in H^2_\dR(M')$;
\item On each of the collar components $S^1 \times (-1,0] \times S^1
\subset \nN(\p M\paper)$ and $(-1,0] \times S^1 \times S^1
\subset \nN(\p M\spine)$, $\eta$ is a constant multiple of $d\phi \wedge d\theta$;
\item $\eta$ vanishes on  $\pi\spine^{-1}(\Sigma\remove) = \Sigma\remove \times S^1$.
\end{enumerate}
The third condition is possible due to the cohomological assumption,
and combining this assumption with the second condition implies that
$\eta$ also vanishes on all components of $\nN(\p M\paper)$ adjacent
to $\pi\spine^{-1}(\Sigma\remove)$.  We can now define $\eta$ as a closed
$2$-form on $\nN(\p_v E')$ and $\nN(\p_h E')$ by pulling back via the
projections $(-1,0] \times M\paper' \to M\paper'$ and 
$(-1,0] \times M\spine$ respectively, and the second condition implies that
$\eta$ remains well defined after gluing these collars together to form~$E'$,
thus we shall regard $\eta$ as a closed $2$-form on~$E'$.  By construction,
$\eta$ vanishes near $\p_h\remove E$, hence $\eta$ can also be regarded as
defining a closed $2$-form on $\widetilde{E}'$.  
Its restriction
$$
\eta^- := \eta|_{T M^-}
$$
is cohomologous to~$\Omega$ after identifying $M^-$ with~$M$.  The following
is an immediate consequence of the fact that the nondegeneracy of $2$-forms and 
the ``weakly contact'' condition are both open.

\begin{lemma}
\label{lemma:weakContact}
There exists a constant $C_0 > 0$ such that for all
$C \ge C_0$, the $2$-form
$$
\widetilde{\omega}_K' := C \widetilde{\omega}_K + \eta
$$
is symplectic on $X$, the boundary components $(M^-,\xi_-)$ and
$(\widetilde{M}\convex,\widetilde{\xi})$ are weakly concave and convex respectively, and
$\widetilde{\omega}_K'$ is positive on the closed surface fibers in $\p_v\flat\widetilde{E}$.
\qed
\end{lemma}

Finally, observe that since $\Omega$ and $\eta$ are cohomologous on~$M'$
and $\Omega|_\xi > 0$,
\cite{MassotNiederkruegerWendl}*{Lemma~2.10} provides a symplectic form
on $[0,1] \times M'$ that restricts to $\Omega$ on $\{0\} \times M'$ and
$C\, d\alpha' + \eta$ on $\{1\} \times M'$, where one has the freedom to
choose $\alpha'$ as any contact form for~$\xi$ at the expense of inserting
a sufficiently large constant $C > 0$.  We can therefore make these
choices and increase the value of $C \ge C_0$ if necessary so that the
weak symplectic cobordism $(X,\widetilde{\omega}_K')$ provided by
Lemma~\ref{lemma:weakContact} can be attached on top of
$[0,1] \times M'$.  All together, this provides a weak symplectic cobordism
with the properties stated in Theorem~\ref{thm:spineRemoval} and thus
completes the proof.

\section{Nonfillability via spine removal}
\label{sec:spineRemoval}

In this section we use spine removal surgery to prove
Theorems~\ref{thm:nonseparating} and~\ref{thm:planarTorsion}.
Theorem~\ref{thm:nonseparating} will be an immediate 
corollary of the following result, using the method of
\cite{AlbersBramhamWendl}; it says essentially that any contact manifold
with a partially planar domain can be
given a \emph{symplectic cap} that contains a nonnegative symplectic sphere.

\begin{thm}
\label{thm:cap}
Suppose $(M',\xi)$ is a contact $3$-manifold containing an
$\Omega$-separating partially planar domain for some closed $2$-form
$\Omega$ with $\Omega|_\xi > 0$.  Then there exists a compact symplectic
manifold $(X,\omega)$ with $\p X = -M'$ and $\omega|_{TM'} = \Omega$
such that $(X,\omega)$ contains a symplectically embedded $2$-sphere
with vanishing self-intersection number.
\end{thm}
\begin{proof}
Let $M \subset M'$ denote the partially planar domain, $M\planar\paper \subset M$
its planar piece, and $\Sigma_1 \times S^1,\ldots,\Sigma_r \times S^1
\subset M\spine$ the smallest collection of spinal components that 
contain $\p M\planar\paper$.
Since $\Omega$ is exact on all these components, Theorem~\ref{thm:spineRemoval}
provides a spine removal cobordism $(X_0,\omega)$ with
$$
\p X_0 = - M' \amalg \widetilde{M}'
$$
and $\omega|_{TM'} = \Omega$, constructed by attaching handles 
$\Sigma_i \times \DD^2$
along each of the spinal components surrounding $\p M\planar\paper$.  The
surgered manifold $\widetilde{M}'$ is then disconnected and can be written as
$$
\widetilde{M}' = \widetilde{M}_1' \amalg \widetilde{M}_2',
$$
where $\widetilde{M}_1'$ is a symplectic sphere bundle over $S^1$, 
and $\widetilde{M}_2'$ is
either a contact manifold $(\widetilde{M}_2',\xi_2)$ with 
$\omega|_{\xi_2} > 0$
or another symplectic fibration over $S^1$ with closed fibers.
Both components can now be capped using the method of Eliashberg
\cite{Eliashberg:cap}, and the symplectic $S^2$-fibers of~$\widetilde{M}_1$
give the desired symplectic spheres with vanishing self-intersection.
\end{proof}

We recall briefly why this result implies Theorem~\ref{thm:nonseparating}:
if $(W,\omega)$ is a closed symplectic $4$-manifold and 
$M \hookrightarrow W$ is a (weak) contact embedding that
does not separate~$W$, then by cutting~$W$ open
along~$M$ we obtain a (weak) symplectic cobordism between $(M,\xi)$ and
itself.  Attaching infinitely many copies of this cobordism to each other
in a sequence, one constructs a ``noncompact symplectic filling'' 
$(W_\infty,\omega_\infty)$ 
of $(M,\xi)$ which is nonetheless geometrically bounded.  If $(M,\xi)$
contains a partially planar domain for which $\omega_\infty$ is exact
on the spine, then one can attach
the cap from Theorem~\ref{thm:cap} and then choose a geometrically
bounded compatible almost complex structure~$J_\infty$ so that the 
symplectic spheres in the
cap become embedded $J_\infty$-holomorphic spheres which are Fredholm regular
and have index~$2$.  Arguing as in McDuff \cite{McDuff:rationalRuled},
the moduli space generated by these spheres is then compact and foliates
all of~$W_\infty$, but this is impossible since the latter is noncompact.
The full details for the case $[\Omega] = 0 \in H^2_\dR(M)$ are carried out
in \cite{AlbersBramhamWendl}, and the generalization for nontrivial
cohomology classes following the above scheme is immediate.

For planar torsion, we will make use of the following simple lemma in
the style of \cite{McDuff:rationalRuled}:
\begin{lemma}
\label{lemma:McDuff}
Suppose $(W,\omega)$ is a compact symplectic $4$-manifold, possibly with
boundary, such that $\p W$ carries a positive contact structure
dominated by~$\omega$.  Suppose moreover that~$W$ contains a symplectically
embedded sphere $S_1 \subset W$ with vanishing self-intersection number.
Then $\p W = \emptyset$, and any other symplectically embedded
surface $S_2 \subset W \setminus S_1$ with vanishing self-intersection 
is also a sphere and satisfies
$$
\int_{S_1} \omega = \int_{S_2} \omega.
$$
\end{lemma}
\begin{proof}
Choose a compatible almost complex structure~$J$ which preserves the
contact structure at the boundary and makes both $S_1$
and~$S_2$ $J$-holomorphic.  Then~$S_1$ is a Fredholm regular index~$2$
curve, and arguing as in \cite{McDuff:rationalRuled}, we find that 
the set of all $J$-holomorphic curves homotopic to~$S_1$
foliates~$W$ except at finitely many nodal singularities, which are
intersections of finitely many $J$-holomorphic exceptional spheres.  Then if
$\p W \ne \emptyset$, some holomorphic sphere must touch $\p W$ tangentially,
thus violating $J$-convexity.  Moreover, positivity of intersections
implies that no curve in this family can have any isolated intersection
with~$S_2$, thus $S_2$ itself must belong to the family, implying that it
is a sphere with the same symplectic area as~$S_1$.
\end{proof}

\begin{proof}[Proof of Theorem~\ref{thm:planarTorsion}] 
Consider again the spine removal cobordism $(X_0,\omega)$ from the proof of
Theorem~\ref{thm:cap} with $\p X_0 = - M' \amalg \widetilde{M}'$ and
$\widetilde{M}' = \widetilde{M}_1' \amalg \widetilde{M}_2'$, but now
under the extra assumption that the partially
planar domain $M \subset M'$ is not symmetric.  This implies in particular
that in addition to the planar piece $M\planar\paper$, the paper 
$M\paper \subset M$ contains another connected component 
$M\other\paper \subset M\paper$ for which at least one of the following is true:
\begin{enumerate}
\item $\p M\other\paper$ is not contained in $\Sigma_1 \times S^1 \cup \ldots \cup
\Sigma_r \times S^1$;
\item The pages in $M\other\paper$ have positive genus;
\item The pages in $M\other\paper$ have genus zero but there is a spinal component
$\Sigma_i \times S^1$ that contains differing numbers of boundary components
of pages in $M\planar\paper$ and $M\other\paper$.
\end{enumerate}
In the first case, it follows that $\widetilde{M}_2'$ carries a contact structure
dominated by~$\omega$, so after capping $\widetilde{M}_1'$ we have a contradiction
to Lemma~\ref{lemma:McDuff}.  In the second case, either the same thing
happens or $\widetilde{M}_2'$ is a symplectic fibration over~$S^1$ with closed
pages of positive genus, so capping both $\widetilde{M}_1'$ and $\widetilde{M}_2'$ with
Lefschetz fibrations as in \cite{Eliashberg:cap} gives disjoint
symplectically embedded surfaces with zero self-intersection, 
one rational and one not, again contradicting the lemma.  For the third
case we instead may obtain two disjoint symplectically embedded
spheres, but they can be arranged to have different symplectic area.
\end{proof}

\end{document}